\documentclass[11pt]{article}
\usepackage{amssymb,amsmath,amsfonts,amsthm,mathrsfs}
\usepackage{graphicx,graphics,psfrag,epsfig,calrsfs,cancel}
\usepackage[colorlinks=true, pdfstartview=FitV, linkcolor=blue, citecolor=red, urlcolor=blue]{hyperref}

\usepackage{caption,subfig,color}
\usepackage{tikz}
\usetikzlibrary{calc}

\usepackage{geometry}
\newtheorem{assumption}{Assumption}

\newtheorem{lemma}{Lemma}
\newtheorem{proposition}{Proposition}
\newtheorem{theorem}{Theorem}
\newtheorem{corollary}{Corollary}
\newtheorem{definition}{Definition}

\newcommand{\N}{{\mathbb N}}
\newcommand{\Z}{{\mathbb Z}}
\newcommand{\R}{{\mathbb R}}
\newcommand{\C}{{\mathbb C}}

\newcommand{\E}{{\mathbb E}}

\newcommand{\M}{{\mathbb M}}
\newcommand{\D}{{\mathbb D}}
\newcommand{\Dbar}{\overline{\mathbb D}}
\newcommand{\U}{{\mathcal U}}
\newcommand{\Ubar}{\overline{\mathcal U}}
\newcommand{\cercle}{{\mathbb S}^1}
\newcommand{\eps}{\varepsilon}
\newcommand{\dps}{\displaystyle}
\newcommand{\Ng}{| \! | \! |}
\newcommand{\Nd}{| \! | \! |}

\newcommand{\bxi}{{\boldsymbol \xi}}
\newcommand{\bpi}{{\boldsymbol \Pi}}
\newcommand{\bfeta}{{\boldsymbol \eta}}
\newcommand{\bkappa}{{\boldsymbol \kappa}}
\newcommand{\bdelta}{{\boldsymbol \Delta}}

\begin{document}

\title{Transparent numerical boundary conditions\\
for evolution equations:\\
Derivation and stability analysis}

\author{Jean-Fran\c{c}ois {\sc Coulombel}\thanks{CNRS and Universit\'e de Nantes, Laboratoire de Math\'ematiques 
Jean Leray (UMR CNRS 6629), 2 rue de la Houssini\`ere, BP 92208, 44322 Nantes Cedex 3, France. Email: 
{\tt jean-francois.coulombel@univ-nantes.fr}. Research of the author was supported by ANR project BoND, 
ANR-13-BS01-0009-01.}}
\date{\today}
\maketitle

\begin{abstract}
The aim of this article is to propose a systematic study of transparent boundary conditions for finite difference 
approximations of evolution equations. We try to keep the discussion at the highest level of generality in order 
to apply the theory to the broadest class of problems.

We deal with two main issues. We first derive transparent numerical boundary conditions, that is, we exhibit 
the relations satisfied by the solution to the pure Cauchy problem when the initial condition vanishes outside 
of some domain. Our derivation encompasses discretized transport, diffusion and dispersive equations with 
arbitrarily wide stencils. The second issue is to prove sharp {\it stability} estimates for the initial boundary 
value problem obtained by enforcing the boundary conditions derived in the first step. We focus here on 
discretized transport equations. Under the assumption that the numerical boundary is non-characteristic, 
our main result characterizes the class of numerical schemes for which the corresponding transparent boundary 
conditions satisfy the so-called Uniform Kreiss-Lopatinskii Condition introduced in \cite{gks}. Adapting some 
previous works to the non-local boundary conditions considered here, our analysis culminates in the derivation 
of trace and semigroup estimates for such transparent numerical boundary conditions. Several examples and 
possible extensions are given.
\end{abstract}
\bigskip


\bigskip
\bigskip

Throughout this article, we use the notation
\begin{align*}
&\U := \{\zeta \in \C,|\zeta|>1 \}\, ,\quad \D := \{\zeta \in \C,|\zeta|<1 \}\, ,\quad \cercle := \{\zeta \in \C,|\zeta|=1 \} \, ,\\
&\Ubar := \U \cup \cercle \, ,\quad \Dbar := \D \cup \cercle \, .
\end{align*}
We let ${\mathcal M}_{n_1,n_2} ({\mathbb K})$ denote the set of $n_1 \times n_2$ matrices with entries in 
${\mathbb K} = \R \text{ or } \C$. In the case $n_1=n_2=n$, we use the notation ${\mathcal M}_n ({\mathbb K})$ 
for the set of square matrices of size $n$. If $M \in {\mathcal M}_n (\C)$, $M^*$ denotes the conjugate transpose 
of $M$ and $\text{\rm sp}(M)$ denotes the spectrum of $M$. We let $I$ denote the identity matrix or the identity 
operator when it acts on an infinite dimensional space. The subscript in $I_k$ is intended to make the dimension 
$k$ of the underlying vector space $\C^k$ precise when needed. We use the notation $x^* \, y$ for the Hermitian 
product $\sum_i \overline{x_i} \, y_i$ of two vectors $x,y \in \C^n$. For two vectors $x,y \in \C^n$, the quantity 
$\sum_i x_i \, y_i$ is denoted $x \cdot y$; it coincides with the Euclidean product when the vectors have real 
coordinates. The norm of a vector $x \in \C^n$ is $|x| := (x^* \, x)^{1/2}$. The induced matrix norm on 
${\mathcal M}_n (\C)$ is denoted $\| \cdot \|$.

The letter $C$ denotes a constant that may vary from line to line or within the same line. The dependence of 
the constants on the various parameters is made precise throughout the text. If a constant $C$ depends on 
some parameter $\upsilon$, we write either $C_\upsilon$ or $C(\upsilon)$ to make this dependence explicit.

In what follows, we let $d \ge 1$ denote a fixed integer, which will stand for the dimension of the space domain 
$\R^d$ or $\Z^d$ we are considering. We shall use the space $\ell^2$ of square integrable sequences. Sequences 
may be valued in $\C^k$ for some integer $k$. In that case, we write $\ell^2 (\Z^d;\C^k)$ to emphasize that 
sequences are vector valued. Some sequences will be indexed by $\Z^{d-1}$ while some will be indexed by 
$\Z^d$ or a subset of $\Z^d$. We thus introduce some specific notation for the norms. Let $\Delta x_i>0$ for 
$i=1,\dots,d$ be $d$ space steps. We shall make use of the $\ell^2(\Z^{d-1})$ norm that we define as follows: 
for all $v \in \ell^2 (\Z^{d-1})$,
\begin{equation*}
\| v \|_{\ell^2(\Z^{d-1})}^2 := 
\left( \prod_{k=2}^d \Delta x_k \right) \, \sum_{i=2}^d \sum_{j_i\in \Z} |v_{(j_2,\dots,j_d)}|^2 \, .
\end{equation*}
The corresponding scalar product is denoted $\langle \cdot,\cdot \rangle_{\ell^2(\Z^{d-1})}$. Then for all integers 
$m_1 \le m_2$, we set
\begin{equation*}
\Ng u \Nd_{m_1,m_2}^2 := \Delta x_1 \, \sum_{j_1=m_1}^{m_2} \|u_{(j_1,\cdot)} \|_{\ell^2(\Z^{d-1})}^2 \, ,
\end{equation*}
to denote the $\ell^2$ norm on the set $[m_1,m_2] \times \Z^{d-1}$ ($m_1$ may equal $-\infty$ and $m_2$ 
may equal $+\infty$). The corresponding scalar product is denoted $\langle \cdot,\cdot \rangle_{m_1,m_2}$. 
In the particular case $d=1$, the space step is denoted $\Delta x$ and the $\ell^2$ norm on the interval 
$[m_1,m_2]$ reduces to
\begin{equation*}
\Ng u \Nd_{m_1,m_2}^2 := \Delta x \, \sum_{j=m_1}^{m_2} |u_j|^2 \, .
\end{equation*}
The $\ell^2(\Z^{d-1})$ norm reduces to the norm of vectors. Other notation is introduced when needed 
throughout the text or is meant to be self-explanatory.

\section{Introduction}
\label{intro}

\subsection{The context}

We are concerned here with the approximation of partial differential equations of the evolutionary type in 
the whole space $\R^d$. For simplicity, we restrict here to linear partial differential equations with constant 
coefficients of the form
\begin{equation}
\label{edp}
\partial_t v +P(\partial_1,\dots,\partial_d) \, v =0 \, ,
\end{equation}
where the differential operator $P$ is a polynomial expression of the spatial partial derivatives $\partial_1, \dots, 
\partial_d$ (we write $\partial_j$ for the partial derivative with respect to the $j$-th space variable $x_j$ and 
use $\partial_t$ for the partial derivative with respect to time). The differential operator $P$ may have complex 
coefficients so that the above framework encompasses the Schr\"odinger equation, as well as prototype evolution 
equations for real valued functions such as the transport, heat or Airy equation. We restrict here for simplicity 
to the case of {\it scalar} evolution equations of order $1$ in time: the unknown $v$ in \eqref{edp} is either real 
or complex valued and there is no second or higher order time derivative in \eqref{edp}. There would be no 
great effort to consider higher order equations such as the wave or beam equations but the functional framework 
would be slightly different.

When considered on the whole space domain $\R^d$, solving \eqref{edp} usually relies on the Fourier transform 
and we assume that well-posedness holds for \eqref{edp} in $L^2 (\R^d)$. More precisely, we assume
$$
\forall \, \bxi \in \R^d \, ,\quad \text{\rm Re } P(i\, \xi_1,\dots,i\, \xi_d) \ge 0 \, ,
$$
so the solution to the Cauchy problem \eqref{edp} satisfying $v|_{t=0}=v_0$ reads
$$
\forall \, t \ge 0 \, ,\quad v(t,x) =\dfrac{1}{(2\, \pi)^d} \, \int_{\R^d} {\rm e}^{i \, x \cdot \bxi} \, 
{\rm e}^{-t \, P(i\, \xi_1,\dots,i\, \xi_d)} \, \widehat{v_0}(\bxi) \, {\rm d}\bxi \, ,\quad 
\widehat{v_0}(\bxi) := \int_{\R^d} {\rm e}^{-i \, x \cdot \bxi} \, v(x) \, {\rm d}x \, .
$$
Of course, for hyperbolic or dispersive equations, $P(i\, \xi_1,\dots,i\, \xi_d)$ is a purely imaginary number 
and the solution $v$ to \eqref{edp} is also defined for $t \le 0$. However, the theory developed below for 
numerical schemes is mostly restricted to evolution equations in positive times because, even though the 
original partial differential equation \eqref{edp} may be time reversible, most of its finite difference approximations 
will not be so. Hence we consider $t \ge 0$ in what follows, which is of course no restriction when dealing with 
parabolic equations. The above Fourier representation of the solution implies by Plancherel's Theorem the 
uniform bound
$$
\forall \, t \ge 0 \, ,\quad \| v(t,\cdot) \|_{L^2(\R^d)} \le \| v_0 \|_{L^2(\R^d)} \, ,
$$
and we shall be interested below in numerical approximations of \eqref{edp} for which the same $L^2$ decay 
(or conservation) property holds, or a slightly relaxed version of it.
\bigskip

We now turn to the finite difference approximation of \eqref{edp}. We introduce a time step $\Delta t>0$ and 
some space steps $\Delta x_i>0$, $i=1,\dots,d$. The solution $v$ to \eqref{edp} is `approximated' by a 
piecewise constant function
\begin{equation}
\label{function}
u(t,x) := u_j^n \, ,\quad \forall \, (t,x) \in [n \, \Delta t,(n+1) \, \Delta t) \times 
\prod_{k=1}^d \, [j_k \, \Delta x_k,(j_k+1) \, \Delta x_k) \, ,
\end{equation}
where the sequence $(u^n_j)_{n \in \N, j\in \Z^d}$ is defined as the solution to the recurrence relation:
\begin{equation}
\label{cauchy}
\begin{cases}
{\dps \sum_{\sigma=0}^{s+1}} \, Q_\sigma \, u^{n+\sigma} =0 \, ,& n\ge 0 \, ,\\
(u^0,\dots,u^s) = (f^0,\dots,f^s) \in \ell^2 (\Z^d)^{s+1} \, . &
\end{cases}
\end{equation}
In \eqref{cauchy}, each (real or complex) sequence $u^n,\dots,u^{n+s+1}$ is defined on $\Z^d$, that is, we 
consider a pure Cauchy problem on the whole space, and the operators $Q_\sigma$ are given by:
\begin{equation}
\label{defop}
\forall \, \sigma =0,\dots,s+1 \, ,\quad 
Q_\sigma := \sum_{\ell=-r}^p a_{\ell,\sigma}(\Delta t,\Delta x) \, {\bf S}^\ell \, ,\quad 
\text{\rm with } ({\bf S}^\ell u)_j := u_{j+\ell} \, .
\end{equation}
Let us comment a little on \eqref{cauchy}, \eqref{defop}. First of all, the {\it stencil} of the scheme \eqref{cauchy} 
is assumed to be {\it fixed} and finite. More precisely, we consider some {\it fixed} numbers $r_1,\dots,r_d, 
p_1,\dots,p_d \in \N$ and we use in \eqref{defop} the short notation
$$
\sum_{\ell=-r}^p :=\sum_{k=1}^d \, \sum_{\ell_k=-r_k}^{p_k} \, .
$$
Then in \eqref{defop}, the coefficients $a_{\ell,\sigma}(\Delta t,\Delta x)$ are either real or complex numbers 
(depending on whether \eqref{edp} has real or complex coefficients/solutions) and they may depend on the 
time and space steps, which we abbreviate as $(\Delta t,\Delta x)$. Each operator $Q_\sigma$ in \eqref{cauchy} 
thus acts on sequences indexed by the (discrete) spatial variable $j \in \Z^d$, and is bounded on $\ell^2$ for 
any given choice of the discretization parameters $(\Delta t,\Delta x)$ under consideration. The (discrete) time 
variable $n$ enters as a parameter when we apply each operator $Q_\sigma$. For convenience, we write most 
of the time $Q_\sigma \, u_j^{n+\sigma}$ rather than $(Q_\sigma \, u^{n+\sigma})_j$ to denote the application 
of the operator $Q_\sigma$ to the sequence $u^{n+\sigma}$, the resulting sequence being evaluated at the 
space index $j$. For simplicity, we shall not write that the $Q_\sigma$'s depend on the time and space steps 
$(\Delta t,\Delta x)$ even though they will in all applications we have in mind.
\bigskip

Our aim is to cover several situations within the same framework. In particular, we wish to cover {\it explicit} 
schemes (in that case, $Q_{s+1}$ is the identity) that are commonly used for discretizing transport equations 
and which come unavoidably with CFL type restrictions \cite{strang1}, and {\it implicit} discretizations of diffusion 
and/or dispersive equations where the goal is precisely to avoid stringent stability constraints. We therefore 
consider from now on that the time and space steps belong to some set
\begin{equation*}
\bdelta \subset (0,+\infty)^{d+1} \, ,
\end{equation*}
where for instance we wish $\bdelta$ to be the semi-open square $(0,1] \times (0,1]$ when discretizing the 
one-dimensional heat or Schr\"odinger equation by an implicit scheme, and $\bdelta$ can be a semi-open 
interval
$$
\big\{ (\Delta t,\Delta t/\lambda_1,\dots,\Delta t/\lambda_d) \, ,\quad \Delta t \in (0,1] \big\} \, ,
$$
for some {\it fixed} parameters $\lambda_1,\dots,\lambda_d>0$ when we deal with an explicit scheme for a 
transport equation in $\R^d$. In all what follows, we assume that the coefficients $a_{\ell,\sigma}$ in \eqref{defop} 
are defined on $\bdelta$. For convergence purposes, it is tacitly assumed that $\bdelta$ contains at least one 
sequence that converges to zero. Let us keep in mind that some coefficients $a_{\ell,\sigma}$ in \eqref{defop} 
may be unbounded when the parameters $(\Delta t,\Delta x)$ approach the boundary of $\bdelta$. This will 
restrict some of our arguments below.
\bigskip

When implemented on a computer, the numerical scheme \eqref{cauchy} can of course not be used as such 
since it would rely on the storage of an infinite dimensional vector. We must therefore truncate the space domain, 
assuming for instance that the initial data are negligible outside of some domain of interest $\Omega \subset \Z^d$. 
However, as is evidenced by transport equation or by traveling waves in nonlinear equations, there is no reason why 
the solution to \eqref{cauchy} should remain negligible outside of the same fixed domain $\Omega \subset \Z^d$ at 
{\it any} later time $n \in \N$. This prevents in most situations from enforcing the homogeneous Dirichlet conditions 
on the boundary of $\Omega$ in order to compute the restriction to the domain of interest $\Omega$ of the solution 
to \eqref{cauchy}. In this article, we follow a long line of research devoted to the derivation and analysis of 
{\it transparent boundary conditions}, see among many other works 
\cite{aabes,AES,BELV,BMN,ducomet-zlotnik,ehrhardt-arnold,zheng-wen-han,zisowski-ehrhardt}. 
We shall be concerned here with {\it exact} transparent boundary conditions and refer for instance to 
\cite{aabes,abs,ehrhardt,hagstrom,halpern,han-yin,szeftel} for several works dedicated to the construction of 
approximate, more easily implementable, boundary conditions referred to as {\it absorbing}. In this work, we 
wish to understand first what are the relations satisfied by the solution to the pure Cauchy problem \eqref{cauchy} 
when the data $f^0,\dots,f^s$ vanish outside of some domain $\Omega \subset \Z^d$. Then among all these 
relations, we wish to select sufficiently many such that, when combined with the restriction of the recurrence 
relation \eqref{cauchy} to $\Omega$, we get a {\it stable} and hopefully {\it convergent} numerical initial boundary 
value problem. We now make our assumptions on the numerical scheme \eqref{cauchy} precise, and then state 
our main results. Let us already emphasize that this article is devoted to well-posedness issues for transparent 
boundary conditions only. We shall not deal here with consistency relations between the operators $Q_\sigma$ 
in \eqref{cauchy} and the original partial differential equation \eqref{edp}. Such consistency and convergence 
problems will be dealt with in a future work. Observe however that recurrence relations as in \eqref{cauchy} 
also arise in the discretization of higher order (in time) partial differential equations such as the wave equation 
so our framework \eqref{cauchy} is not restricted to the discretization of first order (in time) problems. The main 
technical restriction that we make here is that we focus on {\it scalar} problems: the coefficients $a_{\ell,\sigma}$ 
defining the operators $Q_\sigma$ in \eqref{cauchy} are real or complex numbers and the solutions to 
\eqref{cauchy} are also real or complex valued. The extension of our work to {\it systems} of equations is also 
postponed to a future work.

\subsection{Assumptions on the numerical scheme}

The recurrence relation \eqref{cauchy} is meant to be a defining equation for the sequence $u^{n+s+1} \in 
\ell^2(\Z^d)$, considering that the sequences $u^n,\dots,u^{n+s}$ are known and belong to $\ell^2(\Z^d)$ 
(and therefore \eqref{cauchy} is meant to define uniquely all the $u^n$'s, $n \ge s+1$, in terms of the initial 
data $f^0,\dots,f^s$). For future use, we not only assume that $Q_{s+1}$ is an isomorphism on $\ell^2 (\Z^d)$ 
but make the following slightly stronger assumption.

\begin{assumption}[Solvability of \eqref{cauchy}]
\label{assumption0}
For all discretization parameters $(\Delta t,\Delta x) \in \bdelta$, the operator $Q_{s+1}$ is an isomorphism 
on $\ell^2 (\Z^d)$, or equivalently\footnote{The equivalence between the properties that $Q_{s+1}$ is an 
isomorphism and that its {\it symbol} $\widehat{Q_{s+1}}$ does not vanish on $(\cercle)^d$ is based on the 
fact that we deal with a scalar problem. There are several other places where this restriction plays a role, 
though we shall not always point it out. For systems, one would need to consider the determinant of the 
matrix valued symbol $\widehat{Q_{s+1}}$ to characterize invertibility of $Q_{s+1}$.}:
$$
\forall \, \bkappa \in (\cercle)^d \, ,\quad 
\widehat{Q_{s+1}}(\bkappa) := \sum_{\ell=-r}^p a_{\ell,s+1}(\Delta t,\Delta x) \, \bkappa^\ell \neq 0 \, ,\quad 
\bkappa^\ell :=\kappa_1^{\ell_1} \cdots \kappa_d^{\ell_d} \, .
$$
Moreover, for all $\bfeta=(\eta_2,\dots,\eta_d) \in \R^{d-1}$, $Q_{s+1}$ satisfies the {\rm index condition}:
\begin{equation}
\label{index}
\dfrac{1}{2\, i \, \pi} \, \int_{\cercle} 
\dfrac{\partial_{\kappa_1} \widehat{Q_{s+1}}(\kappa_1,{\rm e}^{i\, \eta_2},\dots,{\rm e}^{i\, \eta_d})}
{\widehat{Q_{s+1}}(\kappa_1,{\rm e}^{i\, \eta_2},\dots,{\rm e}^{i\, \eta_d})} \, {\rm d}\kappa_1 =0 \, .
\end{equation}
\end{assumption}

The index condition \eqref{index} appears in many works devoted to the well-posedness of {\it implicit} 
discretizations of partial differential equations, see for instance \cite{strang2,osher3}. It originates in the 
theory of Toeplitz operators for which we refer to \cite{GF,Nikolski} or \cite[Chapter 27]{lax}. Let us observe 
that the index condition \eqref{index} is not necessary for solving \eqref{cauchy} on the whole space $\Z^d$. 
However, it will play a crucial role when we study well-posedness of the recurrence relation \eqref{cauchy} 
on a `half-space' $\N \times \Z^{d-1}$ or on a strip $[0;N] \times \Z^{d-1}$ in conjunction with the transparent 
boundary conditions derived below. At last, let us observe that Assumption \ref{assumption0} is trivially 
satisfied for {\it explicit} schemes, that is, when $Q_{s+1}$ is the identity.

We now make a crucial assumption on the {\it stability} of the numerical scheme \eqref{cauchy}.

\begin{assumption}[Stability of \eqref{cauchy}]
\label{assumption1}
For all $(\Delta t,\Delta x) \in \bdelta$, there exists a constant $C(\Delta t,\Delta x)>0$ such that for all 
initial data $f^0,\dots,f^s \in \ell^2(\Z^d)$, the solution to \eqref{cauchy} satisfies the uniform bound in 
time
\begin{equation}
\label{bornestabilitecauchy}
\sup_{n \in \N} \, \Ng u^n \Nd_{-\infty,+\infty}^2 \le C(\Delta t,\Delta x) \, \sum_{\sigma=0}^s 
\Ng f^\sigma \Nd_{-\infty,+\infty}^2 \, .
\end{equation}
\end{assumption}

By Fourier analysis, see \cite{RM,livreVB,gko}, it is a well-known fact that Assumption \ref{assumption1} 
can be completely characterized by the fulfillment of a uniform power boundedness property for the 
{\it amplification matrix} associated with \eqref{cauchy}. For future use, we therefore introduce the 
amplification matrix:
\begin{equation}
\label{defA}
{\mathcal A}(\bkappa) := \begin{pmatrix}
-\widehat{Q_s}(\bkappa)/\widehat{Q_{s+1}}(\bkappa) & \dots & \dots & 
-\widehat{Q_0}(\bkappa)/\widehat{Q_{s+1}}(\bkappa) \\
1 & 0 & \dots & 0 \\
 & \ddots & \ddots & \vdots \\
0 & & 1 & 0 \end{pmatrix} \in {\mathcal M}_{s+1}(\C)\, ,
\end{equation}
where the definition of the symbol $\widehat{Q_\sigma}$ for any $\sigma$ is identical to that of 
$\widehat{Q_{s+1}}$ in Assumption \ref{assumption0}, that is:
$$
\forall \, \sigma=0,\dots,s\, ,\quad 
\widehat{Q_\sigma}(\bkappa) := \sum_{\ell=-r}^p a_{\ell,\sigma} (\Delta t,\Delta x) \, \bkappa^\ell \, .
$$
In particular, we do not necessarily restrict the definition of $\widehat{Q_\sigma}$ to the set $(\cercle)^d$. 
The above definition of $\widehat{Q_\sigma}$ makes sense on $(\C \setminus \{ 0\})^d$. However, the 
amplification matrix ${\mathcal A}$ is defined after dividing by $\widehat{Q_{s+1}}$, and this is possible 
by Assumption \ref{assumption0} on a neighborhood of $(\cercle)^d$ in $\C^d$ (in what follows, we 
shall consider situations in which $\widehat{Q_{s+1}}$ may vanish in $\C \times (\cercle)^{d-1}$).

Let us clarify the link between Assumption \ref{assumption1} and the uniform power boundedness property 
for ${\mathcal A}$ since this will play a major role in the analysis below.

\begin{lemma}
\label{lem0}
Let the operators $Q_0,\dots,Q_{s+1}$ satisfy Assumptions \ref{assumption0} and \ref{assumption1}. Then 
the amplification matrix ${\mathcal A}$ in \eqref{defA} satisfies
\begin{equation}
\label{powerbounded}
\forall \, n \in \N \, ,\quad \forall \, \bkappa \in (\cercle)^d \, ,\quad \| {\mathcal A}(\bkappa)^n \|^2 
\le (s+1) \, C(\Delta t,\Delta x) \, ,
\end{equation}
where $C(\Delta t,\Delta x)$ is the same constant as in \eqref{bornestabilitecauchy}. In particular, for all 
$\bxi \in \R^d$, the dispersion relation
\begin{equation}
\label{dispersion}
\sum_{\sigma=0}^{s+1} \widehat{Q_\sigma} ({\rm e}^{i\, \xi_1},\dots,{\rm e}^{i\, \xi_d}) \, z^\sigma =0 \, ,
\end{equation}
has $s+1$ roots $z_1,\dots,z_{s+1}$ in $\Dbar$ and those roots located on $\cercle$ are simple.

For all $(\Delta t,\Delta x) \in \bdelta$, the operators $Q_\sigma$ are thus geometrically regular in the sense 
of \cite{jfcsinum}, see also \cite[Definition 3]{jfcnotes}. In other words the amplification matrix ${\mathcal A}$ 
satisfies the uniform power boundedness \eqref{powerbounded} and if $\underline{\bkappa} \in (\cercle)^d$ 
is such that there exists $\underline{z} \in \cercle \cap \text{\rm sp}({\cal A} (\underline{\bkappa}))$, then there 
exists a (unique) holomorphic function $\lambda$ on a neighborhood ${\mathcal W}$ of $\underline{\bkappa}$ 
in $\C^d$, such that
\begin{align*}
& \lambda(\underline{\bkappa}) =\underline{z} \, ,\\
& \det \big( z\, I +{\mathcal A}(\bkappa) \big)=\vartheta (\bkappa,z) \, \big( z+\lambda(\bkappa) \big) \, ,\quad 
\forall \, (z,\bkappa) \in \C \times {\mathcal W} \, ,
\end{align*}
with $\vartheta$ a holomorphic function of $(\bkappa,z)$ on ${\mathcal W} \times \C$ such that $\vartheta 
(\underline{\bkappa},\underline{z}) \neq 0$, and there exists a vector valued holomorphic function $E(\bkappa) 
\in \C^{s+1}$on ${\mathcal W}$ that satisfies
\begin{align*}
& E(\underline{\bkappa}) \neq 0 \, ,\\
&\forall \, \bkappa \in {\mathcal W} \, ,\quad {\mathcal A}(\bkappa) \, E(\bkappa) =\lambda(\bkappa) \, E(\bkappa) \, .
\end{align*}
\end{lemma}

\noindent The interested reader will observe that the situation encountered here is a particular case of the 
more general geometric regularity condition of \cite{jfcsinum} where several eigenvalues may `cross' at 
$\bkappa =\underline{\bkappa}$. Here the form of the companion matrix ${\mathcal A}$ makes such a 
crossing compatible with the uniform power boundedness \eqref{powerbounded} only if the crossing takes 
place inside the unit disk $\D$ (and not on the boundary $\cercle$).

\begin{proof}[Proof of Lemma \ref{lem0}]
The proof of the uniform bound \eqref{powerbounded} is given in \cite[Chapter 2]{jfcnotes} for one-dimensional 
explicit schemes ($d=1$, $Q_{s+1}=I$). The extension to multidimensional implicit schemes is rather 
straightforward so we omit it. The dispersion relation \eqref{dispersion} is the characteristic polynomial of 
${\mathcal A}$, which implies that the zeroes of \eqref{dispersion} are located in $\Dbar$. Furthermore, if 
one such root belongs to $\cercle$, then the multiplicity of $z$ as a root of \eqref{dispersion} must equal 
the dimension of the eigenspace of ${\mathcal A}(\bkappa)$ associated with $z$. Since eigenspaces of 
companion matrices have dimension $1$, $z$ is a simple root of \eqref{dispersion}. In particular, the 
corresponding eigenvalue of ${\mathcal A}$ depends locally holomorphically on $\bkappa$ and one can 
determine an eigenvector that also depends locally holomorphically on $\bkappa$. As noted in 
\cite[Lemma 7]{jfcnotes}, the geometric regularity of the operators $Q_\sigma$ follows automatically 
from the stability Assumption \ref{assumption1} in the scalar case.
\end{proof}

The uniform bound in time we require on \eqref{cauchy} is the (slightly relaxed) analogue of the $L^2$ decay 
property satisfied by the solutions to the partial differential equation \eqref{edp}. For numerical schemes with 
only one time level, that is when $s=0$, then \eqref{dispersion} is a first degree polynomial equation in $z$ 
whose only root is
$$
-\widehat{Q_0} ({\rm e}^{i\, \xi_1},\dots,{\rm e}^{i\, \xi_d})/\widehat{Q_1} ({\rm e}^{i\, \xi_1},\dots,{\rm e}^{i\, \xi_d}) 
\, ,
$$
and in that case, stability (in the sense of the fulfillment of \eqref{bornestabilitecauchy}) is equivalent to the 
fact that solutions to the recurrence formula
$$
Q_1 \, u^{n+1} +Q_0 \, u^n =0 \, ,
$$
satisfy the $\ell^2$ decay property
$$
\forall \, n \in \N \, ,\quad \Ng u^{n+1} \Nd_{-\infty,+\infty} \le \Ng u^n \Nd_{-\infty,+\infty} \, .
$$
In that case, the constant $C(\Delta t,\Delta x)$ does not even depend on $(\Delta t,\Delta x)$ and can be chosen 
to be $1$. However, when considering schemes with several time levels, there is no obvious generalization of 
this $\ell^2$ decay since the best one can hope for is to have an energy functional that is equivalent to the norm 
on $\ell^2(\Z^d)^{s+1}$ and that is nonincreasing for solutions to \eqref{cauchy}, see \cite{strikwerda-wade}. 
But such an energy will strongly depend on the numerical scheme under consideration, and we therefore state 
Assumption \ref{assumption1} in this rather simplified form which encompasses a wide class of difference 
schemes.

What is important here is that we require the iteration \eqref{cauchy} to satisfy a {\it uniform} bound in time, 
though we allow the constant $C(\Delta t,\Delta x)$ to depend `badly' on the time and space steps. For 
instance we allow $C(\Delta t,\Delta x)$ to be of the form $1/\Delta t$, which would be terrible for convergence 
purposes, but our framework excludes numerical schemes that only satisfy a `Lax-Richtmyer' bound of the form
$$
 \Ng u^{n+1} \Nd_{-\infty,+\infty} \le (1+C\, \Delta t) \, \Ng u^n \Nd_{-\infty,+\infty} \, ,
$$
which in the end gives rise to an exponential bound in time for the $\ell^2$ norm. Of course, the verification 
of Assumption \ref{assumption1} may restrict the set of discretization parameters $\bdelta$, and this is one 
reason for introducing such a set rather than considering the most general case $\bdelta =(0,1]^{d+1}$. 
Assumption \ref{assumption1} also rules out incorporating `lower order terms' in the numerical scheme 
\eqref{cauchy}. We restrict in some sense to the principal part of \eqref{cauchy} (just like the $L^2$ decay 
property for \eqref{edp} is not invariant under lower order perturbations).

Our last main assumption appears in several anterior works on fully discrete initial boundary value problems 
for hyperbolic equations, see among other works \cite{goldberg,goldberg-tadmor,kreiss1,gks,osher1}. It 
is also used in many {\it explicit} computations for deriving and analyzing discrete transparent boundary 
conditions for various prototype equations, see for instance 
\cite{ehrhardt-arnold,aabes,zisowski-ehrhardt,ehrhardt,ducomet-zlotnik,zheng-wen-han,BELV}. We 
state two possible versions of this assumption in order to highlight when the strong form of Assumption 
\ref{assumption2'} is necessary and when one can only use the weak form of Assumption \ref{assumption2}. 
In the defining equation \eqref{defA-d} below, we use the decomposition $j=(j_1,j') \in \Z \times \Z^{d-1}$ for 
any integer $j \in \Z^d$. In particular $r'$ stands for $(r_2,\dots,r_d)$ and so on.

\begin{assumption}[Noncharacteristic discrete boundary (weak form)]
\label{assumption2}
For $\ell_1=-r_1,\dots,p_1$, $z \in \C$ and $\bfeta \in \R^{d-1}$, let us define
\begin{equation}
\label{defA-d}
a_{\ell_1}(z,\bfeta) :=\sum_{\sigma=0}^{s+1} z^\sigma \, \sum_{\ell'=-r'}^{p'} 
a_{(\ell_1,\ell'),\sigma}(\Delta t,\Delta x) \, {\rm e}^{i \, \ell' \cdot \bfeta} \, .
\end{equation}
Then $a_{-r_1}$ and $a_{p_1}$ do not vanish on $\U \times \R^{d-1}$.
\end{assumption}

\begin{assumption}[Noncharacteristic discrete boundary (strong form)]
\label{assumption2'}
The functions $a_{-r_1}$ and $a_{p_1}$ defined in \eqref{defA-d} do not vanish on $\Ubar \times \R^{d-1}$.
\end{assumption}

In what follows, we shall always consider numerical schemes that satisfy Assumption \ref{assumption2}. 
This is sufficient for deriving the transparent boundary conditions for the scheme \eqref{cauchy}. However, 
in the stability analysis of transparent boundary conditions, it is convenient to allow $z \in \U$ to be 
arbitrarily close to the unit circle $\cercle$. This is the reason why for showing our main stability result 
(Theorem \ref{thm2} below), we shall make the stronger Assumption \ref{assumption2'}. The rather 
inconvenient feature of Assumption \ref{assumption2'} is that it excludes numerical schemes that are 
first based on a semi-discretization in space of \eqref{edp}, giving a system of differential equations 
of the form
$$
\dfrac{{\rm d}u_j}{{\rm d}t} =\sum_{\ell=-r}^p \widetilde{a}_\ell(\Delta x) \, u_{j+\ell} \, ,
$$
and then on the application of the Crank-Nicolson rule, giving rise to the numerical scheme
$$
\dfrac{1}{\Delta t} \, (u_j^{n+1}-u_j^n) =\dfrac{1}{2} \, \left( \sum_{\ell=-r}^p \widetilde{a}_\ell(\Delta x) 
\, u_{j+\ell}^{n+1} +\sum_{\ell=-r}^p \widetilde{a}_\ell(\Delta x) \, u_{j+\ell}^n \right) \, .
$$
In that case, Assumption \ref{assumption2'} is not satisfied because either $a_{-r_1}$ or $a_{p_1}$ (or 
even both if $r_1$ and $p_1$ are nonzero) vanishes at $z=-1$. However, Assumption \ref{assumption2'} 
is commonly satisfied in the discretization of hyperbolic equations by explicit methods and by some implicit 
schemes for parabolic equations too, see Section \ref{sect:examples} for some examples.

We make one last, mostly technical, assumption, which restricts a little the class of numerical schemes 
that we consider, but that we might be able to relax later on, to the price of some future refinements in the 
proofs below. Some remarks on Assumption \ref{assumption3} are made later on in the text when we use it.

\begin{assumption}[Technical restriction on the scheme]
\label{assumption3}
One of the following two conditions holds:
\begin{description}
\item[{\rm (i)}] $Q_{s+1}$ is the identity operator (the scheme \eqref{cauchy} is explicit), and for all 
$\bfeta \in \R^{d-1}$ there holds
\begin{equation*}
\begin{cases}
\sum_{\ell'=-r'}^{p'} a_{(-r_1,\ell'),s}(\Delta t,\Delta x) \, {\rm e}^{i \, \ell' \cdot \bfeta} \neq 0 \, ,& 
\text{\rm if } r_1>0 \, ,\\
\sum_{\ell'=-r'}^{p'} a_{(p_1,\ell'),s}(\Delta t,\Delta x) \, {\rm e}^{i \, \ell' \cdot \bfeta} \neq 0 \, ,& 
\text{\rm if } p_1>0 \, .
\end{cases}
\end{equation*}

\item[{\rm (ii)}] $Q_{s+1}$ is not the identity (the scheme \eqref{cauchy} is implicit), and for all 
$\bfeta \in \R^{d-1}$, there holds
\begin{equation*}
\sum_{\ell'=-r'}^{p'} a_{(-r_1,\ell'),s+1}(\Delta t,\Delta x) \, {\rm e}^{i \, \ell' \cdot \bfeta} \neq 0 \, ,\quad 
\sum_{\ell'=-r'}^{p'} a_{(p_1,\ell'),s+1}(\Delta t,\Delta x) \, {\rm e}^{i \, \ell' \cdot \bfeta} \neq 0 \, .
\end{equation*}
\end{description}
\end{assumption}

At this stage, the above assumptions incorporate many standard finite difference discretizations of 
hyperbolic, parabolic and dispersive partial differential equations. Several examples are discussed 
in Section \ref{sect:examples}. Our first goal in the following paragraph is to show that the derivation 
of transparent boundary conditions is indeed independent of the nature of the underlying partial 
differential equation, but only relies on the properties encoded in Assumptions \ref{assumption0}, 
\ref{assumption1} and \ref{assumption2} (we emphasize again that Assumption \ref{assumption3} 
aims mainly at simplifying one of the arguments below but should not be necessary in the most 
general possible framework). We shall then study the stability of the numerical scheme combining 
\eqref{cauchy} with transparent boundary conditions.

\subsection{Derivation of transparent numerical boundary conditions}

In this paper, we focus on the derivation and analysis of transparent numerical boundary conditions 
when the space domain $\Z^d$ is `truncated' only on one side. Of course, this is still far from a realistic 
implementation in a computer, but our goal is to highlight, in the case of one single boundary, the main 
features of \eqref{cauchy} that have an impact on the stability of the transparent numerical boundary 
conditions we construct below. The first step of the analysis is therefore to consider the numerical 
scheme \eqref{cauchy} with initial data that vanish on a `half-space', and to derive the relations satisfied 
by the solution to \eqref{cauchy} in that case. This part of the analysis only uses the weak form of the 
assumption that the discrete boundary is noncharacteristic (Assumption \ref{assumption2} rather than 
Assumption \ref{assumption2'}). Our first main result reads as follows.

\begin{theorem}
\label{thm1}
Let Assumptions \ref{assumption0}, \ref{assumption1}, \ref{assumption2} and \ref{assumption3} be 
satisfied. Then there exists a sequence $(\bpi_n)_{n \in \N}$ of bounded operators\footnote{In the case 
$d=1$, the $\bpi_n$'s are just square matrices of size $p_1+r_1$ since our definition of $\ell^2(\Z^{d-1}; 
\C^{p_1+r_1})$ then reduces to $\C^{p_1+r_1}$.} on $\ell^2(\Z^{d-1};\C^{p_1+r_1})$ that satisfies
\begin{align*}
&\forall \, \delta>0 \, ,\quad 
\sum_{n \in \N} \dfrac{1}{(1+\delta)^n} \, \| \bpi_n \|_{{\mathcal B}(\ell^2(\Z^{d-1}))} <+\infty \, ,\quad 
\text{\rm (growth condition)} \, ,\\
&\forall \, n \in \N \, ,\quad \bpi_n =\sum_{m=0}^n \bpi_m \, \bpi_{n-m} \, ,\quad \text{\rm (algebraic constraints)} \, ,
\end{align*}
and such that for all initial data $f^0,\dots,f^s \in \ell^2(\Z^d)$ in \eqref{cauchy} verifying
$$
\forall \, \sigma=0,\dots,s\, ,\quad \forall \, j_1 \le p_1 \, ,\quad f_{(j_1,\cdot)}^\sigma =0 \, ,
$$
then the solution to \eqref{cauchy} satisfies
\begin{equation}
\label{relations}
\forall \, n \in \N \, ,\quad \forall \, j_1 \le 0 \, ,\quad 
\sum_{m=0}^n \bpi_{n-m} \, \begin{pmatrix}
u_{(j_1+p_1,\cdot)}^m \\
\vdots \\
u_{(j_1+1-r_1,\cdot)}^m \end{pmatrix} =0 \, .
\end{equation}

When the scheme is explicit (case {\rm (i)} in Assumption \ref{assumption3}), the operator $\bpi_0$ is given 
by
$$
\bpi_0 =\begin{pmatrix}
0 & 0 \\
0 & I_{r_1} \end{pmatrix} \, .
$$
\end{theorem}

Let us observe that $\bpi_0$ is a projector because of the algebraic constraints written for $n=0$. When the 
scheme is implicit (case {\rm (ii)} in Assumption \ref{assumption3}), the operator $\bpi_0$ does not have such 
a nice expression as in the explicit case. In particular, $\bpi_0$ is most of the time a genuine {\it nonlocal} 
operator in the tangential variable $j' \in \Z^{d-1}$, meaning that the value of the sequence $\bpi_0 \, v$ at 
an index $j' \in \Z^{d-1}$ does not depend on finitely many values of $v$ close to $j'$ but on the whole sequence 
$v$. (Of course, in the one-dimensional case $d=1$, the $\bpi_n$'s are matrices of size $p_1+r_1$ and the 
discussion on non-locality becomes irrelevant.) It should be noted that in the explicit case, the $\bpi_n$'s for 
$n \ge 1$ are also nonlocal (unless some specific - though unlikely - cancellation appears). Such nonlocality 
is a very common feature of transparent boundary conditions for partial differential equations in several space 
dimensions, see e.g. \cite{ab,szeftel1} for the case of the Schr\"odinger equation.

For $n=0,\dots,s$, the relations in \eqref{relations} are trivially satisfied because the initial data $f^0,\dots,f^s$ 
vanish for $j_1 \le p_1$. Hence the relations in \eqref{relations} start to be really meaningful for $n \ge s+1$ but 
for reasons that will arise later on, it is useful to consider the whole set of relations \eqref{relations} indexed by 
$n \in \N$ and not only by $n \ge s+1$.

Theorem \ref{thm1} provides with the set of relations \eqref{relations} that are satisfied by the solution to 
\eqref{cauchy} when the initial data have support in $\{ j_1 >p_1 \}$. In what follows, we are concerned 
with the numerical scheme that is obtained by combining the iteration \eqref{cauchy} on a half-space, that 
is, we truncate the space domain in one spatial direction, in conjunction with discrete transparent boundary 
conditions obtained from \eqref{relations}. More precisely, we are going to consider the following numerical 
scheme:
\begin{equation}
\label{transparent0}
\begin{cases}
{\dps \sum_{\sigma=0}^{s+1}} \, Q_\sigma \, u_j^{n+\sigma} =\Delta t \, F_j^{n+s+1} \, ,& 
n\ge 0 \, ,\quad j_1 \ge 1 \, ,\\
{\dps \sum_{m=0}^{n+s+1}} \, \bpi_{n+s+1-m} \, \begin{pmatrix}
u_{(p_1,\cdot)}^m \\
\vdots \\
u_{(1-r_1,\cdot)}^m \end{pmatrix} = g^{n+s+1} \, , & n\ge 0 \, ,\\
(u_j^0,\dots,u_j^s) = (f_j^0,\dots,f_j^s) \, , & j_1 \ge 1-r_1 \, ,
\end{cases}
\end{equation}
where the $\bpi_n$'s are the tangential operators given by Theorem \ref{thm1} and whose precise definition 
is given (on the Fourier side) by \eqref{defpin}.

The discrete initial boundary value problem \eqref{transparent0} is meant to describe, for zero interior 
and boundary source terms $F$ and $g$, the dynamics of \eqref{cauchy} when restricted to the half-space 
$\{ j_1 \ge 1-r_1 \}$. There is however a little discrepancy because in Theorem \ref{thm1} we have considered 
initial data that vanish for $j_1 \le p_1$, while, due to the stencil of the operators $Q_\sigma$, we consider 
in \eqref{transparent0} a solution $(u_j^n)$ that is indexed by $j_1 \ge 1-r_1$. In particular, the initial data 
$f^0,\dots,f^s$ in \eqref{transparent0} need not vanish for $j_1 \le p_1$. We also consider the possibility of 
nonzero interior and boundary forcing terms $F$ and $g$ in view of proving stability estimates that will later 
be useful for convergence purposes.

The first obvious question when considering \eqref{transparent0} is to determine whether there exists a 
unique solution $u$ for given source terms $F,g,f$ (respectively interior, boundary forcing terms, and initial 
data). It turns out that existence and uniqueness of a solution to \eqref{transparent0} is {\it automatic} in the 
framework of Theorem \ref{thm1}. There is however a little price to pay because of the formulation of the 
numerical boundary conditions. Since $\bpi_0$ is a projector that is not the identity (unless $p_1=0$), it 
cannot be an isomorphism on $\ell^2(\Z^{d-1})$ and there must necessarily be algebraic constraints on 
the source terms $(g^n)_{n \ge s+1}$ in \eqref{transparent0} for proving existence of a solution to 
\eqref{transparent0}. These constraints are made clear in the following result\footnote{Lemma \ref{lem1} 
is purely algebraic and does not require the vector space $E$ to be a Banach or Hilbert space, nor the 
linear operators to be bounded.}.

\begin{lemma}
\label{lem1}
Let $E$ be a vector space and let $(P_n)_{n \in \N}$ be a sequence of linear operators on $E$ such that
$$
\forall \, n \in \N \, ,\quad P_n =\sum_{m=0}^n P_m \, P_{n-m} \, .
$$
Let $(y_n)_{n \in \N}$ be a sequence with values in $E$. Then there exists a sequence $(x_n)_{n \in \N}$ 
with values in $E$ that satisfies
\begin{equation}
\label{convolution}
\forall \, n \in \N \, ,\quad \sum_{m=0}^n P_{n-m} \, x_m =y_n \, ,
\end{equation}
if and only if there holds
\begin{equation}
\label{compatibilitebord}
\forall \, n \in \N \, ,\quad y_n =\sum_{m=0}^n P_{n-m} \, y_m \, .
\end{equation}
\end{lemma}

\noindent Enforcing algebraic constraints as in \eqref{compatibilitebord} for the source terms in 
\eqref{transparent0}, the unique solvability of \eqref{transparent0} can be stated as follows.

\begin{proposition}
\label{prop1}
Let Assumptions \ref{assumption0}, \ref{assumption1}, \ref{assumption2} and \ref{assumption3} be 
satisfied. Let $f^0,\dots,f^s \in \ell^2$ be the initial data for \eqref{transparent0}, and let $(g^n)_{n \ge s+1}$ 
be a sequence in $\ell^2(\Z^{d-1};\C^{p_1+r_1})$ of boundary source terms for \eqref{transparent0}. With 
the tangential operators $(\bpi_n)_{n \ge 0}$ given in Theorem \ref{thm1}, let us define\footnote{If the initial 
data $f^0,\dots,f^s$ for \eqref{transparent0} vanish for $j_1 \le p_1$, then $g^0, \dots, g^s$ vanish.} for 
$n=0,\dots,s$, :
$$
g^n := \sum_{m=0}^n \, \bpi_{n-m} \, \begin{pmatrix}
f_{(p_1,\cdot)}^m \\
\vdots \\
f_{(1-r_1,\cdot)}^m \end{pmatrix} \, ,
$$
and let us further assume that the following compatibility conditions are satisfied:
\begin{equation}
\label{compatibilitegn}
\forall \, n \ge 0 \, ,\quad g^n =\sum_{m=0}^n \bpi_{n-m} \, g^m \, .
\end{equation}
Then for all sequence of interior source terms $(F^n)_{n \ge s+1}$ with values in $\ell^2$, there exists a 
unique sequence $(u^n)_{n \in \N}$ with values in $\ell^2$ solution to \eqref{transparent0}.
\end{proposition}

One drawback of Theorem \ref{thm1} is that the family of operators $(\bpi_n)_{n \ge 0}$ is not uniquely 
defined, even when enforcing the growth condition and the algebraic constraints. There is however one 
choice that seems to be more natural than others, and this is the one we make in the defining equation 
\eqref{defpin}. Other formulations of the transparent boundary conditions are proposed in Section 
\ref{sect:2}, some of which being analogous to the one encoded in \eqref{transparent0}, and others 
being closer to those used in 
\cite{AES,aabes,zisowski-ehrhardt,ehrhardt-arnold,ehrhardt,ducomet-zlotnik,zheng-wen-han,BELV} etc. 
We discuss how one can pass from one formulation to the other depending on the space dimension $d$.

Once we know that \eqref{transparent0} has a unique solution, there remains to determine whether 
this solution depends continuously on the data. This is a {\it stability} problem, and the last requirement 
for Hadamard well-posedness of \eqref{transparent0}. Stability is to be understood as proving that a 
certain norm of the solution $u$ to \eqref{transparent0} can be estimated in terms of some appropriate 
norms of the source terms in \eqref{transparent0}. This is where the nature of the underlying partial 
differential equation \eqref{edp} comes back into play since what we have in mind is proving an estimate 
for \eqref{transparent0} that is compatible with the `continuous' limit $\Delta t, \Delta x \rightarrow 0$. 
However, since the scale invariance properties of the transport, heat, Schr\"odinger or Airy equations 
are widely different, it seems hopeless at this point to encompass all possible applications within the 
same framework. In what follows we explore one possible notion of stability that is related to the theory 
of {\it hyperbolic} initial boundary value problems. This means that the underlying partial differential 
equation \eqref{edp} we have in mind is a transport equation (with ${\bf a} \in \R^d$ a fixed vector):
$$
\partial_t v +\sum_{j=1}^d {\bf a}_j \, \partial_j v=0 \, .
$$
Dispersive equations such as the Schr\"odinger or Airy equations will be addressed in a near future with 
stability estimates compatible with the ones discussed in \cite{audiard} for the continuous problem.

\subsection{Characterization of strong stability}

The notion of stability that we discuss is inspired from \cite{gks} and is rather restrictive in the sense 
that it requires controlling the trace of the solution to \eqref{transparent0} with possibly non-homogeneous 
boundary forcing terms $g$. From now on, we consider that the ratios $\Delta t/\Delta x_i$, $i=1,\dots,d$, 
are constant, which means that the set $\bdelta$ of discretization parameters is a semi-open interval
$$
\bdelta =\big\{ (\Delta t,\Delta t/\lambda_1,\dots,\Delta t/\lambda_d) \, ,\quad \Delta t \in (0,1] \big\} \, ,
$$
where $\lambda_1,\dots,\lambda_d$ are fixed positive numbers. We also assume that the coefficients 
$a_{\ell,\sigma}$ in \eqref{defop} only depend on the time and space steps through the ratios $\lambda_1, 
\dots, \lambda_d$. In other words, keeping $\Delta t$ as the only free small parameter, we assume 
that the operators $Q_\sigma$ are {\it independent} of $\Delta t$, which means that the scheme 
\eqref{transparent0} is also independent of $\Delta t$ (since one can rename the interior source term 
$\Delta t \, F_j^{n+s+1}$ fictitiously as $\widetilde{F}_j^{n+s+1}$). The scaling invariance $\Delta t/\Delta x_i 
=\text{\rm cst}$ is reminiscent of the scaling invariance $(t,x) \rightarrow (\alpha \, t,\alpha \, x)$ of the 
underlying hyperbolic equation we have in mind. The estimate \eqref{stabilitenumibvp} below is also 
the direct analogue of the weighted in time estimates discussed in \cite[Chapter 4]{benzoni-serre} and 
that are also invariant by the scaling $(t,x) \rightarrow (\alpha \, t,\alpha \, x)$ (the Laplace parameter 
$\gamma$ below being then rescaled as $\gamma \rightarrow \gamma/\alpha$). As already mentioned, 
we plan to adapt the continuous {\it dispersive} estimates of \cite{audiard} to the framework of finite 
difference schemes and transparent boundary conditions in a near future. We now introduce the 
following terminology.

\begin{definition}[Strong stability \cite{gks}]
\label{defstab1}
The finite difference approximation \eqref{transparent0} is said to be \emph{strongly stable} if there exists 
a constant $C$ such that for all $\gamma>0$ and all $\Delta t \in \, (0,1]$, the solution\footnote{Here we 
tacitly assume that the boundary forcing terms and vanishing initial data satisfy the compatibility conditions 
\eqref{compatibilitegn} so that a solution to \eqref{transparent0} does exist.} $(u_j^n)$ to \eqref{transparent0} 
with $(f_j^0) =\dots =(f_j^s) =0$ satisfies the estimate:
\begin{multline}
\label{stabilitenumibvp}
\dfrac{\gamma}{\gamma \, \Delta t+1} \, \sum_{n\ge s+1} \Delta t \, {\rm e}^{-2\, \gamma \, n\, \Delta t} \, 
\Ng u^n \Nd_{1-r_1,+\infty}^2 
+\sum_{n\ge s+1} \, \sum_{j_1=1-r_1}^{p_1} \Delta t \, {\rm e}^{-2\, \gamma \, n\, \Delta t} \, 
\| u_{(j_1,\cdot)}^n \|_{\ell^2(\Z^{d-1})}^2 \\
\le C_1 \left\{ \dfrac{\gamma \, \Delta t+1}{\gamma} \, \sum_{n \ge s+1} \Delta t \, 
{\rm e}^{-2\, \gamma \, n\, \Delta t} \, \Ng F^n \Nd_{1,+\infty}^2 
+\sum_{n\ge s+1} \, \Delta t \, {\rm e}^{-2\, \gamma \, n\, \Delta t} \, \| g^n \|_{\ell^2(\Z^{d-1})}^2 \right\} \, .
\end{multline}
\end{definition}

The main point in the estimate \eqref{stabilitenumibvp} is that the $\ell^2_{n,j'}$ norm of the trace of 
the solution is estimated with the same `weight' as $\gamma$ times the $\ell^2_{n,j}$ norm. Moreover, 
there is no loss of `derivative' from the source terms in \eqref{transparent0} to the solution in the estimate 
\eqref{stabilitenumibvp}. We exhibit below a necessary and sufficient condition for \eqref{transparent0} 
to be strongly stable. The analysis is inspired from and bears quite some resemblance with \cite{jfc}, 
see also \cite{trefethen3}. Before stating our main result, we introduce some more terminology.

\begin{definition}[Non-glancing scheme]
\label{def2}
Let the operators $Q_0,\dots,Q_{s+1}$ satisfy Assumptions \ref{assumption0} and \ref{assumption1}. The 
scheme \eqref{cauchy} is said to be \emph{non-glancing} if for all $\underline{\bkappa} \in (\cercle)^d$ such 
that there exists $\underline{z} \in \cercle \cap \text{\rm sp}({\cal A} (\underline{\bkappa}))$, then the 
holomorphic eigenvalue $\lambda$ of ${\mathcal A}$ given in Lemma \ref{lem1} satisfies
$$
\dfrac{\partial \lambda}{\partial \kappa_1} (\underline{\bkappa}) \neq 0 \, .
$$
\end{definition}

As discussed in \cite{jfc}, several standard discretizations of transport equations such as the upwind, 
Lax-Friedrichs and Lax-Wendroff schemes are non-glancing. At the opposite, and as already noticed 
in \cite{trefethen3}, the leap-frog and Crank-Nicolson approximations of the transport equation admit 
glancing wave packets (while the underlying partial differential equation does not !). Such examples 
are discussed in Section \ref{sect:examples}. Our main result reads as follows. We emphasize that 
from now on we enforce the stronger Assumption \ref{assumption2'} rather than its weak version 
(Assumption \ref{assumption2}).

\begin{theorem}
\label{thm2}
Let Assumptions \ref{assumption0}, \ref{assumption1}, \ref{assumption2'} and \ref{assumption3} be 
satisfied. Then the scheme \eqref{transparent0} is strongly stable in the sense of Definition \ref{defstab1} 
if and only if the scheme \eqref{cauchy} for the pure Cauchy problem is non-glancing.

Assume furthermore that for all $\bxi \in \R^d$, the roots to the dispersion relation \eqref{dispersion} are 
simple. Then if the scheme \eqref{cauchy} is non-glancing, there exists a constant $C>0$ such that for 
all $\gamma>0$ and all $\Delta t \in \, (0,1]$, the solution\footnote{Here again we tacitly assume that 
the boundary forcing terms and the (nonzero) initial data satisfy the compatibility conditions 
\eqref{compatibilitegn} so that a solution to \eqref{transparent0} does exist.} $(u_j^n)$ to 
\eqref{transparent0} with $f^0,\dots,f^s \in \ell^2$ satisfies the estimate:
\begin{multline}
\label{estimation}
\sup_{n \in \N} {\rm e}^{-2\, \gamma \, n\, \Delta t} \, \Ng u^n \Nd_{1-r_1,+\infty}^2 
+\sum_{n\ge s+1} \, \sum_{j_1=1-r_1}^{p_1} \Delta t \, {\rm e}^{-2\, \gamma \, n\, \Delta t} \, 
\| u_{(j_1,\cdot)}^n \|_{\ell^2(\Z^{d-1})}^2 \\
\le C \left\{ \sum_{\sigma=0}^s \Ng f^\sigma \Nd_{1-r_1,+\infty}^2 
+\dfrac{\gamma \, \Delta t+1}{\gamma} \, \sum_{n \ge s+1} \Delta t \, 
{\rm e}^{-2\, \gamma \, n\, \Delta t} \, \Ng F^n \Nd_{1,+\infty}^2 
+\sum_{n\ge s+1} \, \Delta t \, {\rm e}^{-2\, \gamma \, n\, \Delta t} \, \| g^n \|_{\ell^2(\Z^{d-1})}^2 \right\} \, .
\end{multline}
\end{theorem}

Let us observe that in \eqref{estimation}, we could have added for free the quantity
$$
\dfrac{\gamma}{\gamma \, \Delta t+1} \, \sum_{n\ge 0} \Delta t \, {\rm e}^{-2\, \gamma \, n\, \Delta t} \, 
\Ng u^n \Nd_{1-r_1,+\infty}^2 \, ,
$$
on the left hand side, since it is controlled by the stronger `semigroup' norm
$$
\sup_{n \in \N} {\rm e}^{-2\, \gamma \, n\, \Delta t} \, \Ng u^n \Nd_{1-r_1,+\infty}^2 \, ,
$$
uniformly in $\gamma$ and $\Delta t$. The main point in Theorem \ref{thm2} is that for non-glancing 
schemes, one can control both semigroup and trace norms of the solution to \eqref{transparent0}, 
including for the case of non-homogeneous boundary conditions. Such strong stability estimates 
are relevant for two purposes: first they allow for a convergence analysis relying on suitable 
consistency estimates, and second such strong stability estimates persist under any small 
modification of the numerical boundary conditions. This means that one may be able to prove 
first stability and then convergence for a class of {\it absorbing} boundary conditions that is based 
on a sufficiently good approximation of the tangential operators $(\bpi_n)_{n\ge 0}$. We plan to 
investigate this issue in a future work, for instance when the operators $\bpi_n$ are approximated 
by the so-called sum of exponentials, see e.g. \cite{AES}.
\bigskip

The plan of the paper is as follows. In Section \ref{sect:2}, we derive the transparent numerical 
boundary conditions and prove Theorem \ref{thm1}. We also discuss alternative formulations of the 
transparent numerical boundary conditions and make some constructions more explicit in the special 
case $p_1=r_1=1$ which occurs in many practical examples. In Section \ref{sect:3}, we show solvability 
for the numerical scheme \eqref{transparent0} and prove Lemma \ref{lem1} and Proposition \ref{prop1}. 
We also briefly describe, for $d=1$, the case where the space domain $\Z$ is truncated on either side, 
leading to a {\it finite dimensional} problem on an interval $[0,N]$. Section \ref{sect:4} is devoted to the 
proof of Theorem \ref{thm2} which characterizes strong stability in terms of non-existence of glancing 
wave packets. Eventually we discuss in Section \ref{sect:examples} several examples related to 
transport, diffusion or dispersive equations.

\section{Derivation of transparent numerical boundary conditions}
\label{sect:2}

In this Section, we prove Theorem \ref{thm1} and construct transparent boundary conditions in the 
rather wide framework covered by Assumptions \ref{assumption0}, \ref{assumption1}, \ref{assumption2} 
and \ref{assumption3}.

\subsection{Proof of Theorem \ref{thm1}}

The proof splits in several steps. In the first step, we compute the transparent conditions on the `Laplace-Fourier 
side'. In several references, this calculation is performed by using the so-called ${\mathcal Z}$-transform, 
which is the discrete analogue of the Laplace transform. The second step consists in going back to the 
original time and space variables by using the inverse Laplace-Fourier transform. This requires specific 
attention in order to show that the `causality' principle is preserved, meaning that in \eqref{relations}, the 
relation that should eventually contribute to determining $u^n$ does not involve some $u^{n'}$ with $n'>n$. 
Writing the relations \eqref{relations} in a causal way in which `future does not affect the past' amounts to 
proving that the Laurent expansions of several objects that are holomorphic on $\U$ do have a limit at infinity 
(and therefore the Laurent expansions only involve nonpositive powers of $z \in \U$). Justifying that such limits 
at infinity exist is the purpose of the second step of the proof and this is the reason why we make the technical 
Assumption \ref{assumption3}. Once this is achieved, the conclusions of Theorem \ref{thm1} follow rather easily.
\bigskip

$\bullet$ Step 1. The transparent conditions on the Laplace-Fourier side.

We consider the iteration \eqref{cauchy} with initial data in $\ell^2$ vanishing for $j_1 \le p_1$. Thanks to 
Assumption \ref{assumption1}, we know that the solution to \eqref{cauchy} satisfies
\begin{equation}
\label{bornecauchy}
\forall \, \gamma >0 \, ,\quad \sum_{n \ge 0} \Delta t \, {\rm e}^{-2 \, \gamma \, n \, \Delta t} \, 
\Ng u^n \Nd_{-\infty,+\infty}^2 <+\infty \, .
\end{equation}
In particular, for all fixed $x_1 \in \R$, the piecewise constant function $u(\cdot,x_1,\cdot)$ in \eqref{function} 
has a well-defined Laplace-Fourier transform on $\C^+ \times \R^{d-1}$, where we let $\C^+$ denote the set 
of complex numbers with positive real part. Here the Laplace transform refers to the time variable $t$, and 
the (partial) Fourier transform refers to the tangential space variables $x':=(x_2,\dots,x_d)$. Even though the 
$x_1$ variable lies in $\R$, we do not perform Fourier transform with respect to $x_1$. We rather stick to the 
original discrete variable $j_1 \in \Z$ and feel free to use the notation $u_{j_1}(t,x')$ which is entirely analogous 
to the notation introduced in \eqref{function} for the step function $u$. The dual variables to $(t,x')$ are denoted 
$(\tau,\bxi')$ below, with $\tau=\gamma+i\, \theta$ and $\bxi'=(\xi_2,\dots,\xi_d)$, and the Laplace-Fourier transform 
of $u_{j_1}$ is denoted $\widehat{u_{j_1}}$. Moreover, given $(\tau,\bxi') \in \C^+ \times \R^{d-1}$, we always 
use the notation
$$
z:={\rm e}^{\tau \, \Delta t} \in \U \, ,\quad \bfeta :=(\xi_2 \, \Delta x_2,\dots,\xi_d \, \Delta x_d) \in \R^{d-1} \, .
$$
The difference between the Laplace-Fourier transform used here and the ${\mathcal Z}$-transform used in 
\cite{AES,aabes,zisowski-ehrhardt,ehrhardt,BELV} is only of a multiplicative - though {\it crucial} - factor. 
Indeed, we compute
\begin{equation}
\label{laplaceu}
\widehat{u_{j_1}}(\tau,\bxi') =\dfrac{1-z^{-1}}{\tau} \, \sum_{n \ge 0} z^{-n} \, 
\int_{\R^{d-1}} {\rm e}^{-i\, x' \cdot \bxi'} \, u_{j_1}^n(x') \, {\rm d}x' \, ,
\end{equation}
at least if, for all $n \in \N$, the sequence $(u_{(j_1,\cdot)}^n)$ belongs to $\ell^1(\Z^{d-1})$. Otherwise we use 
the continuous extension of the Fourier transform to $L^2(\R^{d-1})$ in case $(u_{(j_1,\cdot)}^n)$ belongs to 
$\ell^2(\Z^{d-1})$ without belonging to $\ell^1(\Z^{d-1})$.

Applying Plancherel's Theorem, we obtain from the bound \eqref{bornecauchy} the property:
$$
\forall \, \gamma >0 \, ,\quad \sum_{j_1 \in \Z} \int_{\R \times \R^{d-1}} \big| 
\widehat{u_{j_1}}(\gamma+i\, \theta,\bxi') \big|^2 \, {\rm d} \theta \, {\rm d}\bxi' <+\infty \, .
$$
In particular, for any $\gamma>0$, the sequence $(\widehat{u_{j_1}}(\gamma+i\, \theta,\bxi'))_{j_1 \in \Z}$ 
belongs to $\ell^2(\Z)$ for almost every $(\theta,\bxi') \in \R \times \R^{d-1}$. We omit below the possible 
negligible set of those $(\theta,\bxi')$ for which the sequence is not in $\ell^2$, and make as if this negligible 
set is always empty. This has no consequence of course in the proof.

With the above notation, it is rather straightforward to derive the recurrence relation satisfied by the sequence 
$(\widehat{u_{j_1}}(\gamma+i\, \theta,\bxi'))_{j_1 \in \Z}$. Namely, we first apply the partial Fourier transform 
with respect to $x'$ and get\footnote{It is convenient here to use also the hat notation for denoting the partial 
Fourier transform with respect to the tangential spatial variables $x'$.}
$$
\begin{cases}
{\dps \sum_{\sigma=0}^{s+1}} \, Q_\sigma^\sharp(\bfeta) \, \widehat{u_{j_1}^{n+\sigma}} (\bxi') =0 \, ,& 
n\ge 0 \, ,\quad j_1 \in \Z \, ,\\
\, & \\
\Big( \widehat{u_{j_1}^0},\dots,\widehat{u_{j_1}^s} \Big) (\bxi') = \Big( \widehat{f_{j_1}^0},\dots,\widehat{f_{j_1}^s} 
\Big) (\bxi') \, , & j_1 \in \Z \, ,
\end{cases}
$$
with
\begin{equation}
\label{defqdiese}
\forall \, \sigma=0,\dots,s \, ,\quad 
Q_\sigma^\sharp(\bfeta) := \sum_{\ell_1 =-r_1}^{p_1} \left( 
\sum_{\ell'=-r'}^{p'} a_{(\ell_1,\ell'),\sigma}(\Delta t,\Delta x) \, {\rm e}^{i \, \ell' \cdot \bfeta} \right) \, 
{\bf S}_1^{\ell_1} \, ,\quad ({\bf S}_1^{\ell_1} w)_{j_1} :=w_{j_1+\ell_1} \, .
\end{equation}
Let us recall that $\bfeta$ is a placeholder for $(\xi_2 \, \Delta x_2,\dots,\xi_d \, \Delta x_d)$. We then use the 
expression \eqref{laplaceu} of the Laplace-Fourier transform $\widehat{u_{j_1}}$ to compute the relation
\begin{equation}
\label{recurrencecauchy}
\forall \, j_1 \in \Z \, ,\quad 
\sum_{\ell_1=-r_1}^{p_1} a_{\ell_1}(z,\bfeta) \, \widehat{u_{j_1+\ell_1}}(\tau,\bxi') =F_{j_1}(\tau,\bxi') \, ,
\end{equation}
where the functions $a_{-r_1},\dots,a_{p_1}$ are defined in \eqref{defA-d} and the source term $F_{j_1}$ in 
\eqref{recurrencecauchy} is given by the relation
$$
F_{j_1}(\tau,\bxi') :=\dfrac{1-z^{-1}}{\tau} \, \sum_{m=0}^s \, \sum_{\sigma=m}^s z^{1+\sigma-m} \, 
Q_{1+\sigma}^\sharp(\bfeta) \, \widehat{f_{j_1}^m}(\bxi') \, .
$$

The precise expression of $F_{j_1}$ is not very useful. What is important is that, because of the support 
assumption on the initial data in Theorem \ref{thm1}, $F_{j_1}$ is zero for any index $j_1 \le 0$. Thanks 
to Assumption \ref{assumption2}, we also know that both $a_{-r_1}(z,\bfeta)$ and $a_{p_1}(z,\bfeta)$ 
are nonzero because $\tau \in \C^+$ and therefore $z={\rm e}^{\tau \, \Delta t} \in \U$. Hence 
\eqref{recurrencecauchy} is a recurrence relation of order (exactly equal to) $p_1+r_1$ for the sequence 
$(\widehat{u_{j_1}} (\tau,\bxi'))_{j_1 \in \Z}$. Introducing the vector
$$
U_{j_1} (\tau,\bxi') :=\begin{pmatrix}
\widehat{u_{j_1+p_1-1}}(\tau,\bxi') \\
\vdots \\
\widehat{u_{j_1-r_1}}(\tau,\bxi') \end{pmatrix} \in \C^{p_1+r_1} \, ,
$$
as well as the companion matrix
\begin{equation}
\label{defM}
\forall \, (z,\bfeta) \in \U \times \R^{d-1} \, ,\quad 
\M (z,\bfeta) := \begin{pmatrix}
-\dfrac{a_{p_1-1}(z,\bfeta)}{a_{p_1}(z,\bfeta)} & \dots & \dots & -\dfrac{a_{-r_1}(z,\bfeta)}{a_{p_1}(z,\bfeta)} \\
1& 0 & \dots & 0 \\
 & \ddots & \ddots & \vdots \\
0 & & 1 & 0 \end{pmatrix} \in {\mathcal M}_{p_1+r_1}(\C) \, ,
\end{equation}
we can equivalently rewrite \eqref{recurrencecauchy} as
$$
U_{j_1+1}(\tau,\bxi') -\M(z,\bfeta) \, U_{j_1} (\tau,\bxi') =\dfrac{1}{a_{p_1}(z,\bfeta)} \, \begin{pmatrix} 
F_{j_1}(\tau,\bxi') \\
0 \\
\vdots \\
0 \end{pmatrix} \, .
$$
In particular, there holds
\begin{equation}
\label{recurrencecauchy'}
\forall \, j_1 \le 0 \, ,\quad U_{j_1+1}(\tau,\bxi') =\M(z,\bfeta) \, U_{j_1} (\tau,\bxi') \, .
\end{equation}
The relations \eqref{relations} of Theorem \ref{thm1} are direct consequences of the recurrence relation 
\eqref{recurrencecauchy'}. Some fundamental properties of the matrix $\M(z,\bfeta)$ are stated in the 
following Lemma.

\begin{lemma}
\label{lem2}
Let Assumptions \ref{assumption0}, \ref{assumption1}, \ref{assumption2} and \ref{assumption3} be satisfied. 
Then for all $(z,\bfeta) \in \U \times \R^{d-1}$, the matrix $\M(z,\bfeta)$ in \eqref{defM} is invertible and has 
no eigenvalue on $\cercle$. If moreover Assumption \ref{assumption3} is satisfied, then $\M(z,\bfeta)$ has 
$r_1$ eigenvalues (counted with multiplicity) in $\D$ and the remaining $p_1$ eigenvalues in $\U$.
\end{lemma}

Let us take the result of Lemma \ref{lem2} for granted for a while. We can therefore introduce the 
eigenprojectors $\Pi^{s,u}(z,\bfeta)$ associated with the decomposition of $\C^{p_1+r_1}$ into
\begin{equation}
\label{decomposition1}
\C^{p_1+r_1} =\E^s (z,\bfeta) \oplus \E^u (z,\bfeta) \, .
\end{equation}
Here $\E^s$ refers to the {\it stable} subspace, that is the generalized eigenspace of $\M(z,\bfeta)$ 
associated with the eigenvalues in $\D$ and $\E^u$ refers to the {\it unstable} subspace associated 
with the eigenvalues of $\M$ in $\U$. By Lemma \ref{lem2}, the rank of $\Pi^s(z,\bfeta)$ is $r_1$ 
and the rank of $\Pi^u(z,\bfeta)$ is $p_1$.

Using the recurrence relation \eqref{recurrencecauchy'} and the fact that $(U_{j_1} (\tau,\bxi'))_{j_1 \in \Z}$ 
belongs to $\ell^2(\Z)$, the vector $U_1(\tau,\bxi')$ must necessarily belong to the unstable subspace 
$\E^u (z,\bfeta)$ and this property is propagated to any $j_1 \le 0$ by the recurrence \eqref{recurrencecauchy'} 
since $\E^u (z,\bfeta)$ is invariant by $\M$ and $\M^{-1}$. Using the projectors $\Pi^{s,u}(z,\bfeta)$ 
associated with \eqref{decomposition1}, we have obtained:
\begin{equation}
\label{relationshat}
\forall \, j_1 \le 1 \, ,\quad \Pi^s(z,\bfeta) \, U_{j_1}(\tau,\bxi') =0 \, .
\end{equation}

The relations \eqref{relationshat} are the Laplace-Fourier counterparts of the relations \eqref{relations} 
stated in Theorem \ref{thm1}. In what follows, we are going to prove Lemma \ref{lem2}. Then we shall 
explain how one can rewrite \eqref{relationshat} in the original physical space, which amounts to 
determining the inverse Laplace-Fourier transform of the left hand side in \eqref{relationshat}.
\bigskip

\begin{proof}[Proof of Lemma \ref{lem2}]
The proof is mostly the same as in \cite{kreiss1}, but since we intend to show that the argument is not 
restricted to discretized {\it hyperbolic} problems, we reproduce it here for the sake of completeness. 
Let $(z,\bfeta) \in \U \times \R^{d-1}$. The determinant of the matrix $\M(z,\bfeta)$ in \eqref{defM} equals
$$
(-1)^{p_1+r_1} \, a_{-r_1}(z,\bfeta)/a_{p_1}(z,\bfeta) \, ,
$$
this quantity being nonzero due to Assumption \ref{assumption2}. More generally, the characteristic 
polynomial of $\M(z,\bfeta)$ is
$$
\dfrac{1}{a_{p_1}(z,\bfeta)} \, \sum_{\ell_1=-r_1}^{p_1} a_{\ell_1}(z,\bfeta) \, X^{\ell_1+r_1} \, ,
$$
and the eigenvalues of $\M(z,\bfeta)$ are the roots $\kappa_1 \in \C \setminus \{ 0\}$ to the dispersion 
relation
$$
\sum_{\sigma=0}^{s+1} \widehat{Q_\sigma} (\kappa_1,{\rm e}^{i\, \eta_2},\dots,{\rm e}^{i\, \eta_d}) \, 
z^\sigma =0 \, .
$$
In particular, $\kappa_1$ does not belong to the unit circle $\cercle$ for otherwise \eqref{dispersion} would 
have a root $z$ in $\U$ for some $\bxi \in \R^d$ (and this would contradict Lemma \ref{lem0}). This means 
that the eigenvalues of $\M(z,\bfeta)$ split into two groups: the {\it stable} ones in $\D$ and the {\it unstable} 
ones in $\U$, which gives rise to the decomposition \eqref{decomposition1}. It remains to determine the 
dimension of each of the vector spaces in \eqref{decomposition1}.
\bigskip

We start with the case of explicit schemes, that is, case (i) in Assumption \ref{assumption3}. Following 
\cite[Lemma 2]{kreiss1} (see also \cite[Lemma 15]{jfcnotes} for a more detailed exposition), the number 
of stable eigenvalues is computed by analyzing their behavior when $z$ tends to infinity. Since $\U \times 
\R^{d-1}$ is connected, the number of eigenvalues of $\M(z,\bfeta)$ in $\D$ does not depend on $(z,\bfeta)$. 
Let now $\bfeta$ be fixed. As $z$ tends to infinity, all stable eigenvalues of $\M(z,\bfeta)$ converge to zero. 
This property holds for otherwise, there would exist a sequence $(z_n)$ with $|z_n|>n$ and there would exist 
a sequence $(\kappa_n)$ such that $\kappa_n \in \D \cap \text{\rm sp} \, (\M(z_n,\bfeta))$, and $\inf_n 
|\kappa_n|>0$. Up to extracting and relabeling, we can assume that the sequence $(\kappa_n)$ converges 
towards some nonzero complex number $\underline{\kappa}$. We now pass to the limit in the expression
$$
\dfrac{1}{z_n^{s+1}} \, \sum_{\ell_1=-r_1}^{p_1} a_{\ell_1}(z_n,\bfeta) \, \kappa_n^{\ell_1+r_1} =0 \, ,
$$
where the rescaling by $z_n^{-s-1}$ has been made in order to have (recall here that we consider explicit 
schemes):
$$
\dfrac{1}{z_n^{s+1}} \, a_{\ell_1}(z_n,\bfeta) \longrightarrow \delta_{\ell_1 0} \, ,
$$
see \eqref{defA-d}. Hence we get $\underline{\kappa}^{r_1} =0$, which is a contradiction. All the stable 
eigenvalues of $\M(z,\bfeta)$ tend to zero as $z$ tends to infinity (and similarly, all the unstable eigenvalues 
tend to infinity as $z$ tends to infinity). Setting $z=1/Z$, the number of stable eigenvalues of $\M(z,\bfeta)$ 
is computed for $Z$ small by counting the number of roots close to zero to the equation
$$
D(\kappa,Z) :=Z^{s+1} \, \sum_{\ell_1=-r_1}^{p_1} a_{\ell_1}(1/Z,\bfeta) \, \kappa^{\ell_1+r_1} =0 \, .
$$
Since $D$ is a polynomial in $(\kappa,Z)$ with $D(\kappa,0) =\kappa^{r_1}$, there are $r_1$ stable 
eigenvalues of $\M(z,\bfeta)$. The $p_1$ other eigenvalues of $\M(z,\bfeta)$ are the unstable ones.
\bigskip

We now deal with the case of implicit schemes, that is, case (ii) in Assumption \ref{assumption3}. Once 
again, we analyze the behavior of the eigenvalues of $\M(z,\bfeta)$ as $z$ tends to infinity. For fixed 
$\bfeta$ and $\kappa_1$, we compute
$$
\lim_{z \to \infty} \dfrac{1}{z^{s+1}} \, \sum_{\ell_1=-r_1}^{p_1} a_{\ell_1}(z,\bfeta) \, \kappa_1^{\ell_1+r_1} 
=\sum_{\ell_1=-r_1}^{p_1} \left( \sum_{\ell'=-r'}^{p'} a_{(\ell_1,\ell'),s+1}(\Delta t,\Delta x) \, 
{\rm e}^{i \, \ell' \cdot \bfeta} \right) \, \kappa_1^{\ell_1+r_1} \, .
$$
Using Assumptions \ref{assumption0} and \ref{assumption3}, we know that the equation in $\kappa_1$:
$$
\kappa_1^{r_1} \, \widehat{Q_{s+1}} (\kappa_1,{\rm e}^{i \, \eta_2},\dots,{\rm e}^{i\, \eta_d}) =
\sum_{\ell_1=-r_1}^{p_1} \left( \sum_{\ell'=-r'}^{p'} a_{(\ell_1,\ell'),s+1}(\Delta t,\Delta x) \, 
{\rm e}^{i \, \ell' \cdot \bfeta} \right) \, \kappa_1^{\ell_1+r_1} =0 \, ,
$$
is a polynomial equation of degree $p_1+r_1$, that its roots are nonzero and that it has no root on $\cercle$. 
Furthermore, the residue Theorem \cite{rudin} shows that the integral on the left hand side of \eqref{index} 
equals the number of zeroes in $\D$ of the holomorphic function $\widehat{Q_{s+1}} (\cdot,{\rm e}^{i \, \eta_2}, 
\dots, {\rm e}^{i\, \eta_d})$ minus the number of its poles in $\D$ (zeroes and poles being counted with 
multiplicity). By Assumption \ref{assumption3}, we know that there is only one pole at the origin with 
order $r_1$, so the number of zeroes in $\D$ is $r_1$. Summarizing, we have shown that the equation
$$
\sum_{\ell_1=-r_1}^{p_1} \left( \sum_{\ell'=-r'}^{p'} a_{(\ell_1,\ell'),s+1}(\Delta t,\Delta x) \, 
{\rm e}^{i \, \ell' \cdot \bfeta} \right) \, \kappa_1^{\ell_1+r_1} =0 \, ,
$$
has $r_1$ roots in $\D$ and it must therefore also have $p_1$ roots in $\U$. By the Rouch\'e Theorem 
\cite{rudin}, this implies that for any sufficiently large $z$, the equation
$$
\sum_{\ell_1=-r_1}^{p_1} a_{\ell_1}(z,\bfeta) \, \kappa_1^{\ell_1+r_1} =0 \, ,
$$
also has $r_1$ roots in $\D$ and $p_1$ roots in $\U$.
\end{proof}

$\bullet$ Step 2. The limit at infinity of the stable and unstable subspaces.

Let us first observe that because of Lemma \ref{lem2}, the stable and unstable spaces $E^{s,u}$ in 
\eqref{decomposition1} depend holomorphically on $z \in \U$ and analytically on $\bfeta \in \R^{d-1}$. 
Moreover, they are $2\, \pi$-periodic with respect to each coordinate of $\bfeta$ because the matrix 
$\M$ itself is $2\, \pi$-periodic with respect to each coordinate of $\bfeta$, see \eqref{defA-d}. Hence 
the projectors $\Pi^{s,u}$, which can be defined by integrals of $(w \, I-\M)^{-1}$ on suitable contours 
\cite{baumgartel}, depend holomorphically on $z \in \U$ and analytically on $\bfeta \in \R^{d-1}$. 
Assumption \ref{assumption3} comes back into play for studying the limit of the projectors $\Pi^{s,u}$ 
as $z$ tends to infinity, which amounts to determining the limits of $E^{s,u}$ as $z$ tends to infinity.

Let us first consider case {\rm (ii)} in Assumption \ref{assumption3}, that is, the case of implicit schemes. 
We have already seen in the proof of Lemma \ref{lem2} that for all $\bfeta \in \R^{d-1}$, the function
$$
\kappa_1 \in \C \setminus \{ 0 \} \longmapsto \widehat{Q_{s+1}} 
(\kappa_1,{\rm e}^{i\, \eta_2},\dots,{\rm e}^{i\, \eta_d}) \, ,
$$
has $r_1$ zeroes in $\D \setminus \{ 0 \}$ and $p_1$ zeroes in $\U$ (as always, zeroes are counted with 
multiplicity). We can then easily determine the behavior of the spectral projectors $\Pi^{s,u}(z,\bfeta)$ as 
$z$ tends to infinity. Indeed, by the definition \eqref{defM} of $\M(z,\bfeta)$, and the fact that both $a_{-r_1}$ 
and $a_{p_1}$ are polynomials of degree $s+1$ in $z$ (Assumption \ref{assumption3}), we find that 
$\M (z,\bfeta)$ has a limit as $z$ tends to infinity. This limit is given by
\begin{equation}
\label{limMinfty}
\begin{pmatrix}
-\dfrac{a_{\infty,p_1-1}(\bfeta)}{a_{\infty,p_1}(\bfeta)} & \dots & \dots & 
-\dfrac{a_{\infty,-r_1}(\bfeta)}{a_{\infty,p_1}(\bfeta)} \\
1& 0 & \dots & 0 \\
 & \ddots & \ddots & \vdots \\
0 & & 1 & 0 \end{pmatrix} \, ,\quad 
a_{\infty,\ell_1}(\bfeta) :=\sum_{\ell'=-r'}^{p'} a_{(\ell_1,\ell'),s+1}(\Delta t,\Delta x) \, 
{\rm e}^{i \, \ell' \cdot \bfeta} \, ,
\end{equation}
and we know from the previous arguments that this matrix has $r_1$ eigenvalues in $\D \setminus \{ 0 \}$ and 
$p_1$ eigenvalues in $\U$. In other words, the singularity at $Z=0$ of the projectors $\Pi^{s,u}(1/Z,\bfeta)$ is 
removable since the splitting between stable and unstable eigenvalues persists up to $Z=0$. In what follows, 
we let $\E^s(\infty,\bfeta)$, resp. $\E^u(\infty,\bfeta)$, denote the stable, resp. unstable, subspace of the matrix 
$\M(\infty,\bfeta)$ in \eqref{limMinfty}. We also let $\Pi^{s,u}(\infty,\bfeta)$ denote the projectors associated with 
the decomposition
$$
\C^{p_1+r_1} =\E^s (\infty,\bfeta) \oplus \E^u (\infty,\bfeta) \, ,
$$
which holds for any $\bfeta \in \R^{d-1}$.

Let us now turn to the case of explicit schemes, case {\rm (i)} in Assumption \ref{assumption3}. This is 
the case treated in \cite{kreiss1,jfcsinum}. From the proof of Lemma \ref{lem2}, we already know that the 
eigenvalues of $\M(z,\bfeta)$ are the roots $\kappa_1$ to the equation
$$
\sum_{\ell_1=-r_1}^{p_1} \kappa_1^{\ell_1+r_1} \, a_{\ell_1}(z,\bfeta) =0 \, .
$$
There are $r_1$ stable roots, all of them converging to zero as $z$ tends to infinity, and there are $p_1$ 
unstable roots, all of them tending to infinity as $z$ tends to infinity. We are going to determine an asymptotic 
expansion of the eigenvalues as $z$ tends to infinity. (Of course determining this behavior makes sense only 
if $r_1$ and/or $p_1$ is nonzero, and this is the reason why in Assumption \ref{assumption3} we have stated 
conditions only when these integers are nonzero, which we assume from now on.) We introduce the function
$$
D \, : \, (Z,\kappa_1) \longmapsto 
Z^{s+1} \, \sum_{\ell_1=-r_1}^{p_1} \kappa_1^{\ell_1+r_1} \, a_{\ell_1}(1/Z,\bfeta) \, ,
$$
that is a polynomial function of $(Z,\kappa_1)$, and that satisfies
$$
D(0,\kappa_1) =\kappa_1^{r_1} \, ,\quad \dfrac{\partial D}{\partial Z} (0,0) 
=\lim_{Z \to 0} Z^s \, a_{-r_1}(1/Z,\bfeta) 
=\sum_{\ell'=-r'}^{p'} a_{(-r_1,\ell'),s}(\Delta t,\Delta x) \, {\rm e}^{i \, \ell' \cdot \bfeta} \neq 0 \, .
$$
Applying the Puiseux expansions theory, for which we refer to \cite{baumgartel}, the $r_1$ eigenvalues of 
$\M(1/Z,\bfeta)$ close to zero thus have an asymptotic expansion of the form
$$
\kappa_1 \sim c \, Z^{1/r_1} +O(Z^{2/r_1}) \, ,\quad c \neq 0\, .
$$
Recall that the frequency $\bfeta$ is fixed here. In particular, the $r_1$ stable eigenvalues of $\M(z,\bfeta)$ 
are simple for all sufficiently large $z$, $\bfeta$ being fixed (the splitting comes from the $r_1$ possible 
branches for the $r_1$-th root of $1/z$). In a similar way, we can prove that the $p_1$ unstable eigenvalues 
of $\M(z,\bfeta)$ have the asymptotic expansion
$$
\kappa_1 \sim d \, z^{1/p_1} +O(1) \, ,\quad d \neq 0\, .
$$
At this stage, we know that for all sufficiently large $z$, the eigenvalues of $\M(z,\bfeta)$ are simple; the 
stable ones behave like $z^{-1/r_1}$ and the unstable ones behave like $z^{1/p_1}$ as $z$ tends to infinity. 
It remains to compute the limit of $\Pi^s(z,\bfeta)$ (the limit of $\Pi^u$ is directly obtained by using $\Pi^u 
=I-\Pi^s$), which will also give the limit of $\E^{s,u}(z,\bfeta)$ as $z$ tends to infinity.

Using the expression of the eigenvectors of a companion matrix, we know that the Vandermonde matrix
$$
\begin{pmatrix}
\kappa_{1,1}^{p_1+r_1-1} & \cdots & \kappa_{1,p_1+r_1}^{p_1+r_1-1} \\
\vdots & & \vdots \\
1 & \cdots & 1 \end{pmatrix} \, ,
$$
diagonalizes $\M$ for $z$ large enough (so that all eigenvalues are simple). We label the eigenvalues in 
such a way that $\kappa_{1,1},\dots,\kappa_{1,r_1}$ are the stable ones, and $\kappa_{1,r_1+1}, \dots, 
\kappa_{1,r_1+p_1}$ are the unstable ones. With this convention, there holds
$$
\Pi^s(z,\bfeta) =\begin{pmatrix}
\kappa_{1,1}^{p_1+r_1-1} & \cdots & \kappa_{1,p_1+r_1}^{p_1+r_1-1} \\
\vdots & & \vdots \\
1 & \cdots & 1 \end{pmatrix} \, \begin{pmatrix}
I_{r_1} & 0 \\
0 & 0 \end{pmatrix} \, \begin{pmatrix}
\kappa_{1,1}^{p_1+r_1-1} & \cdots & \kappa_{1,p_1+r_1}^{p_1+r_1-1} \\
\vdots & & \vdots \\
1 & \cdots & 1 \end{pmatrix}^{-1} \, .
$$
Our aim is to show the property
\begin{equation}
\label{limprojecteur}
\lim_{z \to \infty} \Pi^s(z,\bfeta) =\begin{pmatrix}
0 & 0 \\
0 & I_{r_1} \end{pmatrix} \, .
\end{equation}
Introducing the matrix
\begin{equation}
\label{defW}
W(z,\bfeta) :=\begin{pmatrix}
1 & \cdots & \kappa_{1,1}^{p_1+r_1-1} \\
\vdots & & \vdots \\
1 & \cdots & \kappa_{1,p_1+r_1}^{p_1+r_1-1} \end{pmatrix} \, ,
\end{equation}
and using the previous expression of $\Pi^s(z,\bfeta)$, proving \eqref{limprojecteur} is equivalent to proving 
the property
\begin{equation}
\label{limprojecteur'}
\lim_{z \to \infty} W(z,\bfeta)^{-1} \, \begin{pmatrix}
I_{r_1} & 0 \\
0 & 0 \end{pmatrix} \, W(z,\bfeta) =\begin{pmatrix}
I_{r_1} & 0 \\
0 & 0 \end{pmatrix} \, .
\end{equation}

We introduce the block decomposition of the matrix $W(z,\bfeta)$ in \eqref{defW} and of its inverse 
$W(z,\bfeta)^{-1}$ (we forget temporarily to recall the $(z,\bfeta)$ dependence of all matrices to make 
expressions a little lighter):
$$
W =\begin{pmatrix}
V & V_\sharp \\
V_\flat & V_\natural \end{pmatrix} \, ,\quad W^{-1} =\begin{pmatrix}
\widetilde{V} & \widetilde{V_\sharp} \\
\widetilde{V_\flat} & \widetilde{V_\natural} \end{pmatrix} \, ,
$$
where the upper left block has dimension $r_1 \times r_1$ and all other dimensions follow accordingly. 
The matrix on the left of \eqref{limprojecteur'}, whose limit we wish to compute, is given by:
$$
\begin{pmatrix}
\widetilde{V} \, V & \widetilde{V} \, V_\sharp \\
\widetilde{V_\flat} \, V & \widetilde{V_\flat} \, V_\sharp \end{pmatrix} 
=\begin{pmatrix}
I -\widetilde{V_\sharp} \, V_\flat & \widetilde{V} \, V_\sharp \\
-\widetilde{V_\natural} \, V_\flat & \widetilde{V_\flat} \, V_\sharp \end{pmatrix} \, ,
$$
where we have used the fact that $W^{-1} \, W$ is the identity. For computing the limit as $z$ tends to 
infinity, we need some bounds on all matrices involved in the latter expression. Let us first examine the 
block $V_\sharp$. From the definition \eqref{defW} of $W$ and the labeling of eigenvalues of $\M$, 
the block $V_\sharp$ involves powers at least equal to $r_1$ of stable eigenvalues of $\M$. The block 
$V_\flat$ involves powers at most equal to $r_1-1$ of unstable eigenvalues of $\M$. Hence we have 
the bounds
$$
\| V_\sharp \| =O(|z|^{-1}) \, ,\quad \| V_\flat \| =O(|z|^{(r_1-1)/p_1}) \, ,
$$
where here $\| \cdot \|$ denotes any norm on (not necessarily square) complex matrices, for instance the 
maximum of the modulus of the entries. We now need to estimate the blocks of $W^{-1}$, which is made 
possible thanks to the explicit and somehow classical formula (the formula comes from the Lagrange 
polynomial interpolation theory):
\begin{equation}
\label{formuleinv}
(W^{-1})_{kj} =(-1)^k \, \left( \prod_{m \neq j} (\kappa_{1,m}-\kappa_{1,j}) \right)^{-1} \, 
\sum_{\underset{m_1,\dots,m_{p_1+r_1-k} \neq j}{1 \le m_1 <\cdots <m_{p_1+r_1-k} \le p_1+r_1,}} 
\kappa_{1,m_1} \cdots \kappa_{1,m_{p_1+r_1-k}} \, .
\end{equation}
In this formula, the sum is understood as being equal to $1$ if $k=p_1+r_1$. Repeated and careful 
applications of \eqref{formuleinv} lead to the following bounds for the blocks of $W^{-1}$:
$$
\| \widetilde{V} \| =O(|z|^{1-1/r_1}) \, ,\quad \| \widetilde{V_\sharp} \| =O(|z|^{-r_1/p_1-1/r_1}) \, ,\quad 
\| \widetilde{V_\flat} \| =O(|z|^{1-1/p_1-1/r_1}) \, ,\quad \| \widetilde{V_\natural} \| =O(|z|^{-r_1/p_1}) \, .
$$
In particular, when combined with the bounds on $V_\sharp$ and $V_\flat$, we can show that the four 
products of matrices $\widetilde{V_\sharp} \, V_\flat$, $\widetilde{V} \, V_\sharp$, $\widetilde{V_\natural} 
\, V_\flat$, and $\widetilde{V_\flat} \, V_\sharp$ tend to zero, which completes the proof of 
\eqref{limprojecteur'}.

Since we have \eqref{limprojecteur}, and the other limit
$$
\lim_{z \to \infty} \Pi^u(z,\bfeta) =\begin{pmatrix}
I_{p_1} & 0 \\
0 & 0 \end{pmatrix} \, ,
$$
we get the limits of the stable and unstable subspaces as well. Namely, the limit $\E^s(\infty,\bfeta)$ 
of $\E^s(z,\bfeta)$ is the vector space spanned by the last $r_1$ vectors in the canonical basis of 
$\C^{p_1+r_1}$, and the limit $\E^u(\infty,\bfeta)$ of $\E^u(z,\bfeta)$ is the vector space spanned by 
the first $p_1$ vectors in the canonical basis of $\C^{p_1+r_1}$. With that definition, we obviously 
have the splitting \eqref{decomposition1} that persists up to $z=\infty$, as in the implicit case.
\bigskip

$\bullet$ Step 3. The transparent conditions in the physical variables.

In the previous step, we have seen that the projector $\Pi^s(z,\bfeta)$ onto the stable subspace of 
$\M(z,\bfeta)$ has a limit as $z$ tends to infinity. For explicit schemes, this limit is independent of 
$\bfeta$ and is given by \eqref{limprojecteur}. Consequently, $\Pi^s$ extends as a function on 
$(\U \cup \{ \infty \}) \times \R^{d-1}$ that depends holomorphically on $z$ and analytically on 
$\bfeta$ with the additional property of being $2\, \pi$-periodic with respect to each coordinate 
of $\bfeta$. We can therefore write the Laurent expansion of $\Pi^s$ under the form
$$
\forall \, (z,\bfeta) \in (\U \cup \{ \infty \}) \times \R^{d-1} \, ,\quad 
\Pi^s (z,\bfeta) =\sum_{n \ge 0} z^{-n} \, \Pi_n(\bfeta) \, ,
$$
where the convergence is normal on every compact subset. In particular, since each matrix $\Pi_n$ 
depends analytically and in a periodic way on $\bfeta$, the convergence of the Laurent series is normal 
on any set of the form $\{ |z| \ge 1+\delta \} \times \R^{d-1}$, $\delta>0$, that is:
\begin{equation}
\label{boundpin}
\forall \, \delta >0 \, ,\quad 
\sum_{n \ge 0} \dfrac{1}{(1+\delta)^n} \, \sup_{\bfeta \in \R^{d-1}} \| \Pi_n(\bfeta) \| <+\infty \, .
\end{equation}
We go back to the definition of the vector $U_{j_1}(\tau,\bxi')$ in \eqref{relationshat}, use the Laurent 
expansion of $\Pi^s$ and the expression \eqref{laplaceu} of the Laplace-Fourier transform of $u_{j_1}$ 
to get
\begin{equation}
\label{relationsfourier}
\forall \, n \in \N \, ,\quad \forall \, j_1 \le 1 \, ,\quad 
\sum_{m=0}^n \Pi_{n-m}(\bfeta) \, \begin{pmatrix}
\widehat{u_{j_1+p_1-1}^m} (\bxi') \\
\vdots \\
\widehat{u_{j_1-r_1}^m} (\bxi') \end{pmatrix} =0 \, ,
\end{equation}
where we recall that here, the `hat' denotes partial Fourier transform with respect to $(x_2,\dots,x_d)$, 
and $\bfeta$ is a short notation for $(\xi_2 \, \Delta x_2,\dots,\xi_d \, \Delta x_d)$. The sequence of 
operators $(\bpi_n)_{n \in \N}$ is then defined as the following sequence of Fourier multipliers: for 
any $n \in \N$ and any sequence $v \in \ell^2(\Z^{d-1};\C^{p_1+r_1})$, we identify the sequence $v$ 
and the corresponding step function
$$
v(x') :=v_{j'} \, ,\quad \forall \, x' \in \prod_{k=2}^d \, [j_k \, \Delta x_k;(j_k+1) \, \Delta x_k) \, .
$$
In particular, the Fourier transform of the sequence $v$ means the Fourier transform of the corresponding 
step function. Then $\bpi_n \, v$ is defined as the sequence\footnote{It is indeed a rather standard result 
that Fourier multipliers associated with periodic symbols map the set of $L^2$ step functions into itself. We 
can therefore equivalently view the operator $\bpi_n$ as acting on $\ell^2(\Z^{d-1})$ with values in $\ell^2 
(\Z^{d-1})$ rather than acting on the set of $L^2$ step functions with values in $L^2(\R^{d-1})$.} whose 
Fourier transform is given by
\begin{equation}
\label{defpin}
\bxi' \in \R^{d-1} \longmapsto \Pi_n(\bfeta) \, \widehat{v} (\bxi') \, .
\end{equation}
With this definition for the $\bpi_n$'s, applying the inverse Fourier transform to \eqref{relationsfourier} gives 
\eqref{relations}. To complete the proof of Theorem \ref{thm1}, it only remains to show that the $\bpi_n$'s 
satisfy the growth condition and the algebraic constraints stated in Theorem \ref{thm1}. The growth condition 
is a direct consequence of the bound \eqref{boundpin} on the symbols $(\Pi_n)$ of the Fourier multipliers 
$(\bpi_n)$. The algebraic constraints follow by using the fact that $\Pi^s$ is a projector (up to now we have 
only used that $E^u$ is the kernel of $\Pi^s$). Expanding the equality $\Pi^s (z,\bfeta)^2 =\Pi^s (z,\bfeta)$ in 
Laurent series with respect to $z$, we get the algebraic constraints for the symbols:
$$
\forall \, n \in \N \, ,\quad \forall \, \bfeta \in \R^{d-1} \, ,\quad 
\Pi_n(\bfeta) =\sum_{m=0}^n \Pi_m (\bfeta) \, \Pi_{n-m} (\bfeta) \, .
$$
The relations satisfied by the $\bpi_n$'s follow immediately. This completes the proof of Theorem \ref{thm1}.

\subsection{Alternative formulations of transparent boundary conditions}

In this paragraph, we explain why the formulation of the transparent conditions encoded in the operators 
$\bpi_n$ in \eqref{relations} is not unique and what other choices, that may be more suitable from a practical 
point of view, can be made.

\subsubsection{Alternative formulation with projectors}

In the Laplace-Fourier variables, the recurrence relation \eqref{recurrencecauchy'} implies that the 
vector $U_1(\tau,\bxi')$ belongs to the unstable subspace $\E^u(z,\bfeta)$ of $\M(z,\bfeta)$. Using 
the decomposition \eqref{decomposition1}, we have equivalently formulated this property in writing 
\eqref{relationshat}. The arbitrariness here lies in the choice of the supplementary vector space 
$\E^s(z,\bfeta)$. More precisely, let us assume that one can choose a vector space $\widetilde{\E}^s 
(z,\bfeta)$ of dimension $r_1$ in $\C^{p_1+r_1}$, depending holomorphically on $z \in \U \cup \{ \infty\}$, 
analytically on $\bfeta \in \R^{d-1}$ and $2\, \pi$-periodically with respect to each coordinate of 
$\bfeta$, and such that one has the decomposition\footnote{We recall here that $\E^u$ has a limit 
when $z$ tends to infinity, which we have denoted $\E^u (\infty,\bfeta)$.}:
\begin{equation}
\label{decomposition2}
\forall \, (z,\bfeta) \in (\U \cup \{ \infty \}) \times \R^{d-1} \, ,\quad 
\C^{p_1+r_1} =\widetilde{\E}^s (z,\bfeta) \oplus \E^u (z,\bfeta) \, .
\end{equation}
We then let $\widetilde{\Pi}^{s,u}(z,\bfeta)$ denote the corresponding projectors\footnote{Of course, 
$\widetilde{\Pi}^u(z,\bfeta)$ does not coincide with $\Pi^u(z,\bfeta)$, even though one of the vector 
spaces in \eqref{decomposition2} is the same as in \eqref{decomposition1}.}. Then one can equivalently 
rewrite \eqref{relationshat} as
$$
\forall \, j_1 \le 1 \, ,\quad \widetilde{\Pi}^s(z,\bfeta) \, U_{j_1}(\tau,\bxi') =0 \, ,
$$
and since we already know that $\widetilde{\Pi}^s$ is holomorphic in $z$ on $\U \cup \{ \infty \}$, the Laurent 
series of $\widetilde{\Pi}^s$ only involves nonpositive powers of $z$. We can therefore reproduce the same 
arguments as in Step 3 of the proof of Theorem \ref{thm1}, which gives rise in the end to a set of relations
$$
\forall \, n \in \N \, ,\quad \forall \, j_1 \le 0 \, ,\quad 
\sum_{m=0}^n \widetilde{\bpi}_{n-m} \, \begin{pmatrix}
u_{(j_1+p_1,\cdot)}^m \\
\vdots \\
u_{(j_1+1-r_1,\cdot)}^m \end{pmatrix} =0 \, ,
$$
with a definition of the operators $\widetilde{\bpi}_{n-m}$ that is entirely analogous to the one of the 
$\bpi_n$'s. What really matters is not the sequence of operators $(\bpi_n)$ in \eqref{relations} but rather 
the kernel and the range of $\bpi_0$ for instance. This is the reason why there is some possible freedom 
in the formulation of \eqref{relations}.

Of course, the supplementary vector space $\E^s(z,\bfeta)$ in \eqref{decomposition1} is a rather natural 
choice, at least because when the eigenvalues of $\M(z,\bfeta)$ are simple, the projectors $\Pi^{s,u} 
(z,\bfeta)$ are given in terms of a Vandermonde matrix and its inverse for which explicit expressions 
are available\footnote{There are also explicit expressions when some eigenvalues are not simple but 
the algebra gets more involved.}. Therefore there does not seem to be much simplification in choosing 
another supplementary vector space to $\E^u$, though one should keep in mind that it is a possibility.

\subsubsection{Alternative formulation with linear forms}

The other formulation that we propose seems to be much more used in practice, see e.g. 
\cite{ehrhardt-arnold,AES,zisowski-ehrhardt,ehrhardt,ducomet-zlotnik,zheng-wen-han,BELV}. 
It is specifically recommended in the case $r_1=1$ for then $\E^u(z,\bfeta)$ is a hyperplane in 
$\C^{p_1+r_1}$ which one can consider as the kernel of some linear form.

In full generality, we know that $\E^u$ defines a holomorphic/analytic-periodic vector bundle over 
$(\U \cup \{ \infty \}) \times \R^{d-1}$. By holomorphic/analytic-periodic, it should be clear by now 
that we mean holomorphic with respect to $z \in \U \cup \{ \infty \}$, analytic with respect to $\bfeta 
\in \R^{d-1}$ and $2\, \pi$-periodic with respect to each coordinate of $\bfeta$. In the same way, 
$\E^s$ defines a holomorphic/analytic-periodic vector bundle over $(\U \cup \{ \infty \}) \times \R^{d-1}$, 
and the fiber $\E^s(z,\bfeta)$, resp. $\E^u(z,\bfeta)$, coincides with the range of the projector $\Pi^s 
(z,\bfeta)$, resp. $\Pi^u(z,\bfeta)$. By the transformation $z \rightarrow 1/z$, the set $\U \cup \{ \infty \}$ 
is mapped biholomorphically onto the unit disk $\D$, which is simply connected. By following the argument 
in \cite[Chapter 2.4]{kato}, we can thus determine for all $\bfeta \in \R^{d-1}$ a basis $e_1^s(z,\bfeta), \dots, 
e_{r_1}^s(z,\bfeta)$ of $\E^s(z,\bfeta)$, resp. a basis $e_1^u(z,\bfeta),\dots,e_{p_1}^u(z,\bfeta)$ of $\E^u 
(z,\bfeta)$, that depends holomorphically on $z \in \U \cup \{ \infty \}$. (We do not concentrate at first on 
the dependence with respect to $\bfeta$ and consider $\bfeta$ as a fixed parameter for now.)

The inverse of the matrix
$$
\big( e_1^s(z,\bfeta) \, \cdots \, e_{r_1}^s(z,\bfeta) \, e_1^u(z,\bfeta) \, \cdots \, e_{p_1}^u(z,\bfeta) \big) \in 
{\mathcal M}_{p_1+r_1}(\C) \, ,
$$
provides with $r_1$ row vectors $L_1(z,\bfeta),\dots,L_{r_1}(z,\bfeta) \in {\mathcal M}_{1,p_1+r_1}(\C)$ 
that depend holomorphically on $z \in \U \cup \{ \infty \}$ and such that the unstable subspace $\E^u$ 
reads
$$
\E^u(z,\bfeta) = \{ X \in \C^{p_1+r_1} \, / \, L_1(z,\bfeta) \, X =\cdots = L_{r_1}(z,\bfeta) \, X =0 \} \, .
$$
In other words, for all $\bfeta \in \R^{d-1}$, we have constructed a matrix $L(z,\bfeta) \in 
{\mathcal M}_{r_1,p_1+r_1}(\C)$ of {\it full rank} $r_1$, that depends holomorphically on $z$ on $\U \cup 
\{ \infty \}$, and whose kernel is $\E^u(z,\bfeta)$. The relations \eqref{relationshat} in the proof of Theorem 
\ref{thm1} can be equivalently recast as
$$
\forall \, j_1 \le 1 \, ,\quad L(z,\bfeta) \, U_{j_1}(\tau,\bxi') =0 \, ,
$$
with the major gain, from an algebraic point of view, that $L$ has full rank so that any perturbation (meaning 
a nonzero vector in $\C^{r_1}$ on the right hand side is now admissible in view of a stability analysis, which 
bypasses the algebraic constraints of Lemma \ref{lem1}). We can expand $L$ in Laurent series
$$
L(z,\bfeta) =\sum_{n \ge 0} \dfrac{1}{z^n} \, L_n(\bfeta) \, ,
$$
apply the inverse Laplace transform and rewrite the above relations as
$$
\forall \, n \in \N \, ,\quad \forall \, j_1 \le 1 \, ,\quad 
\sum_{m=0}^n L_{n-m}(\bfeta) \, \begin{pmatrix}
\widehat{u_{j_1+p_1-1}^m} (\bxi') \\
\vdots \\
\widehat{u_{j_1-r_1}^m} (\bxi') \end{pmatrix} =0 \, ,
$$
which is the analogue of \eqref{relationsfourier}.

Assume now that the above construction of the matrix $L$ can be performed in such a way that $L$ 
depends analytically and $2\, \pi$-periodically on $\bfeta$. Analyticity can be obtained by using the 
same procedure as in \cite[Chapter 4.6]{benzoni-serre} (which corresponds to applying the method 
of \cite{kato} coordinate by coordinate), but periodicity seems far from obvious if one constructs the 
above basis by ODE arguments as in \cite{kato}. Assume nevertheless that periodicity can be achieved 
(for instance, as in the example below, because one has explicit expressions). Then one has all the 
desirable properties for applying inverse Fourier transform in the previous relations and write the 
relations in the original physical variables under the form
$$
\forall \, n \in \N \, ,\quad \forall \, j_1 \le 0 \, ,\quad 
\sum_{m=0}^n {\bf L}_{n-m} \, \begin{pmatrix}
u_{(j_1+p_1,\cdot)}^m \\
\vdots \\
u_{(j_1+1-r_1,\cdot)}^m \end{pmatrix} =0 \, ,
$$
with suitable Fourier multipliers ${\bf L}_n$ on $\ell^2(\Z^{d-1})$. The nice feature here is that one no 
longer needs to care about compatibility conditions for the boundary forcing terms (if present). There 
may of course still remain some indeterminacy in the construction of $L$ due again to the choice of a 
supplementary vector space to $\E^u$ and to the choice of a basis.

In the case $d=1$, the argument in \cite{kato} allows us to construct the matrix $L(z)$ as above, 
with $L$ holomorphic on $\U \cup \{ \infty \}$. Such a construction is the one that is used for instance 
in \cite{ehrhardt-arnold,BELV,BMN}. When $d$ is larger than $1$, the main difficulty is to construct 
$L(z,\bfeta)$ with the property that $L$ is $2\, \pi$-periodic with respect to each coordinate of $\bfeta$. 
If we go through the arguments above, we could construct such an $L$ provided that we have a basis 
of $\E^s$ and a basis of $\E^u$ that depend holomorphically/analytically/periodically on $(z,\bfeta)$ 
in $(\U \cup \{ \infty \}) \times \R^{d-1}$. Constructing such bases is far from obvious, because it 
amounts to showing that the vector bundles $\E^s$ and $\E^u$ are trivial, but here, because of 
the periodicity in $\bfeta$, these vector bundles are considered over $(\U \cup \{ \infty \}) \times 
(\cercle)^{d-1}$, which is {\it not contractile} (unless $d=1$). However, there are examples for which 
one can show from an explicit expression of $\E^s(z,\bfeta)$ and $\E^u(z,\bfeta)$ that the bundles 
over $(\U \cup \{ \infty \}) \times (\cercle)^{d-1}$ are indeed trivial, and this allows then to construct 
the matrix $L(z,\bfeta)$ even for problems with $d \ge 2$.

\subsubsection{The simplest and most frequent example}

Let us consider for simplicity the case $d=1$ so that the previous technical difficulties encountered 
because of the tangential Fourier variables $\bfeta$ disappear. Let us assume furthermore $p_1=r_1=1$. 
In other words, the original numerical scheme \eqref{cauchy} takes the form:
\begin{equation}
\label{cauchy11}
\begin{cases}
{\dps \sum_{\sigma=0}^{s+1}} \, a_{-1,\sigma} \, u_{j-1}^{n+\sigma} +a_{0,\sigma} \, u_j^{n+\sigma} 
+a_{1,\sigma} \, u_{j+1}^{n+\sigma} =0 \, ,& n\ge 0 \, , \, j \in \Z \, ,\\
(u^0,\dots,u^s) = (f^0,\dots,f^s) \in \ell^2 (\Z)^{s+1} \, . &
\end{cases}
\end{equation}
Of course, the coefficients $a_{-1,\sigma},a_{0,\sigma},a_{1,\sigma}$ in \eqref{cauchy11} may depend on 
$\Delta t$ and/or $\Delta x$.

Let us emphasize how the various assumptions on \eqref{cauchy} translate in the particular case 
\eqref{cauchy11}. Assumptions \ref{assumption0} and \ref{assumption3} mean:
\begin{description}
 \item[For explicit schemes:] $a_{-1,s+1} = a_{1,s+1} = 0$, $a_{0,s+1}=1$, and $a_{-1,s} \, a_{1,s} \neq 0$ 
 (this last condition may restrict the possible values of the discretization parameters $\Delta t,\Delta x$, see 
 the example of the Lax-Wendroff scheme in Section \ref{sect:examples}).
 
 \item[For implicit schemes:] $a_{-1,s+1} \, a_{1,s+1} \neq 0$, and the polynomial (in $\kappa$)
$$
a_{-1,s+1} +a_{0,s+1} \, \kappa +a_{1,s+1} \, \kappa^2 \, ,
$$
 has one root in $\D$ (necessarily not zero) and one root in $\U$. If the coefficients are real, this means that 
 one root belongs to $(-1,1)$ and one root belongs to $(-\infty,-1) \cup (1,+\infty)$ for the roots cannot be 
 complex conjugate.
\end{description}
Assumption \ref{assumption2} means that both polynomials
$$
\sum_{\sigma=0}^{s+1} a_{-1,\sigma} \, z^\sigma \, ,\qquad 
\sum_{\sigma=0}^{s+1} a_{1,\sigma} \, z^\sigma \, ,
$$
have no root in $\U$ (Assumption \ref{assumption2'} means that they have no root in $\Ubar$). For explicit 
schemes, these polynomials have degree $s$ while they have degree $s+1$ for implicit schemes. In 
particular, Assumption \ref{assumption2'} is automatically satisfied if the scheme is explicit and $s=0$ for 
then the previous two polynomials are constant and nonzero.

Let us now turn to our derivation of the transparent boundary conditions. The matrix $\M$ of interest is 
defined in \eqref{defM}. In one space dimension with $p=r=1$, its expression reduces to
\begin{equation}
\label{defM11}
\M (z) := \begin{pmatrix}
-\dfrac{a_0(z)}{a_1(z)} & -\dfrac{a_{-1}(z)}{a_1(z)} \\
1& 0 \end{pmatrix} \in {\mathcal M}_2(\C) \, ,\quad 
a_\ell(z) :=\sum_{\sigma=0}^{s+1} a_{\ell,\sigma} \, z^\sigma \, .
\end{equation}
We know from the proof of Theorem \ref{thm1} that $\M(z)$ has one eigenvalue $\kappa_s(z)$ in $\D$ 
and one eigenvalue $\kappa_u(z)$ in $\U$, both counted with multiplicity and therefore simple. Since 
the eigenvalues do not cross for $z \in \U$, they both depend holomorphically on $z \in \U$. The stable 
eigenvalue $\kappa_s(z)$ has a limit when $z$ tends to infinity, while the unstable eigenvalue $\kappa_u(z)$ 
has a limit when $z$ tends to infinity only when the scheme is implicit (for an explicit scheme, $\kappa_u(z)$ 
tends to infinity when $z$ tends to infinity). The stable and unstable subspaces read
$$
\E^s(z) =\text{\rm Span } \begin{pmatrix}
\kappa_s(z) \\
1 \end{pmatrix} \, ,\quad \E^u(z) =\text{\rm Span } \begin{pmatrix}
1 \\
\kappa_u(z)^{-1} \end{pmatrix} \, ,
$$
where the choice for parametrizing $\E^u$ has been made in such a way that the generating vector has 
a limit when $z$ tends to infinity (even when the scheme is explicit).

Let us now discuss two possible ways of writing the transparent boundary conditions. We first follow the 
approach based on the spectral projectors as in the proof of Theorem \ref{thm1}. The decomposition:
$$
\forall \, z \in \U \cup \{ \infty \} \, ,\quad \C^2 =\E^s(z) \oplus \E^u (z) \, ,
$$
is endowed with two projectors $\Pi^{s,u}(z)$, whose explicit expression is given by
$$
\Pi^s(z) =\dfrac{1}{\kappa_s(z) -\kappa_u(z)} \, \begin{pmatrix}
\kappa_s(z) & -\kappa_s(z) \, \kappa_u(z) \\
1 & -\kappa_u(z) \end{pmatrix} \, ,\quad 
\Pi^u(z) =I-\Pi^s(z) \, .
$$
Observe in particular that for explicit schemes, knowing $\kappa_s \rightarrow 0$ and $\kappa_u \rightarrow 
\infty$ as $z$ tends to $\infty$, we recover the asymptotic behavior
$$
\lim_{z \to \infty} \Pi^s(z) =\begin{pmatrix}
0 & 0 \\
0 & 1 \end{pmatrix} \, ,\quad \lim_{z \to \infty} \Pi^u(z) =\begin{pmatrix}
1 & 0 \\
0 & 0 \end{pmatrix} \, .
$$
We can either write the vector space $\E^u (z)$ as the kernel of the matrix $\Pi^s(z)$ or as the kernel of 
the linear form
$$
x =\begin{pmatrix}
x_1 \\
x_2 \end{pmatrix} \in \C^2 \longmapsto \dfrac{1}{\kappa_u(z)} \, x_1 -x_2 \, .
$$
Again, it is more convenient to use $\kappa_u(z)^{-1}$ instead of $\kappa_u(z)$ because in our framework, 
$\kappa_u(z)^{-1}$ always has a limit as $z$ tends to infinity.

The Laurent series of both $\kappa_s$ and $\kappa_u^{-1}$ can be obtained by starting from the relations
$$
a_{-1}(z) +a_0(z) \, \kappa_s(z) +a_1(z) \, \kappa_s(z)^2 =0 \, ,\quad 
a_{-1}(z) \, \kappa_u(z)^{-2} +a_0(z) \, \kappa_u(z)^{-1} +a_1(z) =0 \, ,
$$
and by identifying inductively the coefficients in the series
$$
\kappa_s(z) =\sum_{n \ge 0} \dfrac{\kappa_{s,n}}{z^n} \, ,\quad 
\kappa_u(z)^{-1} =\sum_{n \ge 0} \dfrac{\widetilde{\kappa}_{u,n}}{z^n} \, .
$$
For the numerical scheme considered in \cite{ehrhardt-arnold}, one can get an explicit expression of the 
coefficients in terms of the Legendre polynomials. This expression does not come from the general method 
we propose. For explicit schemes, both $\kappa_{s,0}$ and $\widetilde{\kappa}_{u,0}$ are zero, and both 
$\kappa_{s,1}$ and $\widetilde{\kappa}_{u,1}$ are nonzero. One can also determine the Laurent series 
expansion of $\kappa_u$ by writing (the coefficient $\kappa_{u,1}$ being nonzero for explicit schemes):
$$
\kappa_u(z) =\kappa_{u,1} \, z +\sum_{n \ge 0} \dfrac{\kappa_{u,n}}{z^n} \, ,
$$
by plugging this series into the equation
$$
a_{-1}(z) +a_0(z) \, \kappa_u(z) +a_1(z) \, \kappa_u(z)^2 =0 \, ,
$$
and by identifying inductively the coefficients $\kappa_{u,n}$. At least theoretically speaking, this gives access 
to the Laurent series expansion of any useful quantity involving $\kappa_s$ and $\kappa_u$. In particular, we 
have access to the Laurent series of $\Pi^s$, which gives rise in the original physical variables to the relations 
(recall here that we consider a one-dimensional problem so the Fourier multipliers $\bpi_n$ reduce in fact to 
$2 \times 2$ matrices):
$$
\forall \, n \in \N \, ,\quad \forall \, j_1 \le 0 \, ,\quad 
\sum_{m=0}^n \Pi_{n-m} \, \begin{pmatrix}
u_{j_1+1}^m \\
u_{j_1}^m \end{pmatrix} =0 \, .
$$

Viewing the vector space $\E^u(z)$ as the kernel of a linear form, we may use the Laurent expansion of 
$\kappa_u^{-1}$ and rewrite equivalently the latter relations as
\begin{equation}
\label{transparent1}
\forall \, n \in \N \, ,\quad \forall \, j_1 \le 0 \, ,\quad 
u_{j_1}^n =\sum_{m=0}^n \widetilde{\kappa}_{u,n-m} \, u_{j_1+1}^m \, .
\end{equation}
Using the formulation \eqref{transparent1}, the new (though equivalent) way of writing the transparent 
boundary conditions in \eqref{transparent0} becomes
\begin{equation}
\label{transparent11}
\begin{cases}
{\dps \sum_{\sigma=0}^{s+1}} \, a_{-1,\sigma} \, u_{j-1}^{n+\sigma} +a_{0,\sigma} \, u_j^{n+\sigma} 
+a_{1,\sigma} \, u_{j+1}^{n+\sigma} =\Delta t \, F_j^{n+s+1} \, ,& n\ge 0 \, ,\quad j \ge 1 \, ,\\
u_0^{n+s+1} -\widetilde{\kappa}_{u,0} \, u_1^{n+s+1} 
={\dps \sum_{m=0}^{n+s}} \widetilde{\kappa}_{u,n+s+1-m} \, u_1^m +g^{n+s+1} \, , & n\ge 0 \, ,\\
(u_j^0,\dots,u_j^s) = (f_j^0,\dots,f_j^s) \, , & j \ge 0 \, .
\end{cases}
\end{equation}
For explicit schemes, $\widetilde{\kappa}_{u,0}$ is zero so the boundary condition in \eqref{transparent11} 
takes the form of a non-homogeneous Dirichlet boundary condition, whose source term is computed thanks 
to the trace at $j=1$ of the solution at earlier times.

\section{Solvability of the scheme with transparent boundary conditions}
\label{sect:3}

\noindent From now on, we analyze the numerical scheme \eqref{transparent0}. We first prove Lemma 
\ref{lem1} which characterizes the sequence of source terms for which one can construct a solution to 
the sequence of `convolution' equations \eqref{convolution}. Lemma \ref{lem1} will provide with necessary 
and sufficient compatibility conditions for solving \eqref{transparent0}.

\begin{proof}[Proof of Lemma \ref{lem1}]
The whole proof is based on the mere fact that $P_0$ is a projector. Hence a vector $y$ belongs to the range 
of $P_0$ if and only if $P_0 \, y=y$.

Let us first assume that the sequence $(y_n)$ is such that there exists a solution $(x_n)$ to \eqref{convolution}. 
For $n=0$, there holds $y_0 \in \text{\rm Im } P_0$ and therefore \eqref{compatibilitebord} holds for $n=0$. Let 
us assume by induction that \eqref{compatibilitebord} holds up to some integer $n$. Then we can write
$$
y_{n+1} -\sum_{m=0}^n P_{n+1-m} \, x_m =P_0 \, x_{n+1} \in \text{\rm Im } P_0 \, ,
$$
and consequently
\begin{align}
y_{n+1} -P_0 \, y_{n+1} &=\sum_{m=0}^n P_{n+1-m} \, x_m -\sum_{m=0}^n P_0 \, P_{n+1-m} \, x_m \notag \\
&= (P_{n+1} -P_0 \, P_{n+1}) \, x_0 +\sum_{m=1}^n (P_{n+1-m} -P_0 \, P_{n+1-m}) \, x_m \, .\label{prooflem1}
\end{align}
We now use the equations satisfied by the $P_n$'s and write
\begin{align*}
(P_{n+1} -P_0 \, P_{n+1}) \, x_0 &=P_{n+1} \, P_0 \, x_0 +\sum_{p=1}^n P_p \, P_{n+1-p} \, x_0 \\
&=P_{n+1} \, y_0 +\sum_{p=1}^n P_p \, P_{n+1-p} \, x_0 \, .
\end{align*}
We can therefore simplify \eqref{prooflem1} and obtain
\begin{equation}
\label{prooflem1'}
y_{n+1} -P_0 \, y_{n+1}-P_{n+1} \, y_0 =\sum_{m=1}^n (P_{n+1-m} -P_0 \, P_{n+1-m}) \, x_m 
+\sum_{p=1}^n P_p \, P_{n+1-p} \, x_0 \, .
\end{equation}
If $n=0$, the work is over since the sums on the right hand side of \eqref{prooflem1'} vanish and we have 
obtained \eqref{compatibilitebord} for $n+1=1$. Assuming $n \ge 1$ from now on, we substitute $P_{n+1-p} 
\, x_0$ in \eqref{prooflem1'}, $p=1,\dots,n$, for
$$
y_{n+1-p} -\sum_{k=0}^{n-p} P_{n+1-p-k} \, x_k \, ,
$$
by using \eqref{convolution}. Using the algebraic relations satisfied by the $P_k$'s and some manipulations 
on the indices, we end up with
$$
y_{n+1} -P_0 \, y_{n+1}-P_{n+1} \, y_0 =\sum_{p=1}^n P_p \, y_{n+1-p} \, ,
$$
which is nothing but \eqref{compatibilitebord} for $n+1$.

Let us now assume that the sequence $(y_n)$ satisfies the compatibility conditions \eqref{compatibilitebord}. 
Then one immediately sees that the sequence $(x_n)$ defined by $x_n :=y_n$ for all $n$ satisfies 
\eqref{convolution}. The proof of Lemma \ref{lem1} is thus complete.
\end{proof}

\noindent It remains to use Lemma \ref{lem1} for proving the solvability result of Proposition \ref{prop1}.

\begin{proof}[Proof of Proposition \ref{prop1}]
We first show that if there exists a solution to \eqref{transparent0}, then it is necessarily unique. Of course, 
by linearity, this amounts to showing that the only solution to
\begin{equation}
\label{transparent0'}
\begin{cases}
{\dps \sum_{\sigma=0}^{s+1}} \, Q_\sigma \, u_j^{n+\sigma} =0 \, ,& 
n\ge 0 \, ,\quad j_1 \ge 1 \, ,\\
{\dps \sum_{m=0}^{n+s+1}} \, \bpi_{n+s+1-m} \, \begin{pmatrix}
u_{(p_1,\cdot)}^m \\
\vdots \\
u_{(1-r_1,\cdot)}^m \end{pmatrix} = 0 \, , & n\ge 0 \, ,\\
(u_j^0,\dots,u_j^s) = (0,\dots,0) \, , & j_1 \ge 1-r_1 \, ,
\end{cases}
\end{equation}
is zero. We prove this property by induction on the time level. Let us assume that up to some level $n \ge 0$, 
there holds $u^0=\cdots=u^{n+s}=0$ (this property clearly holds for $n=0$ due to the initial data in 
\eqref{transparent0'}). Then $u^{n+s+1} \in \ell^2$ is a solution to
$$
\begin{cases}
Q_{s+1} \, u_j^{n+s+1} =0 \, ,& j_1 \ge 1 \, ,\\
\bpi_0 \, \begin{pmatrix}
u_{(p_1,\cdot)}^{n+s+1} \\
\vdots \\
u_{(1-r_1,\cdot)}^{n+s+1} \end{pmatrix} = 0 \, .& 
\end{cases}
$$
We apply a partial Fourier transform with respect to the tangential variables $j' \in \Z^{d-1}$, which yields, 
because $\bpi_0$ is a Fourier multiplier, see \eqref{defpin}:
$$
\begin{cases}
Q_{s+1}^\sharp (\bfeta) \, \widehat{u^{n+s+1}_{j_1}}(\bxi') =0 \, ,& j_1 \ge 1 \, ,\\
\Pi_0(\bfeta) \, \begin{pmatrix}
\widehat{u^{n+s+1}_{p_1}}(\bxi') \\
\vdots \\
\widehat{u^{n+s+1}_{1-r_1}}(\bxi') \end{pmatrix} = 0 \, ,&
\end{cases}
$$
where the `one-dimensional' finite difference operator $Q_{s+1}^\sharp$ is defined in \eqref{defqdiese}.

In the explicit case (case (i) in Assumption \ref{assumption3}), we compute $Q_{s+1}^\sharp (\bfeta)=I$, 
and we have seen in the proof of Theorem \ref{thm1} that the projector $\Pi_0(\bfeta)$ reduces to
$$
\Pi_0(\bfeta) =\begin{pmatrix}
0 & 0 \\
0 & I_{r_1} \end{pmatrix} \, .
$$
Hence we get $\widehat{u^{n+s+1}_{j_1}}(\bxi')=0$ for all $j_1 \ge 1-r_1$, and applying the inverse Fourier 
transform, we obtain $u^{n+s+1}=0$. Uniqueness follows.

In the implicit case (case (ii) in Assumption \ref{assumption3}), we know from the index condition \eqref{index} 
that a sequence $(w_{j_1})_{j_1 \ge 1-r_1}$ solution to the recurrence relation
\begin{equation}
\label{recurrencew}
\forall \, j_1 \ge 1 \, ,\quad Q_{s+1}^\sharp (\bfeta) \, w_{j_1} =0 \, ,
\end{equation}
belongs to $\ell^2$ if and only if its `initial condition' $(w_{p_1},\dots,w_{1-r_1})^T$ belongs to the stable 
subspace $\E^s(\infty,\bfeta)$ of the matrix $\M(\infty,\bfeta)$ whose expression is given in \eqref{limMinfty}. 
The matrix $\M(\infty,\bfeta)$ is of course the companion matrix that arises when one rewrites the recurrence 
\eqref{recurrencew} as a first order recurrence relation for the augmented vector $(w_{j_1+p_1-1},\dots, 
w_{j_1-r_1})^T$. Moreover, we know from the proof of Theorem \ref{thm1} that the projector $\Pi_0(\bfeta)$, 
which is the limit of $\Pi^s(z,\bfeta)$ as $z$ tends to infinity, is precisely the projection on $\E^s(\infty,\bfeta)$ 
with kernel $\E^u(\infty,\bfeta)$. Since the sequence $u^{n+s+1}$ belongs to $\ell^2([1-r_1,+\infty) \times 
\Z^{d-1})$, the partial Fourier transform $(\widehat{u^{n+s+1}_{j_1}}(\bxi'))_{j_1 \ge 1-r_1}$ belongs 
to $\ell^2$ for almost every $\bxi'$. We therefore get $(\widehat{u^{n+s+1}_{p_1}}(\bxi'), \dots, 
\widehat{u^{n+s+1}_{1-r_1}}(\bxi'))=0$ for almost every $\bxi'$, which yields $u^{n+s+1}=0$ by 
inverse Fourier transform. Uniqueness of a solution to \eqref{transparent0} follows by induction 
on $n$.
\bigskip

It remains to show that there exists a solution to \eqref{transparent0}, provided that the necessary 
compatibility conditions described in Proposition \ref{prop1} are satisfied. We thus assume that the 
source terms $g^n$ in \eqref{transparent0} satisfy \eqref{compatibilitegn} (for $n \ge s+1$, $g^n$ refers 
to the boundary forcing term in \eqref{transparent0}, and for $0 \le n \le s$, $g^n$ refers to the sequence 
constructed from the initial data as in the statement of Proposition \ref{prop1}). Let us assume that up to 
some time level $n+s$, with $n \ge 0$, we have constructed the sequences $u^0,\dots,u^{n+s}$, with
\begin{align*}
& (u^0,\dots,u^s) =(f^0,\dots,f^s) \, ,\quad \text{\rm (initial data),} \\
\forall \, n'=1,\dots,n \, ,\quad & \sum_{k=0}^{n'+s} \, \bpi_{n'+s-k} \, \begin{pmatrix}
u_{(p_1,\cdot)}^m \\
\vdots \\
u_{(1-r_1,\cdot)}^m \end{pmatrix} = g^{n'+s} \, ,\quad \text{\rm (boundary conditions),} \\
\forall \, n'=1,\dots,n \, , \, \forall \, j_1 \ge 1 \, ,\quad 
& \sum_{\sigma=0}^{s+1} \, Q_\sigma \, u_j^{n'+\sigma-1} =\Delta t \, F_j^{n'+s} \, ,\quad 
\text{\rm (numerical scheme).}
\end{align*}
We wish to construct a sequence $u^{n+s+1}$ solution to
\begin{equation}
\label{problemun}
\begin{cases}
Q_{s+1} \, u_j^{n+s+1} =\Delta t \, F_j^{n+s+1} -{\dps \sum_{\sigma=0}^s} \, Q_\sigma \, u_j^{n+\sigma} 
\, ,& j_1 \ge 1 \, ,\\
\bpi_0 \, \begin{pmatrix}
u_{(p_1,\cdot)}^{n+s+1} \\
\vdots \\
u_{(1-r_1,\cdot)}^{n+s+1} \end{pmatrix} = g^{n+s+1} -{\dps \sum_{m=0}^{n+s}} \, \bpi_{n+s+1-m} \, 
\begin{pmatrix}
u_{(p_1,\cdot)}^m \\
\vdots \\
u_{(1-r_1,\cdot)}^m \end{pmatrix} \, .& 
\end{cases}
\end{equation}
Let us first observe that, using Lemma \ref{lem1}, we already know that the sequence
$$
g^{n+s+1} -\sum_{m=0}^{n+s} \bpi_{n+s+1-m} \, \begin{pmatrix}
u_{(p_1,\cdot)}^m \\
\vdots \\
u_{(1-r_1,\cdot)}^m \end{pmatrix} \, ,
$$
belongs to the range of the projector $\bpi_0$ (one can for instance reproduce similar calculations as in 
the proof of Lemma \ref{lem1} and prove that this sequence belongs to the kernel of $\bpi_0-I$). We can 
therefore look for the solution $u^{n+s+1}$ to \eqref{problemun} under the form $u^{n+s+1} =v+w$, with
$$
\begin{pmatrix}
v_{(p_1,\cdot)} \\
\vdots \\
v_{(1-r_1,\cdot)} \end{pmatrix} := g^{n+s+1} -{\dps \sum_{m=0}^{n+s}} \, \bpi_{n+s+1-m} \, 
\begin{pmatrix}
u_{(p_1,\cdot)}^m \\
\vdots \\
u_{(1-r_1,\cdot)}^m \end{pmatrix} \, ,
$$
and for instance $v_{(j_1,\cdot)}:=0$ if $j_1>p_1$. We are then reduced to showing that, for some 
sequence $(\widetilde{F}_j) \in \ell^2$ whose expression is not useful, there exists a sequence $w 
\in \ell^2$ solution to
\begin{equation}
\label{problemun'}
\begin{cases}
Q_{s+1} \, w_j =\widetilde{F}_j \, ,& j_1 \ge 1 \, ,\\
\bpi_0 \, \begin{pmatrix}
w_{(p_1,\cdot)} \\
\vdots \\
w_{(1-r_1,\cdot)} \end{pmatrix} = 0 \, .& 
\end{cases}
\end{equation}
For simplicity, we drop the tilde on the source term in \eqref{problemun'}. We construct the partial Fourier 
transform of the solution $w$ to \eqref{problemun'} rather than $w$ itself. Namely, we are going to construct 
a solution to
\begin{equation}
\label{problemun''}
\begin{cases}
Q_{s+1}^\sharp(\bfeta) \, \widehat{w}_{j_1}(\bxi') =\widehat{F}_{j_1}(\bxi') \, ,& j_1 \ge 1 \, ,\\
\Pi_0(\bfeta) \, \begin{pmatrix}
\widehat{w}_{p_1}(\bxi') \\
\vdots \\
\widehat{w}_{1-r_1}(\bxi') \end{pmatrix} = 0 \, .& 
\end{cases}
\end{equation}

Let us first deal with the explicit case. In that case, one simply has to set
$$
\widehat{w}_{j_1}(\bxi') :=\begin{cases}
0 \, , & \text{\rm if } j_1 \le 0 \, ,\\
\widehat{F}_{j_1}(\bxi') \, , & \text{\rm if } j_1 \ge 1 \, ,
\end{cases}
$$
apply inverse Fourier transform and obtain a solution $w \in \ell^2$ to \eqref{problemun'}. We therefore 
focus on the implicit case. With the vector ${\mathcal W}_{j_1}(\bxi') :=(\widehat{w}_{j_1+p_1-1}(\bxi'), 
\dots, \widehat{w}_{j_1-r_1}(\bxi'))^T$, the system \eqref{problemun''} is equivalently rewritten as
\begin{equation}
\label{problemun'''}
\begin{cases}
{\mathcal W}_{j_1+1}(\bxi') -\M(\infty,\bfeta) \, {\mathcal W}_{j_1}(\bxi') =a_{\infty,p_1}(\bfeta)^{-1} \, 
(\widehat{F}_{j_1}(\bxi'),0,\dots,0)^T \, ,& j_1 \ge 1 \, ,\\
{\mathcal W}_1(\bxi') \in \E^u (\infty,\bfeta) \, .& 
\end{cases}
\end{equation}
The unique $\ell^2$ solution to this problem is written explicitly by using the spectral projectors $\Pi_0(\bfeta)$ 
and $I-\Pi_0(\bfeta)$. We obtain the formula
\begin{align*}
{\mathcal W}_{j_1}(\bxi') =& \, a_{\infty,p_1}(\bfeta)^{-1} \, \sum_{k=1}^{j_1-1} 
\Big( \M(\infty,\bfeta) \, \Pi_0(\bfeta) \Big)^{j_1-1-k} \, (\widehat{F}_{j_1}(\bxi'),0,\dots,0)^T \\
&-a_{\infty,p_1}(\bfeta)^{-1} \, \sum_{k \ge j_1} \Big( \M(\infty,\bfeta)^{-1} \, (I- \Pi_0(\bfeta)) \Big)^{k+1-j_1} 
\, (\widehat{F}_{j_1}(\bxi'),0,\dots,0)^T \, ,
\end{align*}
where the first line corresponds to the component of ${\mathcal W}_{j_1}$ on $\E^s(\infty,\bfeta)$ 
and the second line to the component on $\E^u(\infty,\bfeta)$. The latter formula defines a solution 
to \eqref{problemun'''} and the last thing to prove is that it belongs to $\ell^2$.

The matrix $\M(\infty,\bfeta)$ in \eqref{limMinfty} depends periodically and analytically on $\bfeta$. 
Moreover it has no eigenvalue on $\cercle$ for any $\bfeta$ so we have the bounds
$$
\| (\M(\infty,\bfeta) \, \Pi_0(\bfeta))^k \| \le C \, r^k \, ,\quad 
\| (\M(\infty,\bfeta)^{-1} \, (I-\Pi_0(\bfeta)))^k \| \le C \, r^k \, ,
$$
with $C>0$ and $r \in (0,1)$, uniformly with respect to $\bfeta \in \R^{d-1}$. These bounds express the 
exponential decay of the stable components of $\M$ and of the inverse of the unstable components. By 
standard $\ell^1 \star \ell^2$ convolution bounds, we obtain
$$
\sum_{j_1 \ge 1} |{\mathcal W}_{j_1}(\bxi')|^2 \le C \, \sum_{j_1 \ge 1} |\widehat{F}_{j_1}(\bxi')|^2 \, ,
$$
with a constant $C$ that is uniform with respect to the frequency $\bxi'$. This means that the inverse 
Fourier transform of (the appropriate coordinate of) $({\mathcal W}_{j_1}(\bxi'))$ provides with a 
sequence $w \in \ell^2$ that is a solution to \eqref{problemun'}. The proof of Proposition \ref{prop1} 
is thus complete.
\end{proof}

In the case $d=1$, we can also get a solvability result for the reduction of \eqref{cauchy} to an interval 
when one enforces transparent numerical boundary conditions at each end. Namely, we can reproduce 
the analysis of Section \ref{sect:2} when truncating the initial domain $\Z$ `on the right' rather than 
`on the left'. Introducing some given integer $J \ge p+r+2$, the solution to \eqref{cauchy} with initial data 
$f^0,\dots,f^s \in \ell^2(\Z)$ vanishing outside of the interval $[p+1,J-r-1]$ satisfies \eqref{relations} 
together with
$$
\forall \, n \in \N \, ,\quad \forall \, j \ge J \, ,\quad 
\sum_{m=0}^n \widetilde{\bpi}_{n-m} \, \begin{pmatrix}
u_{(j+p,\cdot)}^m \\
\vdots \\
u_{(j+1-r,\cdot)}^m \end{pmatrix} =0 \, ,\quad \text{\rm with } \quad 
\widetilde{\bpi}_n := \delta_{n0} \, I-\bpi_n \, .
$$
The expression of the matrices $\widetilde{\bpi}_n$ comes from the Laurent series expansion of $\Pi^u(z) 
=I-\Pi^s(s)$. In particular, the restriction of \eqref{cauchy} to the interval $[1-r,J+p]$ reads (recall that we 
consider the one-dimensional case here):
\begin{equation}
\label{interval}
\begin{cases}
{\dps \sum_{\sigma=0}^{s+1}} \, Q_\sigma \, u_j^{n+\sigma} =\Delta t \, F_j^{n+s+1} \, ,& 
n\ge 0 \, , \, j=1,\dots,J \, ,\\
{\dps \sum_{m=0}^{n+s+1}} \bpi_{n+s+1-m} \, \begin{pmatrix}
u_p^m \\
\vdots \\
u_{1-r}^m \end{pmatrix} = g_\ell^{n+s+1} \, , & n\ge 0 \, ,\\
{\dps \sum_{m=0}^{n+s+1}} \widetilde{\bpi}_{n+s+1-m} \, \begin{pmatrix}
u_{J+p}^m \\
\vdots \\
u_{J+1-r}^m \end{pmatrix} = g_r^{n+s+1} \, , & n\ge 0 \, ,\\
(u^0_j,\dots,u^s_j) = (f^0_j,\dots,f^s_j) \, , & j=1-r,\dots,J+p \, .
\end{cases}
\end{equation}
Up to compatibility conditions for the boundary source terms $g_\ell^n,g_r^n$ at the left and right ends 
of the interval, which will take the form \eqref{compatibilitegn} (or a similar one at the right end), the main 
issue for constructing the solution to \eqref{interval} is to prove existence for the linear problem:
$$
\begin{cases}
{\dps \sum_{\ell=-r}^p} a_{\ell,s+1} (\Delta t,\Delta x) \, u_{j+\ell} =F_j \, , & j=1,\dots,J \, ,\\
\bpi_0 \, \begin{pmatrix}
u_p \\
\vdots \\
u_{1-r} \end{pmatrix} = g_\ell \, , \quad (I-\bpi_0) \, \begin{pmatrix}
u_{J+p} \\
\vdots \\
u_{J+1-r} \end{pmatrix} = g_r \, . &
\end{cases}
$$
Since the problem is now {\it finite dimensional}, we are reduced to proving that the only solution 
to the homogeneous problem with $(F_j)_{j=1,\dots,J}=0$, $g_\ell=g_r=0$, is zero. The result 
is straightforward in the explicit case since the recurrence relation reduces to $u_j=0$ for 
$j=1,\dots,J$ and because of the expression of $\bpi_0$ and $I-\bpi_0$, see Theorem \ref{thm1}. 
We thus turn to the implicit case. With the notation of the proof of Theorem \ref{thm1}, the problem 
under study reads
$$
U_{j+1}=\M(\infty) \, U_j \, ,\quad j=1,\dots,J \, ,
$$
with the endpoint conditions
$$
\bpi_0 \, U_1 =0 \, ,\quad (I-\bpi_0) \, U_{J+1}=0 \, .
$$
In the one-dimensional implicit case, $\bpi_0$ is the projector on the stable subspace of $\M(\infty)$ 
parallel to the unstable subspace. Both subspaces form a direct sum of $\C^{p+r}$ and are invariant 
by $\M(\infty)$. Hence we have $U_1 \in \E^u(\infty)$, $U_{J+1} \in \E^s(\infty)$ and those conditions 
propagate to any $j=1,\dots,J+1$ by the recurrence relation $U_{j+1}=\M(\infty) \, U_j$. The only 
possibility for all these conditions to hold is to have $U_j=0$ for all $j$, which shows unique 
solvability for \eqref{interval} and explains why all the problems considered in 
\cite{BELV,BMN,ehrhardt-arnold,zheng-wen-han,zisowski-ehrhardt} are indeed solvable. These 
examples are quickly reviewed in Section \ref{sect:examples} below.

\section{Strong stability and semigroup estimate. Proof of Theorem \ref{thm2}}
\label{sect:4}

Our goal in this Section is to prove Theorem \ref{thm2}. In the first two paragraphs, we show that when 
the scheme \eqref{cauchy} is non-glancing, then \eqref{transparent0} is strongly stable and furthermore 
satisfies the estimate \eqref{estimation}. In the last paragraph, we show that the non-glancing condition 
is necessary for strong stability.

\subsection{Strong stability under the non-glancing condition}

Let us recall that from now on, we assume that each ratio $\Delta t/\Delta x_i$, $i=1,\dots,d$ is constant, 
and that each operator $Q_\sigma$ depends on the discretization parameters only through these ratios. 
In addition to Assumptions \ref{assumption0}, \ref{assumption1}, \ref{assumption2'} and \ref{assumption3}, 
we also assume in all this paragraph that the scheme \eqref{cauchy} is non-glancing. We can then follow 
some of the analysis in \cite{jfc} and use the following result from \cite{jfcsinum}.

\begin{theorem}[Block reduction of $\M$]
\label{thmjfc1}
Let Assumptions \ref{assumption0}, \ref{assumption1}, \ref{assumption2'} and \ref{assumption3} be satisfied, 
and assume that the scheme \eqref{cauchy} is non-glancing. Then for all $\underline{z} \in \Ubar$ and all 
$\underline{\bfeta} \in \R^{d-1}$, there exists an open neighborhood ${\mathcal O}$ of $(\underline{z}, 
\underline{\bfeta})$ in $\C \times \R^{d-1}$ and there exists an invertible matrix $T(z,\bfeta)$ that is 
holomorphic/analytic with respect to $(z,\bfeta) \in {\mathcal O}$ such that:
\begin{equation*}
\forall \, (z,\bfeta) \in {\mathcal O} \, ,\quad T(z,\bfeta)^{-1}\, \M(z,\bfeta) \, T(z,\bfeta) =\begin{pmatrix}
\M_1(z,\bfeta) & & 0 \\
& \ddots & \\
0 & & \M_L(z,\bfeta) \end{pmatrix} \, ,
\end{equation*}
where the number $L$ of diagonal blocks and the size $\nu_\ell$ of each block $\M_\ell$ do not depend 
on $(z,\bfeta) \in {\mathcal O}$, and where each block satisfies one of the following three properties:
   \begin{itemize}
      \item  there exists $\delta>0$ such that for all $(z,\bfeta) \in {\mathcal O}$, $\M_\ell(z,\bfeta)^* \, 
             \M_\ell(z,\bfeta) \ge (1+\delta) \, I$,
      \item  there exists $\delta>0$ such that for all $(z,\bfeta) \in {\mathcal O}$, $\M_\ell(z,\bfeta)^* \, 
             \M_\ell(z,\bfeta) \le (1-\delta) \, I$,
      \item  $\nu_\ell=1$, $\underline{z}$ and $\M_\ell(\underline{z},\underline{\bfeta})$ belong to $\cercle$, 
             and $\underline{z} \, \partial_z \M_\ell (\underline{z},\underline{\bfeta}) \, \overline{\M_\ell 
             (\underline{z},\underline{\bfeta})} \in \R \setminus \{ 0 \}$.
   \end{itemize}
\end{theorem}

\noindent A Corollary of Theorem \ref{thmjfc1} is a (unique) continuation result for the stable and unstable 
subspaces of $\M(z,\bfeta)$.

\begin{corollary}
\label{cor1}
Let Assumptions \ref{assumption0}, \ref{assumption1}, \ref{assumption2'} and \ref{assumption3} be satisfied, 
and assume that the scheme \eqref{cauchy} is non-glancing. Then there exists $\delta>0$ such that the 
stable and unstable subspaces $\E^s,\E^u$ of $\M$ extend holomorphically/analaytically/periodically over 
$\{ \zeta \in \C \, , \, |\zeta|>1-2\, \delta \} \times \R^{d-1}$. In particular, there holds
$$
\sum_{n \ge 0} \dfrac{1}{(1-\delta)^n} \, \sup_{\bfeta \in \R^{d-1}} \| \Pi_n(\bfeta) \| <+\infty \, .
$$
Moreover, for all $z \in \C$ with $|z|>1-2\, \delta$ and $\bfeta \in \R^{d-1}$, the direct sum 
\eqref{decomposition1} holds. 
\end{corollary}

Since the direct sum \eqref{decomposition1} holds up to $|z|=1$ (and even a little beyond the unit circle), 
we can follow the theory in \cite{gks}, see also \cite{jfcnotes} for a thorough exposition, and try to prove 
strong stability for \eqref{transparent0} by verifying the so-called Uniform Kreiss-Lopatinskii condition (the 
main result in \cite{gks} is to show that this algebraic condition is actually {\it equivalent} to strong stability). 
Verifying the Uniform Kreiss-Lopatinskii condition amounts to first performing a Laplace-Fourier transform 
in the time and tangential space variables, which reduces \eqref{transparent0} to the recurrence relation
\begin{equation}
\label{resolvent}
\begin{cases}
{\dps \sum_{\ell_1=-r_1}^{p_1}} a_{\ell_1}(z,\bfeta) \, w_{j_1+\ell_1} =F_{j_1} \, , & j_1 \ge 1 \, ,\\
\Pi^s (z,\bfeta) \, \begin{pmatrix}
w_{p_1} \\
\vdots \\
w_{1-r_1} \end{pmatrix} =G \, . &
\end{cases}
\end{equation}
Then the goal is to show that when $|z| \ge 1$, there is no non-trivial solution to the homogeneous equation 
\eqref{resolvent} (obtained with $(F_{j_1})=0$, $G=0$). The solutions of interest are those that belong to 
$\ell^2$ when $z$ belongs to $\U$, and those whose initial data $(w_{p_1},\dots,w_{1-r_1})^T$ are obtained 
by a continuation argument from $\U$ to its boundary $\cercle$ when $z \in \cercle$. For $z \in \U$, we can 
use Lemma \ref{lem2} and parametrize the set of $\ell^2$ solutions of the recurrence relation
$$
\sum_{\ell_1=-r_1}^{p_1} a_{\ell_1}(z,\bfeta) \, w_{j_1+\ell_1} =0 \, ,\quad j_1 \ge 1 \, ,
$$
by the stable subspace $\E^s(z,\bfeta)$ of $\M(z,\bfeta)$. Among all such vectors, it is clear that the only 
one that satisfies the homogeneous numerical boundary condition
\begin{equation}
\label{homcond}
\Pi^s (z,\bfeta) \, \begin{pmatrix}
w_{p_1} \\
\vdots \\
w_{1-r_1} \end{pmatrix} = 0 \, ,
\end{equation}
is the zero vector. In other words, we have just proved that the system \eqref{resolvent} has no non-zero 
solution when the source terms vanish and $z \in \U$. Hence the so-called Godunov-Ryabenkii condition 
holds for \eqref{transparent0} (non-existence of unstable eigenvalues). Proving that the Uniform 
Kreiss-Lopatinskii condition holds amounts to showing the same `injectivity' property up to $z \in \cercle$. 
As a first step, let us observe that Corollary \ref{cor1} shows that the spectral projector $\Pi^s$ also extends 
holomorphically/analytically/periodically on a neighborhood of $\Ubar \times \R^{d-1}$. It is therefore legitimate 
to consider the resolvent equation \eqref{resolvent} for $z \in \cercle$. In that case, the only vector in the 
extended stable subspace\footnote{Recall that for $z \in \cercle$, initial data in $\E^s(z,\bfeta)$ do not 
necessarily correspond to $\ell^2$ solution of the recurrence relation but can be viewed as all the possible 
limits of such $\ell^2$ solutions.} $\E^s(z,\bfeta)$ that satisfies the homogeneous numerical boundary condition 
\eqref{homcond} is the zero vector. In other words, we have verified that the Uniform Kreiss-Lopatinskii 
condition holds. Applying the main result of \cite{gks} (more precisely, see \cite{jfcnotes} for the extension 
of the theory in \cite{gks} to the general case that we consider here), the scheme \eqref{transparent0} is 
strongly stable.
\bigskip

The above argument may look somehow trivial, but the subtle point is that in the theory of \cite{gks}, one 
assumes that the numerical boundary conditions for the resolvent equation are `well-defined' for $z \in 
\Ubar$ and the difficult part of the job is to extend the stable subspace up to the boundary $\cercle$ of 
$\U$. Here it is not even obvious that the numerical boundary conditions in \eqref{resolvent} are well-defined 
for $z \in \Ubar$. As a matter of fact, the main result in \cite{jfcsinum} shows that the spectral projector 
$\Pi^s$ extends (even continuously) up to $\cercle$ {\it if and only if} the scheme \eqref{cauchy} is 
non-glancing. When glancing (numerical) wave packets occur, the spectral projector $\Pi^s$ has a 
singularity at some point of $\cercle$. This singular behavior will be one obstacle we shall have to 
circumvent in the last paragraph of this Section.

\subsection{Semigroup estimate}

In this paragraph, we follow the analysis in \cite{jfcag,jfcX} and prove the validity of the stability estimate 
\eqref{estimation} when one considers non-zero initial data in \eqref{transparent0}. We therefore assume 
that the scheme \eqref{cauchy} is non-glancing and that for all $\bxi \in \R^d$, the roots to \eqref{dispersion} 
are simple (the latter condition being automatic for $s=0$). Under such assumptions, we can apply the 
following result from \cite{jfcX}:

\begin{theorem}[Existence of dissipative boundary conditions]
\label{thmjfcX}
Let Assumptions \ref{assumption0}, \ref{assumption1} and \ref{assumption2'} be satisfied. Assume furthermore 
that for all $\bxi \in \R^d$, the $s+1$ roots to \eqref{dispersion} are simple. Then there exists a constant $C>0$ 
such that for any given initial data $f^0,\dots,f^s \in \ell^2$, there exists a  sequence 
$(v_j^n)_{j_1 \ge 1-r_1,n\in \N}$ that satisfies
$$
\begin{cases}
{\dps \sum_{\sigma=0}^{s+1}} \, Q_\sigma \, v_j^{n+\sigma} =0 \, ,& 
n\ge 0 \, ,\quad j_1 \ge 1 \, ,\\
(v_j^0,\dots,v_j^s) = (f_j^0,\dots,f_j^s) \, , & j_1 \ge 1-r_1 \, ,
\end{cases}
$$
and
\begin{equation*}
\sup_{n \in \N} \, \Ng v^n \Nd_{1-r_1,+\infty}^2 
+\sum_{n\ge s+1} \, \sum_{j_1=1-r_1}^{p_1} \Delta t \, \| v_{(j_1,\cdot)}^n \|_{\ell^2(\Z^{d-1})}^2 
\le C \, \sum_{\sigma=0}^s \Ng f^\sigma \Nd_{1-r_1,+\infty}^2 \, .
\end{equation*}
\end{theorem}

\noindent With the help of Theorem \ref{thmjfcX}, we go back to the numerical scheme \eqref{transparent0}. 
We tacitly assume again that the source terms in \eqref{transparent0} satisfy the necessary compatibility 
conditions for a solution to exist. We then decompose the solution to \eqref{transparent0} as $u_j^n =v_j^n 
+w_j^n$, where the sequence $(v_j^n)$ is given by Theorem \ref{thmjfcX}, and the remaining part $(w_j^n)$ 
satisfies
\begin{equation}
\label{remainder}
\begin{cases}
{\dps \sum_{\sigma=0}^{s+1}} \, Q_\sigma \, w_j^{n+\sigma} =\Delta t \, F_j^{n+s+1} \, ,& 
n\ge 0 \, ,\quad j_1 \ge 1 \, ,\\
{\dps \sum_{m=0}^{n+s+1}} \, \bpi_{n+s+1-m} \, \begin{pmatrix}
w_{(p_1,\cdot)}^m \\
\vdots \\
w_{(1-r_1,\cdot)}^m \end{pmatrix} = g^{n+s+1} -{\dps \sum_{m=0}^{n+s+1}} \, \bpi_{n+s+1-m} \, \begin{pmatrix}
v_{(p_1,\cdot)}^m \\
\vdots \\
v_{(1-r_1,\cdot)}^m \end{pmatrix}\, , & n\ge 0 \, ,\\
(w_j^0,\dots,w_j^s) = (0,\dots,0) \, , & j_1 \ge 1-r_1 \, .
\end{cases}
\end{equation}
The main point to keep in mind is that we have reduced to the case of vanishing initial data for $(w_j^n)$. 
Since the scheme \eqref{cauchy} is non-glancing, we have already seen that \eqref{transparent0} is strongly 
stable and therefore satisfies the estimate \eqref{stabilitenumibvp} (with the interior and boundary source terms 
as given in \eqref{remainder}). Since we need now to estimate these source terms, we define
$$
\forall \, n \ge 0 \, ,\quad 
\widetilde{g}^{n+s+1} := g^{n+s+1} -{\dps \sum_{m=0}^{n+s+1}} \, \bpi_{n+s+1-m} \, \begin{pmatrix}
v_{(p_1,\cdot)}^m \\
\vdots \\
v_{(1-r_1,\cdot)}^m \end{pmatrix} \, .
$$
The bound given in Corollary \ref{cor1} for the matrices $\Pi_n(\bfeta)$ implies that the Fourier multipliers 
$(\bpi_n)$ satisfy
$$
\sum_{n \ge 0} \| \bpi_n \|_{{\mathcal B}(\ell^2(\Z^{d-1}))} < +\infty \, .
$$
Actually, the decay is even exponential with respect to $n$, but we shall only make use of the fact that 
the norms of these operators belong to $\ell^1$. We use the above definition of the source term 
$\widetilde{g}^{n+s+1}$ and derive the estimates (here $\gamma>0$ is a parameter and the constants 
$C$ below are independent of $\gamma$):
\begin{align*}
\sum_{n\ge s+1} \, \Delta t \, {\rm e}^{-2\, \gamma \, n\, \Delta t} \, \| \widetilde{g}^n \|_{\ell^2(\Z^{d-1})}^2 &\le 
C \, \left\{ 
\sum_{n\ge s+1} \, \Delta t \, {\rm e}^{-2\, \gamma \, n\, \Delta t} \, \| g^n \|_{\ell^2(\Z^{d-1})}^2 \right. \\
& \quad \quad \left. +\sum_{n\ge s+1} \, \Delta t \, {\rm e}^{-2\, \gamma \, n\, \Delta t} \, \left\| \sum_{m=0}^n 
\bpi_{n-m} \, \begin{pmatrix}
v_{(p_1,\cdot)}^m \\
\vdots \\
v_{(1-r_1,\cdot)}^m \end{pmatrix} \right\|_{\ell^2(\Z^{d-1})}^2 \right\} \\
&\le C \, \left\{ 
\sum_{n\ge s+1} \, \Delta t \, {\rm e}^{-2\, \gamma \, n\, \Delta t} \, \| g^n \|_{\ell^2(\Z^{d-1})}^2 \right. \\
& \quad \quad \left. 
+\sum_{n\ge 0} \, \Delta t \, {\rm e}^{-2\, \gamma \, n\, \Delta t} \, \sum_{j_1=1-r_1}^{p_1} 
\| v_{(j_1,\cdot)}^n \|_{\ell^2(\Z^{d-1})}^2 \right\} \\
&\le C \, \left\{ \sum_{\sigma=0}^s \Ng f^\sigma \Nd_{1-r_1,+\infty}^2 
+\sum_{n\ge s+1} \, \Delta t \, {\rm e}^{-2\, \gamma \, n\, \Delta t} \, \| g^n \|_{\ell^2(\Z^{d-1})}^2 \right\} \, ,
\end{align*}
where we have first used the standard $\ell^1 \star \ell^2$ convolution estimate and then the bound 
provided by Theorem \ref{thmjfcX} for the trace of $(v_j^n)$.

Since we have an estimate of the source terms in \eqref{remainder}, we can use the strong stability of 
\eqref{transparent0} and obtain the following estimate for the solution $(w_j^n)$ to \eqref{remainder}:
\begin{multline}
\label{estimation'}
\dfrac{\gamma}{\gamma \, \Delta t+1} \, \sum_{n\ge 0} \Delta t \, {\rm e}^{-2\, \gamma \, n\, \Delta t} \, 
\Ng w^n \Nd_{1-r_1,+\infty}^2 
+\sum_{n\ge s+1} \, \sum_{j_1=1-r_1}^{p_1} \Delta t \, {\rm e}^{-2\, \gamma \, n\, \Delta t} \, 
\| w_{(j_1,\cdot)}^n \|_{\ell^2(\Z^{d-1})}^2 \\
\le C \left\{ \sum_{\sigma=0}^s \Ng f^\sigma \Nd_{1-r_1,+\infty}^2 
+\dfrac{\gamma \, \Delta t+1}{\gamma} \, \sum_{n \ge s+1} \Delta t \, 
{\rm e}^{-2\, \gamma \, n\, \Delta t} \, \Ng F^n \Nd_{1,+\infty}^2 
+\sum_{n\ge s+1} \, \Delta t \, {\rm e}^{-2\, \gamma \, n\, \Delta t} \, \| g^n \|_{\ell^2(\Z^{d-1})}^2 \right\} \, .
\end{multline}
The goal now is to combine \eqref{estimation'} with the estimate of Theorem \ref{thmjfcX} in order to 
derive the estimate \eqref{estimation} of Theorem \ref{thm2} we are aiming at. Both \eqref{estimation'} 
and the estimate of Theorem \ref{thmjfcX} provide with an estimate for the traces of $(v_j^n)$ and 
$(w_j^n)$ that is sufficient for deriving the estimate of the trace of $(u_j^n)$ in \eqref{estimation}. 
Unfortunately, this is not over yet since on the left hand side of \eqref{estimation'}, we only control 
the norm
$$
\dfrac{\gamma}{\gamma \, \Delta t+1} \, \sum_{n\ge 0} \Delta t \, {\rm e}^{-2\, \gamma \, n\, \Delta t} \, 
\Ng w^n \Nd_{1-r_1,+\infty}^2 \, ,
$$
and not the (stronger) semigroup norm 
$$
\sup_{n \in \N} \, {\rm e}^{-2\, \gamma \, n\, \Delta t} \, \Ng w^n \Nd_{1-r_1,+\infty}^2 \, .
$$
However, at this stage, the {\it exact same} argument as in \cite[Paragraph 3.1]{jfcX} using the multiplier 
technique developed in that article provides with the semigroup estimate of $(w_j^n)$. This part of the 
argument in \cite{jfcX} is not restricted to the `local' numerical boundary conditions considered in that 
paper but applies as long as one already controls the trace of the solution to \eqref{remainder} (which is 
provided here, as in \cite{jfcX}, by the strong stability of \eqref{transparent0}). Hence we can improve the 
estimate \eqref{estimation'} into
\begin{multline}
\label{estimation''}
\sup_{n \in \N} \, {\rm e}^{-2\, \gamma \, n\, \Delta t} \, \Ng w^n \Nd_{1-r_1,+\infty}^2 
+\sum_{n\ge s+1} \, \sum_{j_1=1-r_1}^{p_1} \Delta t \, {\rm e}^{-2\, \gamma \, n\, \Delta t} \, 
\| w_{(j_1,\cdot)}^n \|_{\ell^2(\Z^{d-1})}^2 \\
\le C \left\{ \sum_{\sigma=0}^s \Ng f^\sigma \Nd_{1-r_1,+\infty}^2 
+\dfrac{\gamma \, \Delta t+1}{\gamma} \, \sum_{n \ge s+1} \Delta t \, 
{\rm e}^{-2\, \gamma \, n\, \Delta t} \, \Ng F^n \Nd_{1,+\infty}^2 
+\sum_{n\ge s+1} \, \Delta t \, {\rm e}^{-2\, \gamma \, n\, \Delta t} \, \| g^n \|_{\ell^2(\Z^{d-1})}^2 \right\} \, ,
\end{multline}
and combining \eqref{estimation''} with the estimate for $(v_j^n)$ provided by Theorem \ref{thmjfcX}, 
we complete the proof of \eqref{estimation}.

\subsection{Necessity of the non-glancing condition}

Our goal in this Paragraph is to show the last part of Theorem \ref{thm2}, meaning that the non-glancing 
condition is necessary for strong stability of \eqref{transparent0}. We therefore assume from now on that 
the scheme \eqref{cauchy} is glancing and that strong stability holds for \eqref{transparent0}. Our goal 
will be to obtain a contradiction.

By the analysis of \cite{gks}, strong stability of \eqref{transparent0} is equivalent to the fulfillment of a 
uniform stability estimate for the solution to the resolvent equation \eqref{resolvent}. More precisely, 
since we have assumed that \eqref{transparent0} is strongly stable, then there exists a constant $C>0$ 
such that for all $z \in \U$ and all $\bfeta \in \R^{d-1}$, for all $(F_{j_1})_{j_1 \ge 1} \in \ell^2$ and for all 
$G \in \E^s (z,\bfeta)$, the resolvent equation \eqref{resolvent} has a unique solution $(w_{j_1})_{j_1 \ge 1-r_1} 
\in \ell^2$ and that solution satisfies
$$
\dfrac{|z|-1}{|z|} \, \sum_{j_1 \ge 1-r_1} |w_{j_1}|^2 +\sum_{j_1=1-r_1}^{p_1} |w_{j_1}|^2 
\le C \, \left\{ \dfrac{|z|}{|z|-1} \, \sum_{j_1 \ge 1} |F_{j_1}|^2 +|G|^2 \right\} \, .
$$
Let us be a little more specific and consider the resolvent equation \eqref{resolvent} in the particular 
case $G=0$. Then using the companion matrix $\M$ in \eqref{defM} to rewrite the recurrence relation 
in \eqref{resolvent}, we can find an expression for the solution to \eqref{resolvent}. In particular, there 
holds
\begin{equation}
\label{solresolvent'}
{\mathcal W}_1 := \begin{pmatrix}
w_{p_1} \\
\vdots \\
w_{1-r_1} \end{pmatrix} =-\sum_{k \ge 1} \M(z,\bfeta)^{-k} \, \Pi^u(z,\bfeta) \, \begin{pmatrix}
F_k/a_{p_1}(z,\bfeta) \\
0 \\
\vdots \\
0 \end{pmatrix} \, .
\end{equation}
Thanks to our strong stability assumption, we know that uniformly with respect to $(z,\bfeta) \in \U \times 
\R^{d-1}$, the vector ${\mathcal W}_1$ defined in \eqref{solresolvent'} satisfies
\begin{equation}
\label{estimW1'}
|{\mathcal W}_1|^2 \le C \, \dfrac{|z|}{|z|-1} \, \sum_{j_1 \ge 1} |F_{j_1}|^2 \, .
\end{equation}
The inconvenient feature of \eqref{solresolvent'} is that the source term, meaning the vector to which 
we apply the matrix $\M(z,\bfeta)^{-k} \, \Pi^u(z,\bfeta)$ on the right hand side, should be proportional 
to the first vector of the canonical basis. However, one can argue as in \cite[Proposition 4]{jfcnotes} and 
show that the same estimate as \eqref{estimW1'} holds for arbitrary source terms in $\C^{p_1+r_1}$. 
More precisely, if strong stability holds and under the assumptions of Theorem \ref{thm2}, there exists 
a constant $C>0$ such that for all $z \in \U$ with $|z| \le 2$, for all $\bfeta \in \R^{d-1}$ and for all 
$({\mathcal F}_{j_1})_{j_1 \ge 1} \in \ell^2$, the vector
\begin{equation}
\label{solresolvent}
{\mathcal W}_1 := -\sum_{k \ge 1} \M(z,\bfeta)^{-k} \, \Pi^u(z,\bfeta) \, {\mathcal F}_k \, ,
\end{equation}
satisfies the estimate
\begin{equation}
\label{estimW1}
|{\mathcal W}_1|^2 \le C \, \dfrac{|z|}{|z|-1} \, \sum_{j_1 \ge 1} |{\mathcal F}_{j_1}|^2 \, .
\end{equation}

Our goal now is to show that, if \eqref{cauchy} is glancing, then the estimate \eqref{estimW1} breaks 
down for a convenient choice of the frequency $z$ (that should be sufficiently close to $\U$) and of 
the source term $({\mathcal F}_{j_1})_{j_1 \ge 1} \in \ell^2$. Let us therefore recall the following result 
from \cite{jfcsinum}, which will be the starting point for our construction of the source term 
$({\mathcal F}_{j_1})_{j_1 \ge 1} \in \ell^2$ in \eqref{solresolvent}.

\begin{theorem}[Block reduction of $\M$]
\label{thmg}
Let Assumptions \ref{assumption0}, \ref{assumption1}, \ref{assumption2'} and \ref{assumption3} be satisfied, 
and assume that the scheme \eqref{cauchy} is glancing. Then there exists $\underline{z} \in \Ubar$ and 
$\underline{\bfeta} \in \R^{d-1}$, there exists an open neighborhood ${\mathcal O}$ of $(\underline{z}, 
\underline{\bfeta})$ in $\C \times \R^{d-1}$ and there exists an invertible matrix $T(z,\bfeta)$ that is 
holomorphic/analytic with respect to $(z,\bfeta) \in {\mathcal O}$ such that:
\begin{equation*}
\forall \, (z,\bfeta) \in {\mathcal O} \, ,\quad T(z,\bfeta)^{-1}\, \M(z,\bfeta) \, T(z,\bfeta) =\begin{pmatrix}
\M_g(z,\bfeta) & 0 \\
0 & \M_\sharp(z,\bfeta) \end{pmatrix} \, ,
\end{equation*}
where the diagonal block $\M_g$ has size $m \times m$, $m \ge 2$, and it satisfies
\begin{equation*}
\M_g(\underline{z},\underline{\bfeta})=\underline{\kappa} \, \begin{pmatrix}
1 & 1 & 0 & 0 \\
0 & \ddots & \ddots & 0 \\
\vdots & \ddots & \ddots & 1 \\
0 & \dots & 0 & 1 \end{pmatrix} \, ,\quad \underline{\kappa} \in \cercle \, .
\end{equation*}
Moreover the lower left coefficient $\underline{\nu}$ of $\partial_z \M_g(\underline{z},\underline{\bfeta})$ 
is such that for all $\theta \in \C$ with $\text{\rm Re } \theta>0$, and for all complex number $\zeta$ such that 
$\zeta^m =\overline{\underline{\kappa}} \, \underline{\nu} \, \underline{z} \, \theta$, then $\text{\rm Re } \zeta 
\neq 0$.
\end{theorem}

\noindent In the terminology of \cite{jfcsinum}, the block $\M_g$ is of the fourth type (the subscript $g$ here 
refers to `glancing'). In all what follows, we keep the tangential frequency $\bfeta$ in \eqref{solresolvent} 
fixed and equal to that $\underline{\bfeta} \in \R^{d-1}$ given in Theorem \ref{thmg}. We shall also choose 
$z=z_\eps :=(1+\eps)\, \underline{z}$, with $\eps>0$ small enough so that $(z_\eps,\underline{\bfeta})$ 
belongs to the neighborhood ${\mathcal O}$ of $(\underline{z},\underline{\bfeta})$ given by Theorem 
\ref{thmg}. Since $\bfeta$ is fixed, we forget to recall the $\bfeta$-dependence of all quantities from now on.

The $m$ first column vectors of the matrix $T(z)$ in Theorem \ref{thmg} are denoted $T_1(z),\dots,T_m(z)$. 
They satisfy
$$
\M(z) \, \begin{pmatrix}
T_1(z) & \cdots & T_m(z) \end{pmatrix} =\begin{pmatrix}
T_1(z) & \cdots & T_m(z) \end{pmatrix} \, \M_g(z) \, .
$$
At $z=\underline{z}$, $\M_g(\underline{z})$ has the (only) eigenvalue $\underline{\kappa} \in \cercle$ with 
algebraic multiplicity $m$. The geometric multiplicity is $1$. For $z \in \U$ close to $\underline{z}$, we know 
from Lemma \ref{lem2} that $\M_g(z)$ has no eigenvalue on $\cercle$ for eigenvalues of $\M_g(z)$ are also 
eigenvalues of $\M(z)$. Therefore the number $\mu$ of stable eigenvalues of $\M_g(z)$ when $z \in \U$ is 
close to $\underline{z}$ is constant. This number is denoted $\mu$ from now on. Its value is given in 
\cite[Proposition 4.1]{jfcpise}:
$$
\mu =\begin{cases}
m/2 \, ,& \text{if $m$ is even,} \\
(m \pm 1)/2  \, ,& \text{if $m$ is odd.}
\end{cases}
$$
The choice between $\pm 1$ when $m$ is odd depends on the lower left coefficient $\underline{\nu}$ 
of $\partial_z \M_g (\underline{z})$. The eigenvalues and eigenvectors of $\M_g(z)$ have a Puiseux 
expansion close to $\underline{z}$, see \cite{baumgartel}. Such expansions are computed as in 
\cite[Proposition 4.1]{jfcpise} (see similar arguments for the continuous problem in \cite{kreiss2,sarason1}). 
The expansions read
\begin{align*}
\kappa (z) &= \underline{\kappa} \, \big( 1+\zeta_1 \, w^{1/m} +\zeta_2 \, w^{2/m} +\cdots +\zeta_k \, w^{k/m} 
+\cdots \big) \, ,\\
r(z) &={\bf r}_0 +{\bf r}_1 \, w^{1/m} +{\bf r}_2 \, w^{2/m} +\cdots +{\bf r}_k \, w^{k/m} +\cdots \, ,
\end{align*}
with $w:=(z-\underline{z})/\underline{z}$, and the vectors ${\bf r}_0,\dots,{\bf r}_{m-1}$ form a basis of 
$\C^m$. Moreover the first coefficient $\zeta_1$ in the expansion of $\kappa(z)$ is nonzero and can be 
chosen to be real if $m$ is odd. Independently of $m$, the number $\zeta_1$ is such that for any $m$-th 
root of unity $\omega$, the real part of $\zeta_1 \, \omega$ is nonzero.

Let us label the $m$-th roots of unity as $\omega_1,\dots,\omega_m$ and specify $z=z_\eps := (1+\eps) \, 
\underline{z}$, $\eps>0$ small enough. Then the eigenvalues $\kappa_\ell(\eps)$, $\ell=1,\dots,m$, of 
$\M_g(z_\eps)$ have the expansions
$$
\kappa_\ell(\eps) = \underline{\kappa} \, \big( 1+\zeta_1 \, \omega_\ell \, \eps^{1/m} \big) +O(\eps^{2/m}) \, ,
$$
and the associated eigenvectors read
$$
r_\ell(\eps) =\sum_{k=0}^{m-1} {\bf r}_k \, \omega_\ell^k \, \eps^{k/m} +O(\eps) \, .
$$
The $m$-th roots of unity are labeled in such a way that $\kappa_1(\eps),\dots,\kappa_\mu(\eps)$ are the 
stable eigenvalues of $\M_g(z_\eps)$, and $\kappa_{\mu+1}(\eps),\dots,\kappa_m(\eps)$ are the unstable 
eigenvalues. To each eigenvector $r_\ell(\eps)$ for $\M_g(z_\eps)$, there corresponds an eigenvector
$$
{\mathcal T}_\ell(\eps) := \begin{pmatrix}
T_1(z_\eps) & \cdots & T_m(z_\eps) \end{pmatrix} \, r_\ell(\eps) \, ,
$$
for the matrix $\M(z_\eps)$, with the same eigenvalue $\kappa_\ell(\eps)$. In particular, ${\mathcal T}_1, 
\dots,{\mathcal T}_\mu$ are stable eigenvectors and ${\mathcal T}_{\mu+1},\dots,{\mathcal T}_m$ are 
unstable eigenvectors.

The goal is now to choose a source term $({\mathcal F}_{j_1})$ of size $\sim 1$ in $\ell^2$ but such that 
the projection on the unstable subspace of $\M(z_\eps)$ is large and proportional to some given unstable 
eigenvector. Namely, we first define
\begin{equation}
\label{defFeps}
{\mathcal F}(\eps) := \sum_{\ell=1}^{\mu+1} \alpha_\ell(\eps) \, {\mathcal T}_\ell(\eps) 
=\begin{pmatrix}
T_1(z_\eps) & \cdots & T_m(z_\eps) \end{pmatrix} \, \sum_{\ell=1}^{\mu+1} \alpha_\ell(\eps) \, r_\ell(\eps) 
\, ,
\end{equation}
where the coefficients $\alpha_1(\eps),\dots,\alpha_{\mu+1}(\eps)$ are chosen such that
\begin{equation}
\label{defalphaell}
\begin{pmatrix}
1 & \cdots & 1 \\
\vdots & & \vdots \\
\omega_1^\mu & \cdots & \omega_{\mu+1}^\mu \end{pmatrix} \, \begin{pmatrix}
\alpha_1(\eps) \\
\vdots \\
\alpha_{\mu+1}(\eps) \end{pmatrix} =\begin{pmatrix}
0 \\
\vdots \\
0 \\
\eps^{-\mu/m} \end{pmatrix} \, .
\end{equation}
With this definition of the coefficients $\alpha_\ell(\eps)$, and using the above Puiseux expansions 
of the eigenvectors $r_\ell(\eps)$, we have
$$
{\mathcal F}(\eps) =\begin{pmatrix}
T_1(\underline{z}) & \cdots & T_m(\underline{z}) \end{pmatrix} \, {\bf r}_\mu +o(1) \, ,
$$
so, for some positive constant $c>0$, we have
$$
|{\mathcal F}(\eps)| =c+o(1) \, ,
$$
as $\eps$ tends to zero. With $z=z_\eps$, we choose in \eqref{solresolvent} the source term
$$
\forall \, j_1 \ge 1 \, ,\quad 
{\mathcal F}_{j_1} := \dfrac{\big( |\kappa_{\mu+1}(\eps)|^2 -1 \big)^{1/2}}{\overline{\kappa_{\mu+1}(\eps)^{j_1}}} 
\, {\mathcal F}(\eps) \, ,
$$
with the vector ${\mathcal F}(\eps)$ given in \eqref{defFeps}. This choice gives
$$
\sum_{j_1 \ge 1} |{\mathcal F}_{j_1}|^2 =|{\mathcal F}(\eps)|^2 =c+o(1) \, ,
$$
and the corresponding vector ${\mathcal W}_1$ in \eqref{solresolvent} reads
\begin{align*}
{\mathcal W}_1 &= -\alpha_{\mu+1}(\eps) \, \sum_{k \ge 1} 
\dfrac{\big( |\kappa_{\mu+1}(\eps)|^2 -1 \big)^{1/2}}{\overline{\kappa_{\mu+1}(\eps)^k}} \, \M(z_\eps)^{-k} \, 
{\mathcal T}_{\mu+1}(\eps) \\
&= -\alpha_{\mu+1}(\eps) \, \sum_{k \ge 1} 
\dfrac{\big( |\kappa_{\mu+1}(\eps)|^2 -1 \big)^{1/2}}{|\kappa_{\mu+1}(\eps)|^{2\, k}} \, {\mathcal T}_{\mu+1}(\eps) \\
&= -\dfrac{\alpha_{\mu+1}(\eps)}{\big( |\kappa_{\mu+1}(\eps)|^2 -1 \big)^{1/2}} \, {\mathcal T}_{\mu+1}(\eps) \, .
\end{align*}
The bound \eqref{estimW1} then gives
$$
\dfrac{|\alpha_{\mu+1}(\eps)|^2}{|\kappa_{\mu+1}(\eps)|^2 -1} \, |{\mathcal T}_{\mu+1}(\eps)|^2 \le 
\dfrac{C}{\eps} \, ,
$$
with $C>0$ uniform with respect to $\eps$. The conclusion follows from the asymptotics of the quantities 
on the left hand side of this last inequality. Namely, from our construction of the eigenvectors ${\mathcal T}_\ell$ 
and the Puiseux expansion of the eigenvalues (recall that the real part of $\zeta_1 \, \omega_{\mu+1}$ is 
non-zero and therefore positive since $\kappa_{\mu+1}(\eps)$ is an unstable eigenvalue), we have
$$
|\kappa_{\mu+1}(\eps)|^2 -1 \sim c \, \eps^{1/m} \, ,\quad |{\mathcal T}_{\mu+1}(\eps)|^2 \sim c \, ,
$$
with $c$ a positive constant that does not depend on $\eps$. Eventually, the inverse of the Vandermonde 
matrix in \eqref{defalphaell} has a nonzero lower right coefficient (use \eqref{formuleinv}):
$$
\alpha_{\mu+1}(\eps) =\dfrac{\eps^{-\mu/m}}{\prod_{\ell=1}^\mu (\omega_{\mu+1}-\omega_\ell)} \, .
$$
In other words, we have shown that for a suitable constant $C>0$, and for all $\eps>0$ arbitrarily small, 
there holds
$$
1 \le C \, \eps^{(2\, \mu+1)/m-1} +o(\eps^{(2\, \mu+1)/m-1})\, .
$$
Because of the already mentioned result of \cite[Proposition 4.1]{jfcpise}, this forces $m$ to be odd and 
$\mu$ to equal $(m-1)/2$ for otherwise we are led to a contradiction.
\bigskip

It therefore remains to deal with the last possible case: $m$ ($\ge 3$) is an odd number, and $\mu =(m-1)/2$. 
In particular, there are at least two unstable eigenvalues for $\M_g(z_\eps)$. The proof of Theorem \ref{thm2} 
in this case is a slight refinement of the above argument, which consists in first defining (compare with 
\eqref{defFeps}):
$$
{\mathcal F}(\eps) := \sum_{\ell=1}^{\mu+2} \alpha_\ell(\eps) \, {\mathcal T}_\ell(\eps) \, ,
$$
with coefficients $\alpha_1(\eps),\dots,\alpha_{\mu+2}(\eps)$ defined by the relation
$$
\begin{pmatrix}
1 & \cdots & 1 \\
\vdots & & \vdots \\
\omega_1^\mu & \cdots & \omega_{\mu+2}^\mu \end{pmatrix} \, \begin{pmatrix}
\alpha_1(\eps) \\
\vdots \\
\alpha_{\mu+2}(\eps) \end{pmatrix} =\begin{pmatrix}
0 \\
\vdots \\
0 \\
\eps^{-(\mu+1)/m} \end{pmatrix} \, .
$$
In that case, we have
$$
\Pi^u (z_\eps) \, {\mathcal F}(\eps) =\alpha_{\mu+1}(\eps) \, {\mathcal T}_{\mu+1}(\eps) 
+\alpha_{\mu+2}(\eps) \, {\mathcal T}_{\mu+2}(\eps) \, ,
$$
which is a little less nice than before because the unstable part of ${\mathcal F}(\eps)$ has components 
on two (mostly parallel) eigenvectors of $\M(z_\eps)$ so there might be some cancelation between those 
two components. We keep nevertheless the same definition as above for the source term in 
\eqref{solresolvent}, namely:
$$
\forall \, j_1 \ge 1 \, ,\quad {\mathcal F}_{j_1} :=\Big( |\kappa_{\mu+1}(\eps)|^2 -1 \Big)^{1/2} \, 
\dfrac{1}{\overline{\kappa_{\mu+1}(\eps)^{j_1}}} \, {\mathcal F}(\eps) \, ,
$$
but this time with our new definition of ${\mathcal F}(\eps)$. We still have
$$
\sum_{j_1 \ge 1} |{\mathcal F}_{j_1}|^2 =c+o(1) \, ,
$$
but now the corresponding vector ${\mathcal W}_1$ in \eqref{solresolvent} is given by
\begin{equation*}
-\big( |\kappa_{\mu+1}(\eps)|^2 -1 \big)^{-1/2} \, {\mathcal W}_1 
=\dfrac{\alpha_{\mu+1}(\eps)}{|\kappa_{\mu+1}(\eps)|^2 -1} \, {\mathcal T}_{\mu+1}(\eps) 
+\dfrac{\alpha_{\mu+2}(\eps)}{\overline{\kappa_{\mu+1}(\eps)} \, \kappa_{\mu+2}(\eps)-1} 
\, {\mathcal T}_{\mu+2}(\eps) \, .
\end{equation*}
We use again the explicit formula \eqref{formuleinv} for the inverse of a Vandermonde matrix to derive
\begin{align*}
\alpha_{\mu+1}(\eps) &= \dfrac{\eps^{-(\mu+1)/m}}{(\omega_{\mu+1}-\omega_{\mu+2}) \, 
\prod_{k=1}^\mu (\omega_{\mu+1}-\omega_k)} \, ,\\
\alpha_{\mu+2}(\eps) &= \dfrac{\eps^{-(\mu+1)/m}}{(\omega_{\mu+2}-\omega_{\mu+1}) \, 
\prod_{k=1}^\mu (\omega_{\mu+2}-\omega_k)} \, ,\\
\end{align*}
and the Puiseux expansion of the eigenvalues $\kappa_\ell$ give
$$
|\kappa_{\mu+1}(\eps)|^2 -1 \sim \zeta_1 \, (\overline{\omega_{\mu+1}} +\omega_{\mu+1}) \, 
\eps^{1/m} \, ,\quad 
\overline{\kappa_{\mu+1}(\eps)} \, \kappa_{\mu+2}(\eps)-1 \sim \zeta_1 \, 
(\overline{\omega_{\mu+1}} +\omega_{\mu+2}) \, \eps^{1/m} \, .
$$
Moreover, the eigenvectors ${\mathcal T}_{\mu+1}(\eps)$ and ${\mathcal T}_{\mu+2}(\eps)$ share the 
same finite non zero limit as $\eps$ tends to zero. Simplifying the previous expression of ${\mathcal W}_1$ 
by the non-zero quantity $\zeta_1 \, (\omega_{\mu+2}-\omega_{\mu+1})$, we obtain that there exists a 
non-zero vector ${\bf W}$ in $\C^{p_1+r_1}$ such that
$$
\eps^{(\mu+3/2)/m} \, {\mathcal W}_1 \rightarrow \left( 
\dfrac{1}{(\overline{\omega_{\mu+1}} +\omega_{\mu+1}) \, \prod_{k=1}^\mu (\omega_{\mu+1}-\omega_k)} 
-\dfrac{1}{(\overline{\omega_{\mu+1}} +\omega_{\mu+2}) \, \prod_{k=1}^\mu (\omega_{\mu+2}-\omega_k)} 
\right) \, {\bf W} \, . 
$$
Since $2\, \mu+3$ is larger than $m$, we shall obtain a contradiction as in the previous simpler analysis 
provided that we can show that the quantity
\begin{equation}
\label{argfinal}
(\overline{\omega_{\mu+1}} +\omega_{\mu+1}) \, \prod_{k=1}^\mu (\omega_{\mu+1}-\omega_k) 
-(\overline{\omega_{\mu+1}} +\omega_{\mu+2}) \, \prod_{k=1}^\mu (\omega_{\mu+2}-\omega_k) \, ,
\end{equation}
is non-zero.

Let us recall a little the situation here. The $\omega_\ell$, $\ell=1,\dots,\mu$ are the $m$-th roots of unity 
for which $\kappa_\ell$ is a stable eigenvalue. This corresponds to those $m$-th roots of unity for which 
$\zeta_1 \, \omega_\ell$ has negative real part (recall that $\zeta_1$ is a non-zero real number). In 
particular, the set $\{ \omega_1,\dots,\omega_\mu \}$ is invariant by complex conjugation. Furthermore, 
$\omega_{\mu+1}$ and $\omega_{\mu+2}$ are any $m$-th roots of unity for which $\zeta_1 \, \omega_\ell$ 
has positive real part. In particular, one can always choose $\omega_{\mu+2}$ as the complex conjugate 
of $\omega_{\mu+1}$. In that case, \eqref{argfinal} reduces to showing
$$
(\overline{\omega_{\mu+1}} +\omega_{\mu+1}) \, \prod_{k=1}^\mu (\omega_{\mu+1}-\omega_k) 
-2 \, \overline{\omega_{\mu+1}} \, \overline{\prod_{k=1}^\mu (\omega_{\mu+1}-\omega_k)} \neq 0 \, ,
$$
and a sufficient condition for this to happen is
\begin{equation}
\label{argfinal'}
\big( \text{\rm Im } \omega_{\mu+1} \big) \, \text{\rm Im} \left( 
\prod_{k=1}^\mu (\omega_{\mu+1}-\omega_k) \right) \neq 0 \, .
\end{equation}
If $\zeta_1$ is negative, then necessarily $m$ is of the form $3+4\, M$, $M\in \N$, and $\mu=1+2\, M$. 
We then choose
$$
\omega_{\mu+1} :=\exp \left( 2\, i \, \pi \dfrac{M+1}{4\, M+3} \right) \, ,
$$
and one can rather easily check that the condition \eqref{argfinal'} is satisfied. A similar calculation yields 
\eqref{argfinal'} if $\zeta_1$ is positive (in that case $m$ is of the form $5+4\, M$, $M\in \N$).

\section{Examples}
\label{sect:examples}

\subsection{Numerical schemes for the transport equation}

In this first Paragraph, the underlying partial differential equation we consider is a one-dimensional 
transport equation:
$$
\partial_t v +a \, \partial_x v =0 \, ,
$$
with $a \neq 0$ a fixed velocity. We explain how the theory developed in this article applies to two 
possible discretizations of that equation, namely the Lax-Wendroff and leap-frog schemes, both of 
which being second order in time and space (at least for sufficiently smooth solutions).

\paragraph{The Lax-Wendroff scheme.}

Given some time and space steps $\Delta t, \Delta x$, and letting for simplicity $\mu$ denote the 
dimensionless parameter
$$
\mu := \dfrac{\Delta t}{\Delta x} \, a \neq 0 \, ,
$$
the Lax-Wendroff scheme reads:
\begin{equation}
\label{lw}
\begin{cases}
u_j^{n+1} -\dfrac{\mu}{2} \, (1+\mu) \, u_{j-1}^n +(\mu^2-1) \, u_j^n 
+\dfrac{\mu}{2} \, (1-\mu) \, u_{j+1}^n =0 \, ,& n \in \N \, ,\, j \in \Z \, ,\\
u^0 =f \in \ell^2(\Z) \, .
\end{cases}
\end{equation}
The scheme \eqref{lw} fits into the framework of \eqref{cauchy} with $s=0$, $Q_1=I$ (the scheme is 
explicit), and
$$
a_{-1,0}=-\dfrac{\mu}{2} \, (1+\mu) \, ,\quad a_{0,0} =\mu^2-1 \, ,\quad 
a_{1,0} =\dfrac{\mu}{2} \, (1-\mu) \, .
$$
The stencil of the scheme \eqref{lw} depends on whether $|\mu|=1$. To avoid dealing with particular 
cases, we therefore assume $|\mu| \neq 1$, and therefore $p=r=1$ in \eqref{lw}. Then it is a rather 
well-known fact, see e.g. \cite{gko}, that the Lax-Wendroff scheme satisfies the stability Assumption 
\ref{assumption1} if and only if $|\mu|<1$. Since $s=0$, the stability estimate \eqref{bornestabilitecauchy} 
is even satisfied with the constant $1$ (the $\ell^2$ norm of the solution to \eqref{lw} is nonincreasing 
with respect to $n$). We thus fix from now on a positive constant $\lambda$ such that $\lambda \, |a|<1$. 
The set $\bdelta$ of discretization parameters is
$$
\bdelta := \big\{ (\Delta t,\Delta t/\lambda) \, ,\quad \Delta t \in (0,1] \big\} \, ,
$$
and we consider the iteration \eqref{lw} with $\mu :=\lambda\, a$. The coefficients in \eqref{lw} are 
therefore independent of the discretization parameters $(\Delta t,\Delta x) \in \bdelta$. Because 
$a_{-1,0}$ and $a_{1,0}$ are nonzero, both Assumptions \ref{assumption0} and Assumption 
\ref{assumption3} are satisfied. We have also seen that Assumption \ref{assumption1} is satisfied 
thanks to our restriction on $\lambda$. The definition \eqref{defA-d} reduces here to
$$
a_{-1}(z)=a_{-1,0} \neq 0 \, ,\quad a_1(z)=a_{1,0} \neq 0\, ,
$$
and therefore Assumption \ref{assumption2'} is satisfied (the strong version of the non-characteristic 
boundary assumption). We are now going to verify that the Lax-Wendroff scheme is non-glancing. 
The amplification matrix ${\mathcal A}$ in \eqref{defA} reduces here to the complex number
$$
\forall \, \kappa \in \C \setminus \{ 0\} \, ,\quad {\mathcal A}(\kappa) 
=\dfrac{\mu}{2} \, (1+\mu) \, \kappa^{-1} -\dfrac{\mu}{2} \, (1-\mu) \, \kappa +1-\mu^2 \, .
$$
We thus compute
$$
\forall \, \xi \in \R \, ,\quad |{\mathcal A}({\rm e}^{i\, \xi})|^2 =1-4\, \mu^2 \, (1-\mu^2) \, \sin^4 \dfrac{\xi}{2} 
\le 1 \, .
$$
Therefore the only $\kappa \in \cercle$ for which ${\mathcal A}(\kappa)$ has modulus $1$ is $\kappa=1$, 
but then ${\mathcal A}'(1)=-\mu$ is nonzero. The Lax-Wendroff scheme \eqref{lw} is therefore 
non-glancing\footnote{This is actually a consequence here of the fact that the scheme is dissipative of 
order $4$, see \cite[Chapter 5]{gko}.} and we can apply Theorem \ref{thm2} when the computational 
domain $\Z$ is truncated on one side with the transparent numerical boundary conditions which we 
now make explicit.

We examine the transparent boundary condition at $j=0$ associated with \eqref{lw}. For convenience, 
we adopt here the formulation \eqref{transparent1} using a linear form for describing the unstable 
subspace $\E^u(z)$ of the matrix $\M(z)$ in \eqref{defM11}. Let us recall that in \eqref{transparent1}, 
the coefficients $\widetilde{\kappa}_{u,n}$ correspond to the Laurent series of $\kappa_u^{-1}$, where 
$\kappa_u(z) \in \U$ is the unstable root to the equation
$$
a_{-1}(z) +a_0(z) \, \kappa +a_1(z) \, \kappa^2 =0 \, .
$$
For the Lax-Wendroff scheme \eqref{lw}, $\kappa_u(z)^{-1}$ satisfies the equation
$$
\forall \, z \in \U \, ,\quad 
-\mu \, (1+\mu) \, \kappa_u(z)^{-2} +2 \, (z+\mu^2-1) \, \kappa_u(z)^{-1} +\mu \, (1-\mu) =0 \, .
$$
We plug the series $\sum_{n \ge 1} \widetilde{\kappa}_{u,n} \, z^{-n}$ in the latter equation and identify 
inductively the coefficients, which yields:
\begin{align}
\forall \, n \ge 2 \, ,\quad \widetilde{\kappa}_{u,n+1} &= (1-\mu^2) \, \widetilde{\kappa}_{u,n}
+\dfrac{1}{2} \, \mu \, (1+\mu) \, \sum_{m=1}^{n-1} \widetilde{\kappa}_{u,m} \, \widetilde{\kappa}_{u,n-m}
\, ,\label{coeffslw} \\
\text{\rm with } \quad \widetilde{\kappa}_{u,1} &=-\dfrac{1}{2} \, \mu \, (1-\mu) \, ,\quad 
\widetilde{\kappa}_{u,2} =(1-\mu^2) \, \widetilde{\kappa}_{u,1} \, .\notag
\end{align}
From the relation $\kappa_s(z) \, \kappa_u(z) =-(1+\mu)/(1-\mu)$, we also get the Laurent series expansion 
of $\kappa_s$:
\begin{equation}
\label{coeffslw'}
\kappa_s(z) =\sum_{n \ge 1} \dfrac{\kappa_{s,n}}{z^n} \, ,\quad \kappa_{s,n} := -\dfrac{1+\mu}{1-\mu} \, 
\widetilde{\kappa}_{u,n} \, .
\end{equation}
The stable eigenvalue $\kappa_s(z)$ is used to write the transparent numerical boundary condition at 
the right end $j=J+1$ of the computational domain.

If one thus  truncates the computation domain $\Z$ on both sides and therefore reduces to an interval 
$[0,J+1]$, the resulting numerical scheme reads
\begin{equation*}
\begin{cases}
u_j^{n+1} -\dfrac{\mu}{2} \, (1+\mu) \, u_{j-1}^n +(\mu^2-1) \, u_j^n 
+\dfrac{\mu}{2} \, (1-\mu) \, u_{j+1}^n =0 \, ,& n \in \N \, , \, j=1,\dots,J \, ,\\
u_0^{n+1} ={\dps \sum_{m=0}^n} \widetilde{\kappa}_{u,n+1-m} \, u_1^m \, ,& n \in \N \, ,\\
u_{J+1}^{n+1} ={\dps \sum_{m=0}^n} \kappa_{s,n+1-m} \, u_J^m \, ,& n \in \N \, ,
\end{cases}
\end{equation*}
with coefficients $\widetilde{\kappa}_{u,n}, \kappa_{s,n}$ defined in \eqref{coeffslw}, \eqref{coeffslw'}, and 
some given initial condition $(u_0^0,\dots,u_{J+1}^0)^T \in \R^{J+2}$. The above numerical scheme is 
rather easily implemented thanks to the recursive formula for the coefficients $\widetilde{\kappa}_{u,n}$ 
and $\kappa_{s,n}$. But of course there are easier and efficient {\it local} strategies that are based on 
absorbing boundary conditions and that may work quite as well, see e.g. \cite{ehrhardt,goldberg}.

\paragraph{The leap-frog scheme.}

Given some time and space steps $\Delta t, \Delta x$, and letting again for simplicity $\mu$ denote the 
dimensionless parameter
$$
\mu := \dfrac{\Delta t}{\Delta x} \, a \, ,
$$
the leap-frog scheme reads:
\begin{equation}
\label{lf}
\begin{cases}
u_j^{n+2} +\mu \, (u_{j+1}^{n+1}-u_{j-1}^{n+1}) -u_j^n =0 \, ,& n \in \N \, ,\, j \in \Z \, ,\\
(u^0,u^1) =(f^0,f^1) \in \ell^2(\Z)^2 \, .
\end{cases}
\end{equation}
The scheme \eqref{lf} fits into the framework of \eqref{cauchy} with $s=1$, $Q_2=-Q_0=I$ (the scheme is 
explicit), and
$$
a_{-1,1}=-\mu \, ,\quad a_{0,0} =0 \, ,\quad a_{1,0} =\mu \, ,\quad p=r=1 \, .
$$
Assumptions \ref{assumption0} and \ref{assumption3} are thus satisfied. It is also a standard result that 
the leap-frog scheme is $\ell^2$ stable if and only if $|\mu|<1$, see \cite{RM}. In that case, Assumption 
\ref{assumption1} is satisfied with a constant $C$ that only depends on $\mu$. From the above expression 
of the coefficients, we can also easily check that Assumption \ref{assumption2'} is satisfied.

The amplification matrix ${\mathcal A}$ in \eqref{defA} reads
$$
{\mathcal A}(\kappa) =\begin{pmatrix}
-\mu \, (\kappa-\kappa^{-1}) & 1 \\
1 & 0 \end{pmatrix} \, .
$$
The eigenvalues of ${\mathcal A}({\rm e}^{i\, \xi})$, $\xi \in \R$, are
$$
\pm \sqrt{1-\mu^2 \, \sin^2 \xi} -i \, \mu \, \sin \xi \in \cercle \, .
$$
The derivative of these functions with respect to $\xi$ vanishes when $\xi-\pi/2$ belongs to $\Z \, \pi$, which 
means that the leap-frog scheme admits glancing wave packets, which prevents from applying Theorem 
\ref{thm2}. We can nevertheless derive the transparent boundary conditions by computing the roots 
$\kappa_s(z) \in \D$ and $\kappa_u(z) \in \U$ to the equation
$$
-\mu \, z \, \kappa^{-1} +(z^2-1) +\mu \, z \, \kappa =0 \, ,\quad z \in \U \, .
$$
In particular, the inverse of the unstable root has the following Laurent series expansion
$$
\kappa_u(z)^{-1} =\sum_{n \ge 1} \widetilde{\kappa}_{u,n} \, z^{-n} \, ,
$$
with:
\begin{align*}
\forall \, n \ge 2 \, ,\quad \widetilde{\kappa}_{u,n+1} &= \widetilde{\kappa}_{u,n-1} 
-\mu \, \sum_{m=1}^{n-1} \widetilde{\kappa}_{u,m} \, \widetilde{\kappa}_{u,n-m} \, , \\
\widetilde{\kappa}_{u,1} &= \mu \, ,\quad \widetilde{\kappa}_{u,2} =0 \, .
\end{align*}
The stable root $\kappa_s(z)$ coincides with $-\kappa_u(z)^{-1}$. We can therefore truncate the computation 
domain $\Z$ and implement the numerical scheme
\begin{equation*}
\begin{cases}
u_j^{n+2} -\mu \, (u_{j+1}^{n+1}-u_{j-1}^{n+1}) -u_j^n =0 \, ,& n \in \N \, , \, j=1,\dots,J \, ,\\
u_0^{n+2} ={\dps \sum_{m=0}^n} \widetilde{\kappa}_{u,n+1-m} \, u_1^m \, ,& n \in \N \, ,\\
u_{J+1}^{n+2} ={\dps \sum_{m=0}^n} \kappa_{s,n+1-m} \, u_J^m \, ,& n \in \N \, ,
\end{cases}
\end{equation*}
with any given couple of initial data $(u_0^0,\dots,u_{J+1}^0)^T, (u_0^1,\dots,u_{J+1}^1)^T \in \R^{J+2}$.

\subsection{Numerical schemes for the heat equation}

In this Paragraph, the underlying partial differential equation we consider is the one-dimensional 
heat equation:
$$
\partial_t v -d \, \partial^2_{xx} v =0 \, ,
$$
with $d>0$ the diffusion coefficient. We review our derivation of the transparent boundary condition both 
for the (more than) classical explicit scheme
\begin{equation}
\label{heat1}
\begin{cases}
u_j^{n+1} -\dfrac{d \, \Delta t}{\Delta x^2} \, u_{j-1}^n +2\, \dfrac{d \, \Delta t}{\Delta x^2} \, u_j^n 
-\dfrac{d \, \Delta t}{\Delta x^2} \, u_{j+1}^n =0 \, ,& n \in \N \, ,\, j \in \Z \, ,\\
u^0 =f \in \ell^2(\Z) \, ,
\end{cases}
\end{equation}
and for the implicit scheme based on the BDF2 quadrature rule (see \cite{hnw}):
\begin{equation}
\label{heat2}
\begin{cases}
\left( \dfrac{3}{2} +2\, \dfrac{d \, \Delta t}{\Delta x^2} \right) \, u_j^{n+2} 
-\dfrac{d \, \Delta t}{\Delta x^2} \, u_{j-1}^{n+2} -\dfrac{d \, \Delta t}{\Delta x^2} \, u_{j+1}^{n+2} 
-2 \, u_j^{n+1} +\dfrac{1}{2} \, u_j^n =0 \, ,& n \in \N \, ,\, j \in \Z \, ,\\
(u^0,u^1) =(f^0,f^1) \in \ell^2(\Z)^2 \, .
\end{cases}
\end{equation}
The scheme \eqref{heat1} is of first order in time and second order in space, while \eqref{heat2} is 
second order in both time and space (again for sufficiently smooth solutions).

Let us first deal with \eqref{heat1}, which, in the framework of \eqref{cauchy}, corresponds to $s=0$, 
$Q_1=I$ (the scheme is explicit) and
$$
a_{-1,0}(\Delta t,\Delta x)=a_{1,0}(\Delta t,\Delta x) :=-\dfrac{d \, \Delta t}{\Delta x^2} \, ,\quad 
a_{0,0}(\Delta t,\Delta x) :=2\, \dfrac{d \, \Delta t}{\Delta x^2} \, .
$$
We emphasize here the dependence of the coefficients on $\Delta t$ and $\Delta x$, though 
of course they only depend on the ratio $\Delta t/\Delta x^2$. The numerical scheme satisfies 
Assumptions \ref{assumption0}, \ref{assumption2} and \ref{assumption3} (it even satisfies the 
stronger Assumption \ref{assumption2'} though we shall not make much use of this fact). As far 
as the stability Assumption \ref{assumption1} is concerned, it is satisfied if and only if there holds 
$d\, \Delta t/\Delta x^2 \le 1$, which can be readily seen by applying the Fourier transform \cite{RM} 
(see also \cite{strang1} for an alternative explanation of this stability condition based on Bernstein's 
inequality). The admissible set of discretization parameters is therefore
$$
\bdelta := \big\{ (\Delta t,\Delta x) \in (0,1]^2 \, / \, d \, \Delta t \le \Delta x^2 \} \, .
$$

For the scheme \eqref{heat1}, the derivation of the transparent boundary conditions is based on 
the analysis of the polynomial equation:
$$
a_{-1,0}(\Delta t,\Delta x) +(z+a_{0,0}(\Delta t,\Delta x)) \, \kappa +a_{1,0}(\Delta t,\Delta x) \, \kappa^2 
=0 \, ,
$$
which, in view of the definition of the coefficients $a_{\ell,0}(\Delta t,\Delta x)$, amounts to
$$
\mu \, (\kappa-1)^2 =z \, \kappa \, ,\quad \mu := \dfrac{d \, \Delta t}{\Delta x^2} \in (0,1] \, .
$$
For $z \in \U$, this equation has one root $\kappa_s(z) \in \D$ and one root $\kappa_u(z) \in \U$, 
both of which depend holomorphically on $z$. The Laurent series expansion of $\kappa_u(z)^{-1}$ 
reads
$$
\kappa_u(z)^{-1} =\sum_{n \ge 1} \widetilde{\kappa}_{u,n} \, z^{-n} \, ,
$$
with:
\begin{align*}
\forall \, n \ge 2 \, ,\quad \widetilde{\kappa}_{u,n+1} &=-2\, \mu \, \widetilde{\kappa}_{u,n} 
+\mu \, \sum_{m=1}^{n-1} \widetilde{\kappa}_{u,m} \, \widetilde{\kappa}_{u,n-m} \, , \\
\widetilde{\kappa}_{u,1} &= \mu \, ,\quad \widetilde{\kappa}_{u,2} =-2\, \mu^2 \, .
\end{align*}
The stable root $\kappa_s(z)$ coincides with $\kappa_u(z)^{-1}$ so we use the convention $\kappa_{s,n} 
:=\widetilde{\kappa}_{u,n}$ in the numerical scheme just below. We can truncate the computation domain 
$\Z$ in \eqref{heat1} and implement the numerical scheme
\begin{equation*}
\begin{cases}
u_j^{n+1} -\mu \, (u_{j-1}^n -2 \, u_j^n +u_{j+1}^n) =0 \, ,& n \in \N \, , \, j=1,\dots,J \, ,\\
u_0^{n+1} ={\dps \sum_{m=0}^n} \widetilde{\kappa}_{u,n+1-m} \, u_1^m \, ,& n \in \N \, ,\\
u_{J+1}^{n+1} ={\dps \sum_{m=0}^n} \kappa_{s,n+1-m} \, u_J^m \, ,& n \in \N \, ,
\end{cases}
\end{equation*}
with $\mu :=d \, \Delta t/\Delta x^2$ and given initial data $(u_0^0,\dots,u_{J+1}^0)^T \in \R^{J+2}$.

We now consider the implicit scheme \eqref{heat2}. Since the scheme is implicit, we first need to 
determine whether it is well-defined, that is whether Assumption \ref{assumption0} is satisfied. 
We have
$$
\forall \, \kappa \in \C \setminus \{ 0\} \, ,\quad 
\widehat{Q_2}(\kappa) =\dfrac{3}{2} +2\, \mu  -\mu \, (\kappa+\kappa^{-1}) \, ,
$$
and it is therefore immediate to verify that $\widehat{Q_2}$ does not vanish on $\cercle$. Furthermore, 
$\widehat{Q_2}$ has exactly two zeroes that are real (recall $\mu>0$); one is located in the interval 
$(0,1)$ and the remaining one in $(1,+\infty)$. (Actually, one is the inverse of the other.) By the residue 
Theorem, the index condition \eqref{index} is satisfied. Let us now verify Assumption \ref{assumption1}. 
Following \cite{Emmrich1,Emmrich2}, we are going to use the so-called $G$-stability of the BDF2 quadrature 
rule, see \cite[Chapter V.6]{hw}. We mulitply the recurrence relation in \eqref{heat2} by $\Delta x \, u_j^{n+2}$ 
and sum over $j \in \Z$. Using the identity
\begin{multline*}
4\, u_j^{n+2} \, \left( \dfrac{3}{2} \, u_j^{n+2} -2\, u_j^{n+1} +\dfrac{1}{2} \, u_j^n \right) \\
=(u_j^{n+2})^2 +(2\, u_j^{n+2} -u_j^{n+1})^2 -(u_j^{n+1})^2 -(2\, u_j^{n+1} -u_j^n)^2 
+(u_j^{n+2}-2\, u_j^{n+1}+u_j^n)^2 \, ,
\end{multline*}
and discrete integration by parts in $j$, we end up with
\begin{multline*}
\Ng u^{n+2} \Nd_{-\infty,+\infty}^2 +\Ng 2\, u^{n+2} -u^{n+1} \Nd_{-\infty,+\infty}^2 
-\Ng u^{n+1} \Nd_{-\infty,+\infty}^2 -\Ng 2\, u^{n+1} -u^n \Nd_{-\infty,+\infty}^2 \\
=-\dfrac{1}{4} \, \Ng u^{n+2} -2\, u^{n+1} +u^n \Nd_{-\infty,+\infty}^2 
-\dfrac{\mu}{4} \,  \Ng ({\bf S}-I) \, u^{n+2} \Nd_{-\infty,+\infty}^2 \le 0 \, .
\end{multline*}
In other words, the energy
$$
E^n := \Ng u^{n+1} \Nd_{-\infty,+\infty}^2 +\Ng 2\, u^{n+1} -u^n \Nd_{-\infty,+\infty}^2 \, ,
$$
is nonincreasing for solutions to \eqref{heat2}, independently of $\mu>0$, and this proves that 
Assumption \ref{assumption1} is satisfied with all possible discretization parameters $(\Delta t,\Delta x)$ 
(that is for $\bdelta :=(0,1]^2$, and the corresponding constant $C$ in \eqref{bornestabilitecauchy} is 
independent of $(\Delta t,\Delta x) \in \bdelta$). We also compute
$$
a_{-1}(z) =-\mu \, z^2 \, ,\quad a_1(z) =-\mu \, z^2 \, ,
$$
so Assumptions \ref{assumption2} and \ref{assumption3} are satisfied. We can thus proceed with the 
construction of transparent boundary conditions for \eqref{heat2}. The equation of interest reads
$$
\kappa^2 -\dfrac{1}{\mu \, z^2} \, \left\{ \left( \dfrac{3}{2} +2\, \mu \right) \, z^2 -2\, z +\dfrac{1}{2} \right\} 
\, \kappa +1 =0 \, .
$$
When $z$ belongs to $\U$, it has one root $\kappa_u(z) \in \U$ and one root $\kappa_s(z) =\kappa_u(z)^{-1} 
\in \D$. We determine the Laurent series expansion of $\kappa_s$:
$$
\kappa_s(z) =\sum_{n \ge 0} \dfrac{\kappa_{s,n}}{z^n} \, .
$$
Plugging this series in the polynomial equation satisfied by $\kappa_s$, we end up with the recursive 
relations:
\begin{align*}
\kappa_{s,0}^2 -\dfrac{1}{\mu} \, \left( \dfrac{3}{2} +2\, \mu \right) \, \kappa_{s,0} +1 &=0 \, ,\quad 
\kappa_{s,0} \in (0,1) \, , \\
\left( 2\, \kappa_{s,0} -\dfrac{3}{2\, \mu} -2 \right) \, \kappa_{s,1} &=-\dfrac{2\, \kappa_{s,0}}{\mu} \, , \\
\forall \, n \ge 2 \, ,\quad \left( 2\, \kappa_{s,0} -\dfrac{3}{2\, \mu} -2 \right) \, \kappa_{s,n} &= 
-\dfrac{2\, \kappa_{s,n-1}}{\mu} +\dfrac{\kappa_{s,n-2}}{2\, \mu} 
-\sum_{m=1}^{n-1} \kappa_{s,m} \, \kappa_{s,n-m} \, .
\end{align*}
It should be noted that a crucial fact that we use here is that $\kappa_{s,0}$ is a simple root of 
the equation $\widehat{Q_2}(\kappa)=0$, which enables us indeed to determine the sequence 
$(\kappa_{s,n})$ inductively.

Since the Laurent series expansion of $\kappa_s$ and $\kappa_u^{-1}$ coincide, the truncation of 
\eqref{heat2} on a finite interval $[0,J+1]$ reads
\begin{equation*}
\begin{cases}
\left( \dfrac{3}{2} +2\, \dfrac{d \, \Delta t}{\Delta x^2} \right) \, u_j^{n+2} 
-\dfrac{d \, \Delta t}{\Delta x^2} \, u_{j-1}^{n+2} -\dfrac{d \, \Delta t}{\Delta x^2} \, u_{j+1}^{n+2} 
-2 \, u_j^{n+1} +\dfrac{1}{2} \, u_j^n =0 \, ,& n \in \N \, ,\, j=1,\dots,J \, ,\\
u_0^{n+1} -\kappa_{s,0} \, u_1^{n+1}={\dps \sum_{m=0}^n} \kappa_{s,n+1-m} \, u_1^m 
\, ,& n \in \N \, ,\\
u_{J+1}^{n+1} -\kappa_{s,0} \, u_J^{n+1} ={\dps \sum_{m=0}^n} \kappa_{s,n+1-m} \, u_J^m \, ,& 
n \in \N \, ,
\end{cases}
\end{equation*}
with the previous recursive definition for the $\kappa_{s,n}$'s, and any couple of initial conditions 
$(u_0^0,\dots,u_{J+1}^0)^T$, $(u_0^1,\dots,u_{J+1}^1)^T \in \R^{J+2}$.

\subsection{Numerical schemes for dispersive equations}

\paragraph{The two-dimensional Schr\"odinger equation}

We consider the two-dimensional linear Schr\"odinger equation
$$
i \, \partial_t v +\dfrac{1}{2} \, \Delta_x v =0 \, ,\quad (t,x) \in \R \times \R^2 \, .
$$
We consider the numerical scheme proposed in \cite{ehrhardt-arnold} that is based on a centered 
second order differentiation in space and the Crank-Nicolson quadrature rule. This yields the numerical 
scheme:
\begin{multline}
\label{cnSchrod}
i \, \dfrac{u_{j_1,j_2}^{n+1} -u_{j_1,j_2}^n}{\Delta t} +\dfrac{1}{4 \, \Delta x_1^2} \Big( 
u_{j_1+1,j_2}^{n+1}-2\, u_{j_1,j_2}^{n+1} +u_{j_1-1,j_2}^{n+1} 
+u_{j_1+1,j_2}^n-2\, u_{j_1,j_2}^n +u_{j_1-1,j_2}^n \Big) \\
+\dfrac{1}{4 \, \Delta x_2^2} \Big( 
u_{j_1,j_2+1}^{n+1}-2\, u_{j_1,j_2}^{n+1} +u_{j_1,j_2-1}^{n+1} 
+u_{j_1,j_2+1}^n-2\, u_{j_1,j_2}^n +u_{j_1,j_2-1}^n \Big) =0 \, ,
\end{multline}
with some given initial condition $u^0 \in \ell^2(\Z^2;\C)$. For future use, we introduce the positive 
parameters:
$$
\mu_1 := \dfrac{\Delta t}{4 \, \Delta x_1^2} \, ,\quad \mu_2 := \dfrac{\Delta t}{4 \, \Delta x_2^2} \, .
$$
The scheme \eqref{cnSchrod} fits into the framework of \eqref{cauchy} with $p_1=p_2=r_1=r_2=1$, and 
with the operators
\begin{align*}
Q_1 &:=i\, I +\mu_1 \, ({\bf S}_1 +{\bf S}_1^{-1} -2\, I) +\mu_2 \, ({\bf S}_2 +{\bf S}_2^{-1} -2\, I) \, , \\
Q_0 &:=-i\, I +\mu_1 \, ({\bf S}_1 +{\bf S}_1^{-1} -2\, I) +\mu_2 \, ({\bf S}_2 +{\bf S}_2^{-1} -2\, I) \, .
\end{align*}
In particular, there holds
$$
\widehat{Q_1}({\rm e}^{i\, \eta_1},{\rm e}^{i\, \eta_2}) 
= i -4\, \mu_1 \, \sin^2 \dfrac{\eta_1}{2} -4\, \mu_2 \, \sin^2 \dfrac{\eta_2}{2} \, ,
$$
so that not only $\widehat{Q_1}({\rm e}^{i\, \eta_1},{\rm e}^{i\, \eta_2})$ is nonzero (that is $Q_1$ is an 
isomorphism on $\ell^2$), but $\widehat{Q_1}(\cdot,{\rm e}^{i\, \eta_2})$ maps the unit circle $\cercle$ 
into the upper half-plane $\{ \zeta \in \C \, , \, \text{\rm Im } \zeta >0 \}$. Hence we can write
$$
\dfrac{1}{2\, i \, \pi} \, \int_{\cercle} 
\dfrac{\partial_{\kappa_1} \widehat{Q_1}(\kappa_1,{\rm e}^{i\, \eta_2})}
{\widehat{Q_1}(\kappa_1,{\rm e}^{i\, \eta_2})} \, {\rm d}\kappa_1 
=\dfrac{1}{2\, i \, \pi} \, \int_{\cercle} \partial_{\kappa_1} \big( \ln \widehat{Q_1}(\kappa_1,{\rm e}^{i\, \eta_2}) 
\big) \, {\rm d}\kappa_1 =0 \, ,
$$
where we have used the principal determination of the logarithm. Hence Assumption \ref{assumption0} 
is satisfied. It is a rather standard property that \eqref{cnSchrod} preserves the $\ell^2$ norm, so 
Assumption \ref{assumption1} is satisfied with the maximal set of discretization parameters $\bdelta 
:= (0,1]^3$ (the constant $C$ in \eqref{bornestabilitecauchy} can be chosen to be $1$). With the 
definition \eqref{defA-d}, we compute (here $\bfeta=\eta_2$ belongs to $\R$ so we rather use the 
more explicit notation $\eta_2$):
$$
a_{-1}(z,\eta_2) =a_1(z,\eta_2) =(z+1) \, \mu_1 \, ,
$$
so Assumption \ref{assumption2} is satisfied (but Assumption \ref{assumption2'} is not !). We can also 
easily check that Assumption \ref{assumption3} is satisfied since $a_{-1}$ and $a_1$ have degree $1$ 
in $z$ for all $\eta_2$. The derivation of transparent numerical boundary conditions for \eqref{cnSchrod} 
was performed in \cite{ehrhardt-arnold} so we shall not reproduce it here. Approximate (namely, 
absorbing) numerical boundary conditions for \eqref{cnSchrod} are proposed and studied in 
\cite{ehrhardt-arnold,AES,aabes,ducomet-zlotnik}. We also refer to \cite{szeftel1,aabes} and references 
therein for the construction of absorbing boundary conditions for the nonlinear Schr\"odinger equation.

\paragraph{The Airy equation}

We are now going back to a one-dimensional problem and consider as in \cite{zheng-wen-han} the 
Airy equation
$$
\partial_t v +\partial_{xxx}^3 v =0 \, .
$$
The extension to a nonzero first order transport term is considered in \cite{BELV} in view of later 
dealing with the Korteweg - de Vries equation. For simplicity, we restrict to this simple framework 
($U_1=0$ and $U_2=1$ in the notation of \cite{BELV}) and explain how one of the schemes 
considered in \cite{BELV} fits into our framework. More precisely, we consider the so-called 
`rightside Crank-Nicolson' scheme proposed in \cite{Qin}:
\begin{equation}
\label{cnAiry}
\dfrac{u_j^{n+1} -u_j^n}{\Delta t} +\dfrac{1}{2 \, \Delta x^3} \Big( 
u_{j+2}^{n+1} -3\, u_{j+1}^{n+1} +3\, u_j^{n+1} -u_{j-1}^{n+1} 
+u_{j+2}^n -3\, u_{j+1}^n +3\, u_j^n -u_{j-1}^n \Big) =0 \, ,
\end{equation}
with some given initial condition $u^0 \in \ell^2$. The other scheme considered in \cite{BELV} 
is centered in space so the stability analysis is identical to the above one for the discretized 
Schr\"odinger equation. With the notation
$$
\mu := \dfrac{\Delta t}{2 \, \Delta x^3} >0 \, ,
$$
the numerical scheme \eqref{cnAiry} fits into the framework of \eqref{cauchy} with $p=2$, $r=1$, and
\begin{equation*}
Q_1 :=I +\mu \, ({\bf S}^2 -3\, {\bf S}_1 +3\, I -{\bf S}^{-1}) \, ,\quad 
Q_0 :=-I +\mu \, ({\bf S}^2 -3\, {\bf S}_1 +3\, I -{\bf S}^{-1}) \, .
\end{equation*}
In particular, there holds
$$
\widehat{Q_1}({\rm e}^{i\, \xi}) =1+8\, \mu \, \sin^4 \dfrac{\xi}{2} -4\, i \, \mu \, \sin \xi \, \sin^2 \dfrac{\xi}{2} \, .
$$
Hence not only $\widehat{Q_1}({\rm e}^{i\, \xi})$ is nonzero but its real part is not smaller than $1$. 
In particular, we can apply the same argument as above for the discretized Schr\"odinger equation 
and use the principal determination of the logarithm to show that $\widehat{Q_1}$ satisfies the index 
condition \eqref{index}. It is also proved in \cite{Qin} that the scheme \eqref{cnAiry} is stable, that is, 
it satisfies Assumption \ref{assumption1} with all possible discretization parameters, that is with 
$\bdelta :=(0,1]^2$. Since \eqref{cnAiry} is based on the Crank-Nicolson quadrature rule, there 
is no difficulty in verifying that Assumptions \ref{assumption2} and \ref{assumption3} are satisfied. 
Again Assumption \ref{assumption2'} is not satisfied. We refer to \cite{BELV} for a derivation of the 
transparent boundary conditions for the scheme \eqref{cnAiry}. In particular, we recover here in our 
general framework the separation property for the roots $\kappa(z)$ (Theorem 3.1 in \cite{BELV}). 
The analysis in \cite{BELV} makes clear how, for this case with $p=2$, one can write transparent 
boundary conditions by using linear forms rather than projectors. This requires in \cite{BELV} using 
some suitable combinations of the two unstable roots (named $\ell_2(z)$ and $\ell_3(z)$ here) in 
order to preserve holomorphy with respect to $z$ on $\U$.

\paragraph{The linearized Benjamin-Bona-Mahony equation}

We discuss eventually the linearized Benjamin-Bona-Mahony equation considered in the recent 
work \cite{BMN}. The partial differential equation under consideration reads
$$
\partial_t \, (u -\varepsilon \, \partial_{xx}^2 u) +c \, \partial_x u =0 \, ,
$$
with $\varepsilon>0$ and $c>0$ (the sign of $c$ is crucial in the `upwinding' procedure below). One 
numerical scheme proposed in \cite{BMN} is based on a centered difference for the dispersive part 
and on an upwind procedure for the transport part. Then one applies the Crank-Nicolson quadrature 
rule to integrate in time. The resulting scheme reads:
\begin{equation}
\label{cnBBM}
u_j^{n+1} -u_j^n -\dfrac{\varepsilon\, \Delta t}{2 \, \Delta x^2} \Big( 
u_{j+1}^{n+1} -2\, u_j^{n+1} +u_{j-1}^{n+1} -u_{j+1}^n +2\, u_j^n -u_{j-1}^n \Big) 
+\dfrac{c\, \Delta t}{2\, \Delta x} \, \Big( u_j^{n+1}-u_{j-1}^{n+1}+u_j^n -u_{j-1}^n \Big) =0 \, ,
\end{equation}
which fits into the framework of \eqref{cauchy} with $p=r=1$, and
$$
Q_1 = -\left( \dfrac{\varepsilon\, \Delta t}{2 \, \Delta x^2} +\dfrac{c\, \Delta t}{2\, \Delta x} \right) \, {\bf S}^{-1} 
+\left( 1+\dfrac{\varepsilon\, \Delta t}{\Delta x^2} +\dfrac{c\, \Delta t}{2\, \Delta x} \right) \, I 
-\dfrac{\varepsilon\, \Delta t}{2 \, \Delta x^2} \, {\bf S} \, .
$$
In particular, one computes
$$
\text{\rm Re } \widehat{Q_1}({\rm e}^{i\, \eta}) =1 
+2\, \dfrac{\varepsilon\, \Delta t}{\Delta x^2} \, \sin^2 \dfrac{\eta}{2} 
+\dfrac{c\, \Delta t}{\Delta x} \, \sin^2 \dfrac{\eta}{2} \ge 1 \, ,
$$
so not only $Q_1$ is an isomorphism on $\ell^2(\Z)$ but the index condition \eqref{index} is satisfied 
(we use the same argument as for the discretization of the Airy equation). Reproducing more or less 
the same computations as for \eqref{cnAiry}, one can also show that Assumption \ref{assumption1} 
is satisfied with all possible discretization parameters, that is with $\bdelta =(0,1]^2$. From the above 
expression of $Q_1$, one can also easily verify that Assumption \ref{assumption3} is satisfied.

We now compute
$$
a_{-1}(z) =-(z-1) \, \dfrac{\varepsilon\, \Delta t}{2 \, \Delta x^2} -(z+1) \, \dfrac{c\, \Delta t}{2\, \Delta x} 
\, ,\quad 
a_1(z) =-(z-1) \, \dfrac{\varepsilon\, \Delta t}{2 \, \Delta x^2} \, ,
$$
so obviously $a_1$ does not vanish on $\U$ (but it vanishes on $\Ubar$). The only root of $a_{-1}$ 
belongs to $\D$ so Assumption \ref{assumption2} is satisfied (but not Assumption \ref{assumption2'}). 
Hence the calculations made explicit in \cite{BMN} also fit in the general framework that we have 
discussed in this article.

\paragraph{Acknowledgments} The author warmly thanks Christophe Besse, Vincent Colin, Paolo Ghiggini, 
Fran\-\c{c}ois Laudenbach, Laurent Meersseman and Pascal Noble for very helpful and stimulating discussions 
related to this work.

\bibliographystyle{alpha}
\bibliography{Transparent}

\newcommand{\etalchar}[1]{$^{#1}$}
\def\cprime{$'$}
\begin{thebibliography}{BMGN16}

\bibitem[AAB{\etalchar{+}}08]{aabes}
X.~Antoine, A.~Arnold, C.~Besse, M.~Ehrhardt, and A.~Sch{\"a}dle.
\newblock A review of transparent and artificial boundary conditions techniques
  for linear and nonlinear {S}chr\"odinger equations.
\newblock {\em Commun. Comput. Phys.}, 4(4):729--796, 2008.

\bibitem[AB01]{ab}
X.~Antoine and C.~Besse.
\newblock Construction, structure and asymptotic approximations of a
  microdifferential transparent boundary conditions for the linear
  {S}chr\"odinger equation.
\newblock {\em J. Math. Pures Appl.}, 80(7):701--738, 2001.

\bibitem[ABS09]{abs}
X.~Antoine, C.~Besse, and J.~Szeftel.
\newblock Towards accurate artificial boundary conditions for nonlinear
  {P}{D}{E}s through examples.
\newblock {\em Cubo}, 11(4):29--48, 2009.

\bibitem[AES03]{AES}
A.~Arnold, M.~Ehrhardt, and I.~Sofronov.
\newblock Discrete transparent boundary conditions for the {S}chr\"odinger
  equation: fast calculation, approximation, and stability.
\newblock {\em Commun. Math. Sci.}, 1(3):501--556, 2003.

\bibitem[Aud12]{audiard}
C.~Audiard.
\newblock Non-homogeneous boundary value problems for linear dispersive
  equations.
\newblock {\em Comm. Partial Differential Equations}, 37(1):1--37, 2012.

\bibitem[Bau85]{baumgartel}
H.~Baumg{\"a}rtel.
\newblock {\em Analytic perturbation theory for matrices and operators}.
\newblock Birkh\"auser Verlag, 1985.

\bibitem[BELV16]{BELV}
C.~Besse, M.~Ehrhardt, and I.~Lacroix-Violet.
\newblock Discrete artificial boundary conditions for the {K}orteweg-de-{V}ries
  equation.
\newblock {\em Numer. Methods Partial Differential Equations}, 2016.

\bibitem[BGS07]{benzoni-serre}
S.~Benzoni-Gavage and D.~Serre.
\newblock {\em Multidimensional hyperbolic partial differential equations}.
\newblock Oxford University Press, 2007.
\newblock First-order systems and applications.

\bibitem[BMGN16]{BMN}
C.~Besse, B.~M\'esognon-Gireau, and P.~Noble.
\newblock Artificial boundary conditions for the linearized
  {B}enjamin-{B}ona-{M}ahony equation.
\newblock Available at {\tt https://hal.archives-ouvertes.fr/hal-01305360},
  2016.

\bibitem[CG11]{jfcag}
J.-F. Coulombel and A.~Gloria.
\newblock Semigroup stability of finite difference schemes for multidimensional
  hyperbolic initial boundary value problems.
\newblock {\em Math. Comp.}, 80(273):165--203, 2011.

\bibitem[Cou09]{jfcsinum}
J.-F. Coulombel.
\newblock Stability of finite difference schemes for hyperbolic initial
  boundary value problems.
\newblock {\em SIAM J. Numer. Anal.}, 47(4):2844--2871, 2009.

\bibitem[Cou11]{jfcpise}
J.-F. Coulombel.
\newblock Stability of finite difference schemes for hyperbolic initial
  boundary value problems {I}{I}.
\newblock {\em Ann. Sc. Norm. Super. Pisa Cl. Sci. (5)}, X(1):37--98, 2011.

\bibitem[Cou13]{jfcnotes}
J.-F. Coulombel.
\newblock Stability of finite difference schemes for hyperbolic initial
  boundary value problems.
\newblock In {\em HCDTE Lecture Notes. Part I. Nonlinear Hyperbolic PDEs,
  Dispersive and Transport Equations}, pages 97--225. American Institute of
  Mathematical Sciences, 2013.

\bibitem[Cou15a]{jfc}
J.-F. Coulombel.
\newblock Fully discrete hyperbolic initial boundary value problems with
  nonzero initial data.
\newblock {\em Confluentes Math.}, 7(2):17--47, 2015.

\bibitem[Cou15b]{jfcX}
J.-F. Coulombel.
\newblock The {L}eray-{G}\aa rding method for finite difference schemes.
\newblock {\em J. \'Ec. polytech. Math.}, 2:297--331, 2015.

\bibitem[DZ06]{ducomet-zlotnik}
B.~Ducomet and A.~Zlotnik.
\newblock On stability of the {C}rank-{N}icolson scheme with approximate
  transparent boundary conditions for the {S}chr\"odinger equation. {I}.
\newblock {\em Commun. Math. Sci.}, 4(4):741--766, 2006.

\bibitem[EA01]{ehrhardt-arnold}
M.~Ehrhardt and A.~Arnold.
\newblock Discrete transparent boundary conditions for the {S}chr\"odinger
  equation.
\newblock {\em Riv. Mat. Univ. Parma (6)}, 4*:57--108, 2001.
\newblock Fluid dynamic processes with inelastic interactions at the molecular
  scale (Torino, 2000).

\bibitem[Ehr10]{ehrhardt}
M.~Ehrhardt.
\newblock Absorbing boundary conditions for hyperbolic systems.
\newblock {\em Numer. Math. Theory Methods Appl.}, 3(3):295--337, 2010.

\bibitem[Emm09a]{Emmrich1}
E.~Emmrich.
\newblock Convergence of the variable two-step {BDF} time discretisation of
  nonlinear evolution problems governed by a monotone potential operator.
\newblock {\em BIT}, 49(2):297--323, 2009.

\bibitem[Emm09b]{Emmrich2}
E.~Emmrich.
\newblock Two-step {BDF} time discretisation of nonlinear evolution problems
  governed by monotone operators with strongly continuous perturbations.
\newblock {\em Comput. Methods Appl. Math.}, 9(1):37--62, 2009.

\bibitem[GF74]{GF}
I.~C. Gohberg and I.~A. Fel{\cprime}dman.
\newblock {\em Convolution equations and projection methods for their
  solution}.
\newblock American Mathematical Society, 1974.
\newblock Translated from the Russian, Translations of Mathematical Monographs,
  Vol. 41.

\bibitem[GKO95]{gko}
B.~Gustafsson, H.-O. Kreiss, and J.~Oliger.
\newblock {\em Time dependent problems and difference methods}.
\newblock John Wiley \& Sons, 1995.

\bibitem[GKS72]{gks}
B.~Gustafsson, H.-O. Kreiss, and A.~Sundstr{\"o}m.
\newblock Stability theory of difference approximations for mixed initial
  boundary value problems. {II}.
\newblock {\em Math. Comp.}, 26(119):649--686, 1972.

\bibitem[Gol77]{goldberg}
M.~Goldberg.
\newblock On a boundary extrapolation theorem by {K}reiss.
\newblock {\em Math. Comp.}, 31(138):469--477, 1977.

\bibitem[GT81]{goldberg-tadmor}
M.~Goldberg and E.~Tadmor.
\newblock Scheme-independent stability criteria for difference approximations
  of hyperbolic initial-boundary value problems. {II}.
\newblock {\em Math. Comp.}, 36(154):603--626, 1981.

\bibitem[Hag99]{hagstrom}
T.~Hagstrom.
\newblock Radiation boundary conditions for the numerical simulation of waves.
\newblock In {\em Acta numerica, 1999}, volume~8 of {\em Acta Numer.}, pages
  47--106. Cambridge Univ. Press, 1999.

\bibitem[Hal82]{halpern}
L.~Halpern.
\newblock Absorbing boundary conditions for the discretization schemes of the
  one-dimensional wave equation.
\newblock {\em Math. Comp.}, 38(158):415--429, 1982.

\bibitem[HNW93]{hnw}
E.~Hairer, S.~P. N{\o}rsett, and G.~Wanner.
\newblock {\em Solving ordinary differential equations. {I}}.
\newblock Springer-Verlag, second edition, 1993.
\newblock Nonstiff problems.

\bibitem[HW96]{hw}
E.~Hairer and G.~Wanner.
\newblock {\em Solving ordinary differential equations. {II}}.
\newblock Springer-Verlag, second edition, 1996.
\newblock Stiff and differential-algebraic problems.

\bibitem[HY07]{han-yin}
H.~Han and D.~Yin.
\newblock Absorbing boundary conditions for the multidimensional
  {K}lein-{G}ordon equation.
\newblock {\em Commun. Math. Sci.}, 5(3):743--764, 2007.

\bibitem[Kat95]{kato}
T.~Kato.
\newblock {\em Perturbation theory for linear operators}.
\newblock Classics in Mathematics. Springer-Verlag, 1995.

\bibitem[Kre68]{kreiss1}
H.-O. Kreiss.
\newblock Stability theory for difference approximations of mixed initial
  boundary value problems. {I}.
\newblock {\em Math. Comp.}, 22:703--714, 1968.

\bibitem[Kre70]{kreiss2}
H.-O. Kreiss.
\newblock Initial boundary value problems for hyperbolic systems.
\newblock {\em Comm. Pure Appl. Math.}, 23:277--298, 1970.

\bibitem[Lax02]{lax}
P.~D. Lax.
\newblock {\em Functional analysis}.
\newblock Pure and Applied Mathematics. Wiley-Interscience, 2002.

\bibitem[Nik02]{Nikolski}
N.~K. Nikolski.
\newblock {\em Operators, functions, and systems: an easy reading. {V}ol. 1}.
\newblock Mathematical Surveys and Monographs. American Mathematical Society,
  2002.

\bibitem[Osh69]{osher1}
S.~Osher.
\newblock Systems of difference equations with general homogeneous boundary
  conditions.
\newblock {\em Trans. Amer. Math. Soc.}, 137:177--201, 1969.

\bibitem[Osh72]{osher3}
S.~Osher.
\newblock Stability of parabolic difference approximations to certain mixed
  initial boundary value problems.
\newblock {\em Math. Comp.}, 26:13--39, 1972.

\bibitem[Qin83]{Qin}
M.~Z. Qin.
\newblock Difference schemes for the dispersive equation.
\newblock {\em Computing}, 31(3):261--267, 1983.

\bibitem[RM67]{RM}
R.~D. Richtmyer and K.~W. Morton.
\newblock {\em Difference methods for initial value problems}.
\newblock Graduate Texts in Mathematics. Interscience Publishers John Wiley \&
  Sons, 1967.
\newblock Theory and applications.

\bibitem[Rud87]{rudin}
W.~Rudin.
\newblock {\em Real and complex analysis}.
\newblock McGraw-Hill, 1987.

\bibitem[Sar65]{sarason1}
L.~Sarason.
\newblock On hyperbolic mixed problems.
\newblock {\em Arch. Rational Mech. Anal.}, 18:310--334, 1965.

\bibitem[Str62]{strang1}
G.~Strang.
\newblock Trigonometric polynomials and difference methods of maximum accuracy.
\newblock {\em J. Math. Phys.}, 41:147--154, 1962.

\bibitem[Str64]{strang2}
G.~Strang.
\newblock Wiener-{H}opf difference equations.
\newblock {\em J. Math. Mech.}, 13:85--96, 1964.

\bibitem[SW97]{strikwerda-wade}
J.~C. Strikwerda and B.~A. Wade.
\newblock A survey of the {K}reiss matrix theorem for power bounded families of
  matrices and its extensions.
\newblock In {\em Linear operators ({W}arsaw, 1994)}, volume~38 of {\em Banach
  Center Publ.}, pages 339--360. Polish Acad. Sci., 1997.

\bibitem[Sze04]{szeftel1}
J.~Szeftel.
\newblock Design of absorbing boundary conditions for {S}chr\"odinger equations
  in {$\Bbb R^d$}.
\newblock {\em SIAM J. Numer. Anal.}, 42(4):1527--1551, 2004.

\bibitem[Sze06]{szeftel}
J.~Szeftel.
\newblock Absorbing boundary conditions for the nonlinear {S}chr\"odinger
  equation.
\newblock {\em Numer. Math.}, 103:103--127, 2006.

\bibitem[Tre84]{trefethen3}
L.~N. Trefethen.
\newblock Instability of difference models for hyperbolic initial boundary
  value problems.
\newblock {\em Comm. Pure Appl. Math.}, 37:329--367, 1984.

\bibitem[VB82]{livreVB}
R.~Vichnevetsky and J.~B. Bowles.
\newblock {\em Fourier analysis of numerical approximations of hyperbolic
  equations}, volume~5 of {\em SIAM Studies in Applied Mathematics}.
\newblock Society for Industrial and Applied Mathematics, 1982.

\bibitem[ZE06]{zisowski-ehrhardt}
A.~Zisowsky and M.~Ehrhardt.
\newblock Discrete transparent boundary conditions for parabolic systems.
\newblock {\em Math. Comput. Modelling}, 43(3-4):294--309, 2006.

\bibitem[ZWH08]{zheng-wen-han}
C.~Zheng, X.~Wen, and H.~Han.
\newblock Numerical solution to a linearized {K}d{V} equation on unbounded
  domain.
\newblock {\em Numer. Methods Partial Differential Equations}, 24(2):383--399,
  2008.

\end{thebibliography}
\end{document}